\documentclass[a4paper,11pt]{amsart}
\usepackage[cm]{fullpage}
\usepackage{hyperref}

\usepackage{mymacros}

\title{Lifts of cycles in tropical hypersurfaces and the Gamma conjecture}
\author{Yuto Yamamoto}
\address{
RIKEN iTHEMS, Wako, Saitama 351-0198, Japan
}
\email{yuto.yamamoto@riken.jp}
\date{}
\begin{document}

\begin{abstract}
For a complex hypersurface of dimension $d \geq 1$ in a toric variety, we construct lifts of tropical $(p, q)$-cycles with $p+q=d$ in the associated tropical hypersurface. 
The tropical cycles we consider are described by Minkowski weights, and their lifts are realized as topological cycles admitting a torus fibration structure over the tropical cycles. 
The intersection numbers of these lifted cycles are computed in terms of tropical intersection theory.
We further derive the asymptotic formulas for the period integrals of the lifts in the tropical limit, which are closely related to the mirror symmetric Gamma conjecture.
Throughout the paper, we assume that the tropicalization is dual to a unimodular triangulation of the Newton polytope.
\end{abstract}

\maketitle

\section{Introduction}\label{sc:intro}

The tropical (co)homology groups of a tropical space are defined as the (co)homology groups of certain constructive (co)sheaves on that space.
When the tropical space is smooth, the tropical (co)homology groups correspond to the grade pieces of the limiting mixed Hodge structure of the corresponding degenerating family of complex projective varieties, and in particular, the Hodge numbers of the complex projective varieties coincide with the dimensions of the tropical (co)homology groups \cite{MR2669728, MR2681794, MR3961331}.
It is expected that further geometric information of complex projective varieties, such as intersection pairings and period integrals, can also be characterized in terms of tropical spaces.
For instance, tropical homology groups have (tropical) intersection pairings (\cite{MR3330789, MR4347312}), and it has been known that the tropical (sheaf) homology group (of degree $(1, 1)$) of a tropical K3 surface coincides, as lattices, with the orthogonal complement in the K3 lattice of the hyperbolic lattice spanned by a fiber and a section of the elliptic fibration (\cite{MR2024634, MR4179831, MR4347312}).

One approach in research of this direction is to construct a direct correspondence between tropical cycles and ordinary topological cycles.
A \emph{tropical cycle} is a cycle that represents a class of a tropical homology group, and it has a bidegree $(p, q)$ $(p, q \in \bZ_{\geq 0})$.
It is expected that ordinary topological cycles and tropical $(p, q)$-cycles correspond in such a way that an ordinary topological cycle admits a torus fibration structure over a tropical cycle.
The integers $p, q$ are the dimension of its fiber and the dimension of the tropical cycle respectively.
We call such a topological cycle a \emph{lift} of the tropical cycle.
There are many previous researches that construct lifts of tropical cycles such as \cite{MR2079993, MR3228462, MR3993277, MR4194298, MR4179831, MR4179650, MR4125753, MR4294119, MR4284602, MR4782805, MR4922780}.
The lifts were used to compute period integrals in \cite{MR4194298, MR4179831, MR4782805, MR4922780}, and it was also shown in \cite[Theorem 7]{MR4347312} that the intersection numbers of lifts coincide with those of tropical cycles in the case where the tropical space is an integral affine manifold with singularities.
Some of these works construct Lagrangian lifts to study the correspondence between tropical cycles and Lagrangian submanifolds.
There are also more recent works \cite{MR4904614, CLL25} constructing special Lagrangian lifts.

Tropical cycles for which lifts were constructed so far in previous researches are all of degree $(0, d)$, $(d, 0)$, $(1, d-1)$, or $(d-1, 1)$, where $d$ is the dimension of the variety\footnote{Tropical $2$-cycles (in $3$-dimensional integral affine manifolds with singularities) of \cite[Definition 7.2]{MR3228462} can be regarded as tropical $(1, 2)$-cycles (cf.~\cite[Example 6.11]{Yam21}).
Tropical curves and tropical hypersurfaces are also regarded as tropical cycles of degree $(d-1, 1)$ and $(1, d-1)$ respectively.}.
In this article, we focus on the case of hypersurfaces in toric varieties, and construct lifts for a certain class of tropical cycles of degrees $(p, q)$ with $p+q=d$. 
The degrees are not limited to those mentioned above.
An application to computation of intersection numbers of cycles will also be provided.
We also give the formulas of the asymptotics of the period integrals of the lifts in the tropical limit, which closely relate to the mirror symmetric Gamma conjecture.
In the case of Calabi--Yau manifolds, it claims as follows:

\begin{conjecture}{\rm(\cite[Conjecture A]{MR4194298}, \cite[Section 4]{Iri23})}\label{cj:gamma}
Let $X$ be a Calabi--Yau manifold equipped with a symplectic form $\omega$ and let 
$\lc Z_t \rc_{t\in \Delta^\ast}$ be a family of Calabi--Yau manifolds parametrized by $t$ in a small punctured disc $\Delta^\ast$ that corresponds to $(X,\omega)$ under mirror symmetry. 
For a suitable choice of a holomorphic volume form $\Omega_t$ on $Z_t$ and of a coordinate $t$, if a Lagrangian cycle $C_t\subset Z_t$ is mirror to a coherent sheaf $E$ on $X$, then 
\begin{align}
\int_{C_t} \Omega_t = \int_X t^{-\omega} \cdot \widehat{\Gamma}_X 
\cdot (2\pi \sqrt{-1})^{\deg/2} \ch(E) + O\lb t^\epsilon \rb
\end{align}
as $t\to 0$ in a fixed angular sector, where $\epsilon>0$ is some constant.
\end{conjecture}

We state the main theorems of this article in the following subsections.

\subsection{Tropical cycles from Minkowski weights}\label{sc:main1}

Throughout this article, we work in the same setup as \cite{MR4782805}.
Let $K$ be the convergent Puiseux series field over $\bC$, i.e., the field of formal series $\sum_{j \in \bZ}c_j x^{j/n}$ $\lb c_j \in \bC, n \in \bZ_{\geq 1} \rb$ that have only finitely many coefficients with negative index and whose positive part is convergent in a neighborhood of $x=0$. 
It has the standard non-archimedean valuation 
\begin{align}\label{eq:val}
	\mathrm{val} \colon K \longrightarrow \bQ \cup \{ \infty \},\quad k=\sum_{j \in \bZ}c_j x^{j/n} \mapsto \min \lc j/n \in \bQ \relmid c_j \ne 0 \rc.
\end{align}
Let $d$ be a positive integer.
Consider a free $\bZ$-module $N$ of rank $d+1$, and its dual $M:=\Hom (N, \bZ)$.
We write $N_Q:=N \otimes_\bZ Q$, $M_Q:=M \otimes_\bZ Q$ for a $\bZ$-module $Q$.
Let $\Delta \subset M_\bR$ be a lattice polytope of dimension $d+1$ such that $W:=\rint (\Delta) \cap M \neq \emptyset$, where $\rint (\Delta)$ denotes the relative interior of $\Delta$.
We consider a Laurent polynomial $f=\sum_{m \in A} k_m z^m \in K \ld M \rd$ $(k_m \neq 0, \forall m \in A:= \Delta \cap M)$ over $K$ such that the function 
\begin{align}\label{eq:lambda}
\lambda \colon A \to \bQ, \quad m \mapsto \mathrm{val}(k_m)
\end{align}
extends to a strictly-convex\footnote{In this article, we say that a function $\lambda \colon \Delta \to \bR$ is \emph{convex} if it satisfies $\lambda \lb t m_1+(1-t)m_2 \rb \leq t \lambda(m_1)+(1-t)\lambda(m_2)$ for any $t \in \ld 0, 1\rd$ and $m_1, m_2 \in \Delta$.} piecewise affine function on a unimodular triangulation $\scrT$ of $\Delta$, i.e., a triangulation consisting only of $(d+1)$-dimensional simplices of the minimal volume $1/\lb d+1\rb!$ and their faces.

The tropicalization of the Laurent polynomial $f$ is the piecewise affine function $\trop \lb f \rb \colon N_\bR \to \bR$ defined by
\begin{align}
\trop(f)(n):=\min_{m \in A} \lc \mathrm{val} (k_m) + \la m, n \ra \rc,
\end{align}
and the tropical hypersurface $X\lb \trop \lb f \rb \rb \subset N_\bR$ defined by $\trop \lb f \rb$ is the corner locus of $\trop \lb f \rb$.
The tropical hypersurface $X\lb \trop \lb f \rb \rb \subset N_\bR$ induces a polyhedral decomposition $\scrP$ of $N_\bR$ that is dual to the triangulation $\scrT$ (\cite[Proposition 2.1]{MR2079993}).
The correspondence is given by
\begin{align}\label{eq:dual}
\scrT \to \scrP, \quad 
\tau \mapsto 
\lc n \in N_\bR \relmid \trop(f)(n)
=
\mathrm{val} (k_m) + \la m, n \ra, \forall m \in \tau \cap M \rc.
\end{align}
For $w \in W$, we write the element in $\scrP$ dual to $\lc w \rc \in \scrT$ as $\nabla_w$, i.e.,
\begin{align}\label{eq:nabla}
\nabla_w=\lc n \in N_\bR \relmid 
\trop(f)(n)
=
\mathrm{val} (k_w) + \la w, n \ra \rc.
\end{align}
The normal fan of $\nabla_w \subset N_\bR$ coincides with
\begin{align}\label{eq:fan_w}
\Sigma_{w}:=\lc \bR_{\geq 0} \cdot \lb \tau-w \rb \relmid \tau \in \scrT, \tau \ni w \rc.
\end{align}
Since the triangulation $\scrT$ is unimodular, the fan $\Sigma_{w}$ is also unimodular.

For each $\sigma= \bR_{\geq 0} \cdot \lb \tau-w \rb \in \Sigma_{w}$ such that $\dim \sigma \geq 1$, we choose an arbitrary point $b_\sigma$ of the relative interior of the face of $\nabla_w$ corresponding to $\tau$ by \eqref{eq:dual}.
For $q \in \lc 0, \cdots, d\rc$, we set
\begin{align}\label{eq:S_q}
    \scrS_q^w:=\lc \lc \sigma_1 \prec \cdots \prec \sigma_{q+1}\rc \relmid \sigma_i \in \Sigma_w(i), \forall i \in \lc 1, \cdots, q+1\rc \rc,
\end{align}
where $\Sigma_w (i)$ $(i \in \lc 0, \cdots, d+1\rc)$ denotes the set of cones of dimension $i$ in $\Sigma_w$.
For a sequence of cones $\scS=\lc \sigma_1 \prec \cdots \prec \sigma_{q+1} \rc \in \scrS_q^w$, we consider the unique affine map (a singular $q$-simplex) from the standard $q$-simplex $\Delta^q$ to the tropical hypersurface $X\lb \trop \lb f \rb \rb$ which sends the $i$-th vertex of $\Delta^q$ to $b_{\sigma_i}$, the point chosen for the cone $\sigma_i \in \Sigma_w$.
We write it as $\Delta_\scS \colon \Delta^q \to X\lb \trop \lb f \rb \rb$.
By abuse of notation, the image of the singular $q$-simplex $\Delta_\scS$, which equals the convex hull of the set of points $\lc b_{\sigma_i} \relmid 1 \leq i \leq q+1 \rc$ will also be denoted by $\Delta_\scS \subset X\lb \trop \lb f \rb \rb$ in the following.
Let $e_i \in M$ $\lb 1 \leq i \leq q+1\rb$ be the elements such that 
$\lc e_i \relmid 1 \leq i \leq j \rc$ are the primitive generators of 
$\sigma_j$ for all $1 \leq j \leq q+1$.
We also fix a generator $\vol(N)$ of $\bigwedge^{d+1} N$.
The map
\begin{align}\label{eq:fd-q}
\la \bigwedge_{i=1}^{q+1} e_i \wedge \bullet, \vol(N) \ra  \colon \bigwedge^{d-q} M \to  \bZ
\end{align}
defines an element of $\bigwedge^{d-q} N$, which we will write as $f(\scS) \in \bigwedge^{d-q} N$.

Let $\MW^{d-q} \lb \Sigma_w \rb$ be the group of Minkowski weights of codimension $d-q$ on $\Sigma_w$ of \cite{MR1415592}.
Each element (Minkowski weight) $a \in \MW^{d-q} \lb \Sigma_w \rb$ is a function 
$a \colon \Sigma_w(q+1) \to \bZ$ satisfying the balancing condition.
(See \eqref{eq:minkowski} for the definition.)
Let $Y_w$ denote the toric variety associated with the fan $\Sigma_w$.
It is proved in \cite[Theorem 3.1]{MR1415592} that the group of Minkowski weights $\MW^{d-q} \lb \Sigma_w \rb$ is canonically isomorphic to the Chow cohomology group $A^{d-p}(Y_w)$ of the toric variety $Y_w$ (cf.~\eqref{eq:amw}), and the cup product of $A^\ast(Y_w)$ can be written using Minkowski weights by the so-called fan displacement rule (cf.~\pref{th:MW}).
For a Minkowski weight $a \in \MW^{d-q} \lb \Sigma_w \rb$, we consider the formal sum
\begin{align}\label{eq:trop-chain}
c(a):=\sum_{\scS \in \scrS_q^w} a({\sigma[\scS]}) \cdot f (\scS) \cdot \Delta_\scS,
\end{align}
where $\sigma[\scS]$ denotes the last cone $\sigma_{q+1} \in \Sigma_w(q+1)$ in the sequence $\scS=\lc \sigma_1 \prec \cdots \prec \sigma_{q+1}\rc$.

\begin{theorem}\label{th:main1}
The following statements hold:
\begin{enumerate}
\item The formal sum $c(a)$ of \eqref{eq:trop-chain} associated with $a \in \MW^{d-q} \lb \Sigma_w \rb$ is a tropical singular $(d-q, q)$-cycle in the tropical hypersurface $X\lb \trop \lb f \rb \rb$, and defines a class of the tropical $(d-q,q)$-homology group $H_q \lb X\lb \trop \lb f \rb \rb, \scF_{d-q}^\bZ \rb$.
\item Let $w_1, w_2 \in \rint(\Delta)$, 
and $a_1 \in \MW^{d-q} \lb \Sigma_{w_1} \rb, a_2 \in \MW^{q} \lb \Sigma_{w_2} \rb$.
The tropical intersection number $\la c(a_1), c(a_2)\ra$ of the tropical cycles $c(a_1), c(a_2)$, which is defined after moving the tropical cycles so that they intersect transversely can be computed combinatorially as follows:
\begin{enumerate}
\item When $w_1=w_2$, we have
\begin{align}\label{eq:w1=w2-intersection}
\la c(a_1), c(a_2) \ra=(-1)^d \cdot \psi(a_1 \cup a_2),
\end{align}
where $a_1 \cup a_2 \in \MW^d\lb \Sigma_{w_1=w_2} \rb$ is the cup product of $a_1$ and $a_2$, and $\psi$ is the map defined by
\begin{align}
\psi \colon \MW^d\lb \Sigma_{w_1=w_2} \rb \to \bZ, \quad a \mapsto \sum_{\sigma \in \Sigma_{w_1=w_2}(1)} a(\sigma).
\end{align}
\item When $w_1 \neq w_2$ and the convex hull $\conv (\lc w_1, w_2 \rc)$ of $\lc w_1, w_2 \rc$ is in $\scrT$, 
let $\Sigma_{w_1, w_2}$ be the fan in $\left. M_\bR \middle/ \bR \cdot \lb w_1-w_2 \rb\right.$, which consists of the images of $\bR_{\geq 0} \cdot (\tau-w_1) \in \Sigma_{w_1}$ (or $\bR_{\geq 0} \cdot (\tau-w_2) \in \Sigma_{w_2}$\footnote{Notice that the images of $\bR_{\geq 0} \cdot (\tau-w_1)$ and $\bR_{\geq 0} \cdot (\tau-w_2)$ by the projection to $\left. M_\bR \middle/ \bR \cdot \lb w_1-w_2 \rb\right.$ are equal.}) with 
$\tau \succ \conv (\lc w_1, w_2 \rc)$ by the projection to $\left. M_\bR \middle/ \bR  \cdot \lb w_1-w_2 \rb\right.$.
We have
\begin{align}
\la c(a_1), c(a_2) \ra
=(-1)^{d+1} \cdot \overline{a}_1 \cup \overline{a}_2 \in \MW^d\lb \Sigma_{w_1,w_2} \rb \cong \bZ,
\end{align}
where $\overline{a}_1 \in \MW^{d-q} \lb \Sigma_{w_1, w_2} \rb, \overline{a}_2 \in \MW^{q} \lb \Sigma_{w_1, w_2} \rb$ are the Minkowski weights that send the cone which is the image of $\bR_{\geq 0} \cdot (\tau-w_i)$ by the projection to $a_i \lb \bR_{\geq 0} \cdot (\tau-w_i) \rb \in \bZ$ respectively $(i=1, 2)$.
\item When $\conv(\lc w_1, w_2 \rc) \nin \scrT$, we have
\begin{align}
\la c(a_1), c(a_2) \ra=0.
\end{align}
\end{enumerate}
\end{enumerate}
\end{theorem}

For the definitions of tropical cycles, their intersection numbers, and tropical homology groups, see \pref{sc:tropical-homology}.
These were introduced in \cite{MR3330789, MR3961331}.
These notions are also formulated for integral affine manifolds with singularities, another type of tropical spaces (cf.~\cite{MR2669728, MR4347312, Yam21}).

In the proof of \pref{th:main1}, we will see that the balancing condition of the Minkowski weight $a$ ensures that \eqref{eq:trop-chain} becomes a cycle.
We will also see in \pref{th:main1}(2) that the fan displacement that we take to compute the cup product of $A^\ast(Y_w)$ in terms of Minkowski weights corresponds to moving tropical cycles so that they intersect transversely.

In a certain special case (the case of the tropical $(1, 1)$-homology groups of floor decomposed hypersurfaces of dimension $2$ in tropical projective spaces), the intersection pairings of tropical homology groups are explicitly computed in \cite{MR3088918}.
This computation is carried out by induction on the degree of the defining polynomial, together with the Mayer--Vietoris exact sequence.
The assumption that a floor decomposition exists is crucial in the computation.
The intersection pairing of a tropical K3 surface that is an integral affine manifold with singularities is also computed, by taking a basis of the tropical (sheaf) homology group explicitly \cite{MR2024634, MR4179831, MR4347312}.
The cup product of (the ambient part of) the tropical cohomology of a tropical Calabi--Yau hypersurface is also computed in \cite{MR4484542} using \v{C}ech cohomology with respect to acyclic coverings.

While tropical cycles considered in \pref{th:main1} are topological cycles, 
there is also a notion of tropical algebraic or analytic cycles, which was defined in \cite{MR2275625}.
In \cite{MR3330789}, tropical algebraic/analytic cycles are called \emph{straight tropical cycles} to distinguish them from tropical topological cycles\footnote{A straight tropical cycle of dimension $q$ naturally defines a tropical topological $(q, q)$-cycles (\cite[Proposition 4.3]{MR3330789}).}.
The intersection theory for straight tropical cycles has been studied in many papers such as \cite{MR2591823, MR2887109, MR3032930}.

\subsection{Lifts of tropical cycles}\label{sc:main2}

Let $t \in \bR_{>0}$ be an element in a small neighborhood of $0$, where all coefficients $k_m$ of the Laurent polynomial $f$ converge.
We consider the Laurent polynomial $f_t \in \bC\ld M \rd$ over $\bC$ obtained by substituting $t$ for the indeterminate $x$ in the coefficients of $f$.
We define
\begin{align}
\mathring{Z}_t:=\lc z \in N_{\bC^\ast} \relmid f_t(z)=0 \rc.
\end{align}
Let further $\Sigma$ be a rational simplicial fan in $N_\bR$ that is a refinement of the normal fan of $\Delta$, and $Y_\Sigma \supset N_{\bC^\ast}$ be the associated toric variety over $\bC$.
We write the closure of the hypersurface $\mathring{Z}_t$ in $Y_\Sigma$ as $Z_t$.

Let $l \geq 1$ be an integer.
We consider the polytope $l \cdot \Delta:=\lc l \cdot m \in M_\bR \relmid m \in \Delta \rc$, and its triangulation $l \cdot \scrT:=\lc l \cdot \tau \relmid \tau \in \scrT \rc$.
We set $V_l:=M \cap \rint \lb l \cdot \Delta \rb$, and take an element $v \in V_l$.
Let $\tau_v \in \scrT$ be the minimal cell such that $v \in l \cdot \tau_v$.
By the unimodularity assumption on $\scrT$, one can uniquely write $v= \sum_{m \in A \cap \tau_v} p_m \cdot m$ with $p_m \in \bZ \cap \lb 0, l \rd$ such that $\sum_{m \in A \cap \tau_v} p_m =l$.
We consider
\begin{align}\label{eq:omega}
\omega_t^{l, v}:=
\lb \bigwedge_{i=0}^{d} \frac{dz_i}{z_i} \rb
\frac{1}{\lb f_t \rb^l}
\prod_{m \in A \cap \tau_v} \lb k_{m, t} \cdot z^m \rb^{p_m}, 
\end{align}
where $\lb z_0, \cdots, z_{d} \rb$ are $\bC^\ast$-coordinates on $N_{\bC^\ast} \cong \lb \bC^\ast \rb^{d+1}$, and $k_{m, t} \in \bC^\ast$ is the number obtained by substituting $t$ to the indeterminate $x$ in $k_m$.
\eqref{eq:omega} extends to a meromorphic $(d+1)$-form on $Y_\Sigma$ that has a pole along $Z_t$, and such forms generate $H^{0} \lb Y_\Sigma, \Omega^{d+1} \lb l \cdot Z_t \rb \rb$ (cf.~\cite{MR1290195} or \cite[Section 2]{MR4782805}).
We write the image of $\omega_t^{l, v}$ by the Poincar\'{e} residue map 
\begin{align}\label{eq:Res}
\Res \colon H^{0} \lb Y_\Sigma, \Omega^{d+1} \lb l \cdot Z_t \rb \rb \to H^d \lb Z_t, \bC \rb
\end{align}
(cf.~e.g.~\cite[Section 8]{MR0260733}) as $\Omega_t^{l, v} \in H^d \lb Z_t, \bC \rb$.
We consider the period integrals of $\Omega_t^{l, v}$.
Notice that the cohomology group of a hypersurface decomposes into the residual part (the image of the Poincar\'{e} residue map \eqref{eq:Res}) and the ambient part which consists of cohomology classes coming from the ambient toric variety, and therefore, the former part is essential when it comes to Hodge structures of hypersurfaces (cf.~\cite[Section 5]{MR2019444}).
Notice that the periods of the residues of forms with a higher-order pole are also necessary to describe the Hodge structure of the hypersurface $Z_t$.

For $w \in W$, we set 
\begin{align}\label{eq:Aw}
A_{w}:=\lc m \in A \setminus \lc w \rc \relmid \conv \lb \lc m, w \rc \rb \in \scrT \rc.
\end{align}
Let $D_m^w$ $\lb m \in A_{w} \rb$ denote the toric divisor on the toric variety $Y_{w}$ associated with the $1$-dimensional cone $\bR_{\geq 0} \cdot \lb m-w \rb \in \Sigma_{w}$.
We also set $\lambda_m:=\val(k_m) \in \bQ$ for $m \in A$, and 
\begin{align}
\omega_{\lambda}^{w}&:=\sum_{m \in A_{w}} \lb \lambda_{m}-\lambda_{w} \rb D_{m}^w\\
\quad \sigma^w&:=\sum_{m \in A_{w}} D_{m}^w\\
E_{v, w}&:=
\lb \prod_{m \in A_w \cap \tau_v} \prod_{i=0}^{p_m-1} \lb D_m^w +i \rb \rb
\prod_{i=0}^{p_w-1} \lb \sigma^w -i \rb \in H^{\ast} \lb Y_w, \bR \rb,
\end{align}
where $p_m$ are the integers that we used to write $v= \sum_{m \in A \cap \tau_v} p_m \cdot m$, and if $w \nin \tau_v$, we set $p_w:=0$ and the last product in $E_{v, w}$ means the empty product.
We also define
\begin{align}
\widehat{\Gamma}_{w}&:=\frac{\prod_{m \in A_{w}} \Gamma \lb 1+D_m^w \rb}{\Gamma (1+\sigma^w)} \in H^{\ast} \lb Y_w, \bR \rb
\end{align}
using the power series expansion of the gamma function $\Gamma \lb 1+ x\rb$:
\begin{align}\label{eq:gamma-exp}
\Gamma(1+x)=\exp \lb -\gamma x+ \sum_{k=2}^\infty \frac{\lb -1 \rb^k}{k} \zeta(k) x^k \rb,
\end{align}
where $\gamma$ is the Euler constant, and $\zeta(k)$ is the Riemann zeta value.
Note that the restriction of the class $\widehat{\Gamma}_{w} \in H^{\ast} \lb Y_w, \bR \rb$ to an anticanonical hypersurface of the toric variety $Y_{w}$ is the gamma class of the hypersurface.
See \cite[Section 4.1]{MR4194298}.
One can also find a more explicit expression of $\widehat{\Gamma}_{w}$ at (15) in loc.cit.

The tropical cycle $c(a)$ of \eqref{eq:trop-chain} is supported in
\begin{align}\label{eq:sk-nabla}
\partial \nabla_w^{q}:=\bigcup_{\scS \in \scrS_q^w} 
\Delta_\scS
\subset \partial \nabla_w.
\end{align}
For $\scS=\lc \sigma_1 \prec \cdots \prec \sigma_{q+1} \rc \in \scrS_q^w$, we define
\begin{align}
    \mathring{\Delta}_\scS:=
    \left\{ \begin{array}{ll}
\Delta_\scS & q=0 \\
\Delta_\scS \setminus \Delta_{\scS(-1)} & q \geq 1,
\end{array}
\right.
\end{align}
where $\scS(-1):=\scS \setminus \lc \sigma[\scS] \rc= \lc \sigma_1 \prec \cdots \prec \sigma_{q} \rc \in \scrS_{q-1}^w$.
We consider the map
\begin{align}
    i_t \colon N_{\bC} \to N_{\bC^\ast}
\end{align}
induced by $\bC \to \bC^\ast, c \mapsto t^c$, and the bijection
\begin{align}\label{eq:j_t}
j_t \colon N_\bR \times N_\bR \to N_\bC, 
\quad (n ,n') 
\mapsto 
n+ \frac{2 \pi \sqrt{-1}}{\log t} n'.
\end{align}
The composition $i_t \circ j_t \colon N_\bR \times N_\bR \to N_{\bC^\ast}$ induces the bijection
\begin{align}\label{eq:h_t}
h_t \colon N_\bR \times \lb \left. N_\bR \middle/ N \right. \rb 
\to N_{\bC^\ast}.
\end{align}
Let further
\begin{align}\label{eq:1-proj}
\pi_1 \colon N_\bR \times (\left. N_\bR \middle/ N \right.) \to N_\bR, \quad (n, n') \mapsto n
\end{align}
be the first projection.

Let $c_m \in \bC^\ast$ denote the coefficient of $x^{\lambda_m}$ in $k_m \in K$ $(m \in A)$.
We choose a branch of the argument $\arg \lb - c_m/c_w \rb$ for every $m \in A_w$, and define
\begin{align}
\lb
-\frac{c_m}{c_{w}}
\rb^{-D_m^w}
:=\exp \lb -D_m^w \log \lb -\frac{c_m}{c_{w}} \rb \rb
=\exp \lb -D_m^w 
\lb \log \left| -\frac{c_m}{c_{w}} \right|+\sqrt{-1} \arg \lb -\frac{c_m}{c_{w}} \rb \rb \rb.
\end{align}
Let $K(Y_w)$ be the Grothendieck group of coherent sheaves on the toric variety $Y_w$.

\begin{theorem}\label{th:main2}
For sufficiently small $t >0$, there exists a group homomorphism
\begin{align}\label{eq:Psi}
    \Psi_t^w \colon K \lb Y_w \rb \to H_d \lb \mathring{Z}_t, \bZ \rb
\end{align}
satisfying the following:
\begin{enumerate}
    \item For any $\scE \in K \lb Y_w \rb$, we have the following formulas of the asymptotics of the period integrals:
\begin{enumerate}
\item When $\conv \lb \lc w \rc \cup \tau_v \rb \in \scrT$, we have
\begin{align}\label{eq:period-formula-1}
\int_{\Psi_t^w(\scE)}  \Omega_t^{l, v} 
=
\frac{(-1)^{d+p_w}}{(l-1)!}
\int_{Y_{w}} 
t^{-\omega_{\lambda}^{w}}
\cdot 
\widehat{\Gamma}_w
\cdot 
E_{v, w}
\cdot 
\lb 2\pi \sqrt{-1}\rb^{\deg/2} \ch \lb \scE \rb
\cdot
\prod_{m \in A_{w}} 
\lb
-\frac{c_m}{c_{w}}
\rb^{-D_m^w}
+O \lb t^\epsilon \rb 
\end{align}
as $t \to +0$, where $\epsilon >0$ is some constant.
\item When $\conv \lb \lc w \rc \cup \tau_v \rb \nin \scrT$, we have
\begin{align}\label{eq:period-formula-2}
\int_{\Psi_t^w(\scE)}  \Omega_t^{l, v} 
=
O \lb t^\epsilon \rb
\end{align}
as $t \to +0$, where $\epsilon >0$ is some constant.
\end{enumerate}
\item For any $\scE \in K \lb Y_w \rb$, we set 
\begin{align}\label{eq:k}
    k:=\min \lc \lc d \rc \cup \lc k' \in \bZ_{\geq 0} \relmid \ch_{k'} \lb \scE \rb \neq 0 \rc \rc,
\end{align}
where $\ch_{k'} \lb \scE \rb \in A^{k'}(Y_w) \otimes_{\bZ} \bQ$ denotes the $k'$-th component of the Chern character $\ch(\scE)$.
One actually has $\ch_{k}(\scE) \in A^{k}(Y_w)$.
Let $a_\scE \in \MW^{k}(\Sigma_w) \cong A^{k}(Y_w)$ be the Minkowski weight corresponding to $(-1)^{d+1-k} \cdot \ch_{k}(\scE) \in A^{k}(Y_w)$.
Then there exists a $d$-cycle $C(\scE)$ in $N_\bR \times \lb \left. N_\bR \middle/ N \right.\rb$ satisfying the following:
\begin{enumerate}
    \item The image of $C(\scE)$ by the first projection $\pi_1$ \eqref{eq:1-proj} is contained in $\partial \nabla_w^{d-k}$ (\eqref{eq:sk-nabla} with $q=d-k$), in which the tropical cycle $c(a_\scE)$ is supported.
    \item For any $\scS \in \scrS_{d-k}^w$ ($=\scrS_q^w$ with $q=d-k$) and $n \in \mathring{\Delta}_\scS \subset \partial \nabla_w^{d-k}$, 
    the intersection $C(\scE) \cap \pi_1^{-1}(n)$ defines a $k$-cycle in the torus $\lc n \rc \times \left. \lb N_\bR \middle/ N \right. \rb$, and its homology class
    \begin{align}
\ld C(\scE) \cap \pi_1^{-1}(n) \rd 
\in H_{k} \lb \lc n \rc \times \lb \left. N_\bR \middle/ N \right. \rb, \bZ \rb 
\cong \bigwedge^{k} N
\end{align}
coincides with $a_\scE(\sigma[\scS]) \cdot f (\scS) \in \bigwedge^{k} N$, the coefficient of $\Delta_\scS$ in the tropical cycle $c(a_\scE)$.
\item The image by the map $h_t$ \eqref{eq:h_t} of a small perturbation of $C(\scE)$ 
whose $C^0$-norm is $O \lb(-\log t)^{-1} \rb$ is contained in the hypersurface $\mathring{Z}_t$, and represents the homology class $\Psi_t^w(\scE) \in H_d( \mathring{Z}_t, \bZ)$.
\end{enumerate}
\end{enumerate}
\end{theorem}

In this article, we write the image by the map $h_t$ of the small perturbation of $C(\scE)$ in \pref{th:main2}(2-c) as $C(\scE)_t \subset \mathring{Z}_t$, and call it a \emph{lift} of the tropical cycle $c(a_\scE)$.

\begin{theorem}\label{th:main3}
    For any none-zero Minkowski weight $a \in \MW^{d-q} \lb \Sigma_w \rb$, there exists a (non-unique) lift of the tropical cycle $c(a)$. 
\end{theorem}

As a special case, consider the following setting:
Suppose that the polytope $\Delta$ is reflexive, and let $\check{Z}$ be an anticanonical hypersurface of the toric variety $Y_{w=0}$, which is a Calabi--Yau hypersurface mirror to $Z_t$.
Let $r$ be an integer such that $1 \leq r \leq d$, and $\scD=\lc D_1, \cdots, D_r \rc$ be a set of nef toric divisors in the toric variety $Y_{w=0}$.
We write the closed subscheme defined by a general section of 
$\bigoplus_{i=1}^r \scO_{Y_{w=0}}(D_i)$ 
(resp. $\bigoplus_{i=1}^r \scO_{\check{Z}}(D_i)$)
as $X \subset Y_{w=0}$
(resp. $X' \subset \check{Z}$).
We further suppose $l=1, v=w(=0), k_w=-1, k_m=x^{\lambda_m}$ $(m \in A \setminus \lc w \rc)$.

\begin{corollary}\label{cr:CY}
In the above setting, we have 
\begin{align}\label{eq:period-formula-1'}
\int_{C(\scO_X)_t} \Omega_t^{l=1, v=0} 
=
(-1)^{d+1}
\int_{\check{Z}} 
t^{-\omega_{\lambda}^{w=0}}
\cdot 
\widehat{\Gamma}_{\check{Z}}
\cdot 
\lb 2\pi \sqrt{-1}\rb^{\deg/2} \ch \lb \scO_{X'} \rb
+O \lb t^\epsilon \rb
\end{align}
as $t \to +0$, where $\widehat{\Gamma}_{\check{Z}}$ is the gamma class of the Calabi--Yau hypersurface $\check{Z}$.
The cycle $C(\scO_X)_t$ is a lift of the tropical $(r, d-r)$-cycle $c(a_\scD)$, where $a_{\scD} \in \MW^{r}(\Sigma_{w=0})$ is the Minkowski weight corresponding to 
$(-1)^{d+1-r} \cdot \ld D_1 \rd \cdots \ld D_r \rd \in A_{d+1-r}(Y_{w=0})=A^r(Y_{w=0})$.
\end{corollary}

The formula \eqref{eq:period-formula-1'} coincides with that of the mirror symmetric Gamma conjecture (\pref{cj:gamma}) for the structure sheaf $\scO_{X'}$ of the closed subscheme $X' \subset \check{Z}$.
The discrepancy of the sign stems from the choice of the orientation of the cycle.

In \cite{MR4194298, MR4782805}, we computed the asymptotics of the periods for cycles constructed by transporting the cycle obtained as the positive real locus of the hypersurface.
These cycles are regarded as lifts of tropical cycles of degree $(0, d)$.
The homology class $\Phi_t^w(\scE)$ in \pref{th:main2} is written as a linear combination of the cycles considered in \cite{MR4194298, MR4782805} (\pref{sc:homology}), and the formulas of the asymptotics of period integrals in \pref{th:main2}(1) are derived from the results of \cite{MR4194298, MR4782805} (\pref{sc:proof-main2}).
The cycles considered in \cite{MR4194298} are expected to be isotopic to Lagrangian cycles mirror to ambient line bundles of the mirror manifold in the homological mirror symmetry (cf.~\cite[Remark 1.3]{MR4194298}).
We expect that the cycle $C(\scO_X)_t$ in \eqref{eq:period-formula-1'} is also isotopic to a Lagrangian cycle mirror to the structure sheaf $\scO_{X'}$ similarly.
The construction of the small perturbation in \pref{th:main2}(2-c) is also based on the technique in \cite[Section 5.2]{MR4194298}.

The Hodge structures of complex hypersurfaces in toric varieties have been studied in many papers (e.g.~\cite{MR0260733, MR1290195, MR1733735, MR2019444}).
The mirror symmetric Gamma conjecture originates from Hosono's conjecture \cite[Conjecture 2.2]{MR2282969}.
For development of the study of the mirror symmetric Gamma conjecture, we refer the reader to \cite{MR2553377, MR3112512, MR4194298, Iri23, FWZ23,  MR4922780, AFW25, You25}.
For research of period integrals, Hodge structures, or oscillatory integrals via tropical geometry, see \cite{MR2576286, MR4194298, MR4179831, MR4484542, BL22, FWZ23, MR4782805, MR4950977,  MR4922780}.

The assumption that the triangulation $\scrT$ is unimodular corresponds to the smoothness assumption in tropical geometry.
In general, there are discrepancies between tropical cohomology groups and ordinary cohomology groups if we do not impose the unimodularity assumption (cf.~\cite{MR2681794}).
In \cite{AMH25}, periods of one-parameter degenerations of hypersurfaces in projective spaces are computed without imposing the unimodularity assumption.
In this setting, the periods are described not only in terms of the gamma function but also in terms of Dirichlet $L$-functions.

\subsection{Intersection numbers of lifts}

We also have the following application of \pref{th:main2} to computation of intersection numbers of cycles:

\begin{theorem}\label{th:main4}
Let $w_1, w_2 \in \rint(\Delta)$, 
and $a_1 \in \MW^{d-q} \lb \Sigma_{w_1} \rb, a_2 \in \MW^{q} \lb \Sigma_{w_2} \rb$.
Let further $C(a_1)_t, C(a_2)_t$ be lifts of the tropical cycles $c(a_1), c(a_2)$.
Then one has
\begin{align}\label{eq:intersection-compare}
    \la C(a_1)_t, C(a_2)_t \ra=(-1)^{s} \cdot \la c(a_1), c(a_2)\ra,
\end{align}
where the left-hand side denotes the ordinary intersection number of the cycles $C(a_1)_t, C(a_2)_t$, and $s:=d(d+1)/2+d-q$.
\end{theorem}

Notice that although lifts of tropical cycles are not unique, the intersection number $\la C(a_1)_t, C(a_2)_t \ra$ in \pref{th:main4} does not depend on the choice of lifts.
By combining \pref{th:main4} and \pref{th:main1}(2), we can combinatorially (or tropically) compute the intersection number of lifts of tropical cycles for hypersurfaces.

In the forthcoming paper \cite{RZ21} announced in \cite{MR4294796}, they construct, for a given integral affine manifold with singularities $B$, a topological SYZ fibration $X \to B$ and lifts of tropical cycles in $B$ of arbitrary degree.
The proof of \pref{th:main4} is based on basically the same observation as the one in \cite[Theorem 7]{MR4347312} which gives the correspondence between the intersection number of tropical cycles and that of their lifts, for the topological SYZ fibrations of \cite{MR4294796}.

\subsection{Organization of this article}

After recalling the basic notions such as tropical homology groups, tropical cycles, and Minkowski weights in \pref{sc:prel}, we prove \pref{th:main1} in \pref{sc:proof1}.
We construct lifts of tropical cycles and prove \pref{th:main2} in \pref{sc:proof2}.
Then we prove \pref{th:main3} and \pref{cr:CY} in \pref{sc:proof3}.
\pref{th:main4} is lastly proved in \pref{sc:proof4}.

\section{Preliminaries}\label{sc:prel}

\subsection{Tropical homology}\label{sc:tropical-homology}

We recall the definition of tropical homology groups of \cite{MR3330789, MR3961331}.
We focus on the case where tropical spaces are rational polyhedral complexes in Euclidean spaces, since the tropical hypersurfaces considered in this article are of this type.
Tropical homology groups are defined for more general tropical spaces.
See \cite{MR3330789, MR3961331}.
This subsection is written based on \cite{MR3330789} and \cite[Section 2]{MR3894860}.

Let $d$ be a positive integer.
A \emph{rational polyhedron} in $\bR^{d}$ is the intersection of finitely many half spaces of the form 
\begin{align}
\lc x \in \bR^{d} \relmid \la m, x \ra \leq a \rc
\end{align}
with $m \in \lb \bZ^{d} \rb^\ast$ and $a \in \bR$.
Let $n$ be a non-negative integer, and let $\scrP$ be a rational polyhedral complex of pure dimension $n$ in $\bR^{d}$, i.e., a finite set of rational polyhedra in $\bR^{d}$ satisfying
\begin{itemize}
\item For any $P \in \scrP$, all faces of $P$ are also in $\scrP$. 
\item For $P_1, P_2 \in \scrP$, the intersection $P_1 \cap P_2$ is a face of both $P_1$ and $P_2$.
\item All maximal polyhedra in $\scrP$ are of dimension $n$.
\end{itemize}
We write the support of $\scrP$ as $X:=\bigcup_{P \in \scP} P \subset \bR^{d}$, and define $\dim X:=n$.

Let $p, q \in \bZ_{\geq 0}$.
The \emph{$p$-th integral multi-tangent space} $\scF_p^\bZ(P)$ of $\scrP$ at $P \in \scrP$ is defined by
\begin{align}
    \scF_p^\bZ(P):=\lb\sum_{P'\in \scrP, P'\succ P} \bigwedge^p T(P') \rb \cap  \bigwedge^p \bZ^d,
\end{align}
where $T(P) \subset \bR^{d}$ is the subspace generated by vectors tangent to $P$.
If $P_1 \prec P_2$, we have the inclusion
\begin{align}\label{eq:inclusion}
    \iota_{P_1, P_2} \colon \scF_p^\bZ(P_2) \hookrightarrow \scF_p^\bZ(P_1).
\end{align}
Let $\Delta^q$ denote the standard $q$-simplex, and let $C_q(P)$ be the free abelian group generated by singular $q$-simplices $\Delta^q \to X$ satisfying  
\begin{itemize}
\item The image of the relative interior of $\Delta^q$ is contained in the relative  interior of $P$.
\item The image of the relative interior of any face of $\Delta^q$ is contained in the relative interior of a face of $P$.
\end{itemize}
By abuse of notation, we identify the domain $\Delta^q$ and the image of a singular $q$-simplex $\Delta^q \to X$ with the singular $q$-simplex itself in the following.
We set
\begin{align}
    C_{q}(X, \scF_p^\bZ):=\bigoplus_{P \in \scrP} \scF_p^\bZ(P) \otimes_\bZ C_q(P),
\end{align}
and consider the boundary operator
\begin{align}
    \partial \colon C_{q}(X, \scF_p^\bZ) \to C_{q-1}(X, \scF_p^\bZ),
\end{align}
which is defines as the ordinary boundary operator composed with the inclusion \eqref{eq:inclusion}.
The \emph{tropical homology group} is the homology of the complex $\lb C_{\bullet}(X, \scF_p^\bZ), \partial \rb$, and is denoted by
\begin{align}
    H_q\lb X, \scF_p^\bZ \rb:=H_q \lb C_{\bullet}(X, \scF_p^\bZ), \partial \rb.
\end{align}
A representative of a class of the tropical homology group $H_q\lb X, \scF_p^\bZ \rb$ is called a \emph{tropical $(p, q)$-cycle}.
We say that a tropical $(p, q)$-cycle is \emph{transversal} if for any singular $q$-simplex $\Delta^q \to X$ appearing in the cycle and any face $F \prec \Delta^q$, the dimension of the polyhedron $P \in \scrP$ containing the image of $\rint(F)$ in its relative interior is greater than or equal to $\dim X-(q-\dim F)$ (\cite[Definition 6.4]{MR3330789}).

We recall the intersection product for tropical cycles in the following.
See \cite[Section 6.2]{MR3330789} for more details.
Suppose $p+q=\dim X$.
Let $c_a=\sum_{\Delta} a_\Delta \cdot \Delta \in C_{q}(X, \scF_p^\bZ)$ and $c_b=\sum_{\Delta'} b_{\Delta'} \cdot \Delta' \in C_{p}(X, \scF_q^\bZ)$ be transversal tropical $(p, q), (q, p)$-cycles, where $\Delta \in C_q(P), \Delta' \in C_p(P')$ and $a_\Delta \in \scF_p^\bZ(P), b_{\Delta'} \in \scF_q^\bZ(P')$ for some $P, P' \in \scrP$.
We say that these are a \emph{transversal pair} if for any faces $F \prec \Delta, F' \prec \Delta'$ of any simplices appearing in $c_a, c_b$, whose relative interiors are contained in the relative interior of the same polyhedron $P'' \in \scrP$, the relative interiors of $F$ and $F'$ are transversal in $\rint (P'')$ (\cite[Definition 6.5]{MR3330789}).

Suppose that simplices $\Delta \in C_q(P), \Delta' \in C_p(P')$ 
appearing in a transversal pair $c_a=\sum_{\Delta} a_\Delta \cdot \Delta \in C_{q}(X, \scF_p^\bZ)$, $c_b=\sum_{\Delta'} b_{\Delta'} \cdot \Delta' \in C_{p}(X, \scF_q^\bZ)$ intersect.
Then we have $P=P'$, and $P=P'$ is a maximal polyhedron in $\scrP$.
Furthermore, the intersection points of $\Delta$ and $\Delta'$ are contained in $\rint(P)=\rint(P')$.
For an intersection point $x \in \Delta \cap \Delta'$, let $\Omega_x$ denote the integral volume form on $\rint(P)$, which defines the orientation on $\rint(P)$ 
under which the simplices $\Delta, \Delta'$ intersect positively at the point $x$.
The \emph{tropical intersection number} $\la c_a, c_b \ra$ is defined by
\begin{align}\label{eq:tropical-intersection}
    \la c_a, c_b \ra:=\sum_{P \in \scrP} 
    \sum_{\substack{\Delta \in C_q(P) \\ \Delta' \in C_p(P)}} 
    \sum_{x \in \Delta \cap \Delta'} 
    \Omega_x 
    \lb a_\Delta \wedge b_{\Delta'} \rb,
\end{align}
where $\Delta \in C_q(P), \Delta' \in C_p(P)$ run over all simplices appearing in the tropical cycles $c_a, c_b$.
In this article, we call the integer $\Omega_x \lb a_\Delta \wedge b_{\Delta'} \rb$ the \emph{tropical multiplicity} at the intersection point $x$.

\begin{remark}
The tropical intersection number of a transversal pair of tropical cycles is also defined for more general tropical spaces (cf.~\cite[Section 6.2]{MR3330789}).
It has been proved that when tropical spaces are compact and smooth (in the sense of \cite[Definition 1.14]{MR3330789}), 
every pair of tropical homology classes is represented by a transversal pair of cycles, and the pairing of taking tropical intersection numbers descends to a pairing on the tropical homology groups (\cite[Corollary 6.11, Proposition 6.13]{MR3330789}).
\end{remark}

\begin{remark}
Tropical intersection numbers are also defined for tropical cycles in integral affine manifolds with singularities \cite{MR4347312}.
The definition is analogous to the above one.
See Section 6 of Introduction in loc.cit.
\end{remark}

\subsection{Minkowski weights}

Let $N, M$ be free $\bZ$-modules of rank $d+1$ dual to each other, which we considered in \pref{sc:intro}.
Let further $\Sigma$ be a complete fan in $M_\bR$, corresponding to a complete toric variety $Y_\Sigma$.
The group $\MW^{k} \lb \Sigma \rb$ of \emph{Minkowski weights} of codimension $k$ on $\Sigma$ of \cite{MR1415592} $(0 \leq k \leq d+1)$ is defined by
\begin{align}\label{eq:minkowski}
    \MW^{k} \lb \Sigma \rb:=
    \lc a \colon \Sigma (d+1-k) \to \bZ 
    \relmid 
    \forall \tau \in \Sigma(d-k), 
    \sum_{\sigma \in \Sigma(d+1-k), \sigma \succ \tau} a(\sigma) \cdot m_{\sigma, \tau} \in \bR \cdot \tau \rc,
\end{align}
where $m_{\sigma, \tau} \in \sigma \cap M$ is a point which represents the primitive generator of the ray $\lb \sigma+\bR \cdot \tau \rb/ \bR \cdot \tau$.
The group $\MW^{k} \lb \Sigma \rb$ of Minkowski weights is canonically isomorphic to the Chow cohomology group $A^{k}(Y_\Sigma)$ of the toric variety $Y_\Sigma$ via
\begin{align}\label{eq:amw}
    A^k(Y_\Sigma) \to \MW^{k}\lb \Sigma \rb, \quad 
    \alpha \mapsto \lb \sigma \mapsto \deg \lb \alpha \cap \ld V(\sigma) \rd\rb\rb,
\end{align}
where $V(\sigma) \subset Y_\Sigma$ is the closure of the torus orbit corresponding to the cone $\sigma \in \Sigma(d+1-k)$ (\cite[Theorem 3.1]{MR1415592}).

\begin{theorem}{\rm(\cite[Theorem in Introduction]{MR1415592})}\label{th:MW}
Let $a_1 \in \MW^p(\Sigma)$, 
$a_2 \in \MW^q(\Sigma)$ be Minkowski weights of codimensions $p$ and $q$ $(0 \leq p, q \leq d+1)$.
The Minkowski weight $a_1 \cup a_2 \in \MW^{p+q}(\Sigma)$ corresponding to the cup product of the classes defined by $a_1$ and $a_2$ in the Chow cohomology ring $A^\ast (Y_\Sigma)$ is given by
\begin{align}\label{eq:MW-cup}
\lb a_1 \cup a_2 \rb(\rho)
=
\sum_{\sigma, \sigma'} 
a_1(\sigma) \cdot a_2(\sigma') \cdot
\ld M : \bZ(M \cap \sigma)+\bZ(M \cap \sigma') \rd
\quad
(\rho \in \Sigma (d+1-p-q)),
\end{align}
where $\ld M : \bZ(M \cap \sigma)+\bZ(M \cap \sigma') \rd$ is the index of the sublattice $\bZ(M \cap \sigma)+\bZ(M \cap \sigma') \subset M$, and the sum is taken over
\begin{align}\label{eq:cone-pair}
    \lc (\sigma, \sigma') \in \Sigma (d+1-p) \times \Sigma (d+1-q) \relmid 
    \sigma \succ \rho, \sigma' \succ \rho,
    \sigma \cap (m_0+\sigma') \neq \emptyset \rc,
\end{align}
which is determined by the choice of a generic vector $m_0 \in M_\bR$.
\end{theorem}
We refer the reader to \cite[Example 4.3]{MR1415592} for an example illustrating \pref{th:MW}.

\section{Proof of \pref{th:main1}}\label{sc:proof1}

We will prove \pref{th:main1} in \pref{sc:main1-1}, \pref{sc:main1-2}, and \pref{sc:main1-3}.
In \pref{sc:example}, we provide a couple of examples illustrating \pref{th:main1}.

\subsection{Proof of \pref{th:main1}(1)}\label{sc:main1-1}

Let $\scS=\lc \sigma_1 \prec \cdots \prec \sigma_{q+1}\rc \in \scrS_q^w$ be a sequence of cones.
For elements $e_1, \cdots, e_{q+1}$ considered in \eqref{eq:fd-q}, we choose elements $e_{q+2}, \cdots, e_{d+1}$ of $M$ so that $e_1, \cdots, e_{d+1}$ form a basis of $M$ and $\bigwedge_{i=1}^{d+1} e_i^\ast= \vol(N)$.
Then we have $f(\scS) =\bigwedge_{i=q+2}^{d+1} e_i^\ast$.
The image of the relative interior of the singular $q$-simplex $\Delta_\scS$ is contained in the relative interior of the facet of $\nabla_w$ that is dual to the cone $\sigma_1 \in \Sigma_w(1)$.
We write the facet as $\nabla_{\sigma_1} \prec \nabla_w$.
The tangent space of $\nabla_{\sigma_1}$ is $e_1^\perp$, and we have $\scF_{d-p}^\bZ (\nabla_{\sigma_1})=\lb \bigwedge^{d-p} e_1^\perp\rb \cap \lb \bigwedge^{d-p} N \rb$.
Therefore, we can see that $f(\scS) =\bigwedge_{i=q+2}^{d+1} e_i^\ast \in \scF_{d-p}^\bZ (\nabla_{\sigma_1})$, and \eqref{eq:trop-chain} is a tropical $(d-q, q)$-chain in the tropical hypersurface $X(\trop (f))$.

We prove that \eqref{eq:trop-chain} is a cycle, i.e., $\partial c(a)=0$.
We have
\begin{align}\label{eq:c-boundary}
\partial c(a)
=\sum_{\scS \in \scrS_q^w} a(\sigma[\scS]) \cdot f (\scS) \cdot \sum_{i=1}^{q}(-1)^{i-1} \partial_i \Delta_\scS
+
(-1)^{q}
\sum_{\scS \in \scrS_q^w} a(\sigma[\scS]) \cdot f (\scS) \cdot \partial_{q+1} \Delta_\scS,
\end{align}
where $\partial_i \Delta_\scS$ denotes the facet of $\Delta_\scS$, which does not have the $i$-th vertex of $\Delta_\scS$ ($1 \leq i \leq q+1$).
For $\scS=\lc \sigma_1 \prec \cdots \prec \sigma_{q+1} \rc \in \scrS_q^w$ and $i \in \lc 1, \cdots, q \rc$, we set
\begin{align}
    \sigma_i'&:=\cone (\lc e_1, \cdots, e_{i-1}, e_{i+1} \rc) \in \Sigma_w(i) \\ \label{eq:S[i]}
    \scS[i]&:=\lc \sigma_1 \prec \cdots \prec \sigma_{i-1} \prec \sigma_i' \prec \sigma_{i+1} \prec \cdots \prec \sigma_{q+1} \rc \in \scrS_q^w.
\end{align}
Then we have 
$\sigma[\scS]=\sigma[\scS[i]]=\sigma_{q+1}$, 
$f(\scS)=-f(\scS[i])$, and
$\partial_i \Delta_{\scS[i]}=\partial_i \Delta_\scS$.
From these, we can see that in the former part of \eqref{eq:c-boundary}, 
the terms for $(\scS, i)$ and $(\scS[i], i)$ cancel each other out.
Thus the former part of \eqref{eq:c-boundary} is $0$.

We prove that the latter part of \eqref{eq:c-boundary} is also $0$.
For any sequence of cones $\scS'=\lc \sigma_1 \prec \cdots \prec \sigma_{q}\rc \in \scrS_{q-1}^w$, the coefficient of $\Delta_{\scS'}$ in the latter part of \eqref{eq:c-boundary} is
\begin{align}\label{eq:latter}
(-1)^{q}
\sum_{\sigma \in \Sigma_w(q+1), \sigma \succ \sigma_q} a(\sigma) \cdot f (\lc\sigma_1 \prec \cdots \prec \sigma_q \prec \sigma\rc).
\end{align}
Let $e_1, \cdots, e_{q+1}$ be the elements considered in \eqref{eq:fd-q} for $\lc\sigma_1 \prec \cdots \prec \sigma_q \prec \sigma\rc \in \scrS_q^w$.
Then $e_{q+1} \in \sigma \cap M$ represents the primitive generator of $\left. (\sigma+\bR \cdot \sigma_q)\middle/\bR \cdot \sigma_q \right.$, and \eqref{eq:latter} equals
\begin{align}\label{eq:MW-boundary}
(-1)^{q}
\la
\bigwedge_{i=1}^q e_i
\wedge
\sum_{\sigma \in \Sigma_w(q+1), \sigma \succ \sigma_q} a(\sigma) m_{\sigma, \sigma_q}
\wedge \bullet, \vol(N)
\ra
=0,
\end{align}
by the definition of Minkowski weights \eqref{eq:minkowski}.
Therefore, \eqref{eq:c-boundary} equals $0$, and $c(a)$ is a tropical $(d-q, q)$-cycle.
Thus we conclude \pref{th:main1}(1).

\subsection{Proof of \pref{th:main1}(2-a)}\label{sc:main1-2}
 
Suppose $w_1=w_2=:w \in \rint(\Delta)$.
We consider the triangulation $\scrT_{\partial \nabla_w}$ of $\partial \nabla_w$ defined by 
\begin{align}\label{eq:bary-div}
\scrT_{\partial \nabla_w}
:=
\lc \conv \lb \lc b_{\sigma_1}, \cdots, b_{\sigma_{q+1}} \rc \rb 
\relmid 
\lc 0 \rc \neq \sigma_1 \prec \cdots \prec \sigma_{q+1} \in \Sigma_w, q \geq 0 \rc.
\end{align}
(We do not impose $\sigma_i \in \Sigma_w(i)$ in \eqref{eq:bary-div}.)
The points $b_{\sigma_i} \in \partial \nabla_w$ are the ones that we chose in \pref{sc:main1}.
For each cone $\sigma \in \Sigma_w$, we also choose a point $\check{b}_\sigma \in \rint \sigma$.
Let $S \subset M_\bR$ be the unit sphere in the Euclidean space $M_\bR \cong \bR^{d+1}$ whose center point is $0 \in M_\bR$, and consider its triangulation $\scrT_S^w$ defined by
\begin{align}\label{eq:bary-div2}
\scrT_{S}^w
:=
\lc S \cap \cone \lb \lc \check{b}_{\sigma_1}, \cdots, \check{b}_{\sigma_{q+1}} \rc \rb \relmid \lc 0 \rc \neq \sigma_1 \prec \cdots \prec \sigma_{q+1} \in \Sigma_w, q \geq 0 \rc.
\end{align}
The fan $\Sigma_w$ and the polytope $\nabla_w$ are dual to each other, and the triangulations $\scrT_{\partial \nabla_w}$ and $\scrT_{S}^w$ of the sphere $\partial \nabla_w \cong S$ are isomorphic as simplicial complexes by the correspondence 
\begin{align}
\conv \lb \lc b_{\sigma_1}, \cdots, b_{\sigma_{q+1}} \rc \rb
\leftrightarrow 
S \cap \cone \lb \lc \check{b}_{\sigma_1}, \cdots, \check{b}_{\sigma_{q+1}} \rc \rb.
\end{align}
We take a homeomorphism 
\begin{align}\label{eq:phi-w}
    \phi_w \colon \partial \nabla_w \to S
\end{align}
preserving the triangulations $\scrT_{\partial \nabla_w}$ and $\scrT_{S}^w$.
(In particular, we have $\phi_w \lb b_{\sigma} \rb=S \cap \cone \lb \lc \check{b}_\sigma \rc \rb$ for all cones $\sigma \neq \lc 0 \rc$ in $\Sigma_w$.)

The tropical cycles $c(a_1)$ and $c(a_2)$ are supported in
\begin{align}\label{eq:c1-support}
\bigcup_{\substack{\sigma_i \in \Sigma_w(i) \\ \sigma_1 \prec \cdots \prec \sigma_{q+1}}} \conv \lb \lc b_{\sigma_1}, \cdots, b_{\sigma_{q+1}} \rc \rb
\subset \partial \nabla_w \\\label{eq:c2-support}
\bigcup_{\substack{\sigma'_i \in \Sigma_w(i) \\ \sigma'_1 \prec \cdots \prec \sigma'_{d+1-q}}} \conv \lb \lc b_{\sigma'_1}, \cdots, b_{\sigma'_{d+1-q}} \rc \rb
\subset \partial \nabla_w,
\end{align}
and the homeomorphism $\phi_w$ sends \eqref{eq:c1-support} and \eqref{eq:c2-support} to 
\begin{align}
\bigcup_{\substack{\sigma_i \in \Sigma_w(i) \\ \sigma_1 \prec \cdots \prec \sigma_{q+1}}} S \cap \cone \lb \lc \check{b}_{\sigma_1}, \cdots, \check{b}_{\sigma_{q+1}} \rc \rb
&\subset S \\
\bigcup_{\substack{\sigma'_i \in \Sigma_w(i) \\ \sigma'_1 \prec \cdots \prec \sigma'_{d+1-q}}} S \cap \cone \lb \lc \check{b}_{\sigma'_1}, \cdots, \check{b}_{\sigma'_{d+1-q}} \rc \rb
&\subset S
\end{align}
respectively.
Take a small generic element $m_0 \in M_\bR$ as in \pref{th:MW}, and consider
\begin{align}\label{eq:c2-support2}
\phi_w^{-1} \lb \bigcup_{\substack{\sigma'_i \in \Sigma_w(i) \\ \sigma'_1 \prec \cdots \prec \sigma'_{d+1-q}}} S \cap \lc m_0+\cone \lb \lc \check{b}_{\sigma'_1}, \cdots, \check{b}_{\sigma'_{d+1-q}} \rc \rb \rc \rb
\subset \partial \nabla_w.
\end{align}
This is a small deformation of \eqref{eq:c2-support}.
We deform the tropical cycle $c(a_2)$ so that its support is contained in \eqref{eq:c2-support2}.
Let $c(a_2)'$ denote the tropical cycle obtained by this deformation.

Since the point $m_0$ is generic, the intersection of cones $\sigma_{q+1} \in \Sigma_w(q+1)$ and $m_0+\sigma'_{d+1-q}$ $(\sigma'_{d+1-q} \in \Sigma_w(d+1-q))$ is transverse.
Furthermore, if the intersection is non-empty, then it is of dimension $1$.
Since a tangent vector of the intersection must be contained in both $\sigma_{q+1}$ and $\sigma'_{d+1-q}$, there must be a cone $\rho \in \Sigma_w(1)$ such that $\sigma_{q+1} \succ \rho$ and $\sigma'_{d+1-q} \succ \rho$.
From
\begin{align}
\cone \lb \lc \check{b}_{\sigma_1}, \cdots, \check{b}_{\sigma_{q+1}} \rc \rb &\subset \sigma_{q+1}\\
m_0+\cone \lb \lc \check{b}_{\sigma'_1}, \cdots, \check{b}_{\sigma'_{d+1-q}} \rc \rb &\subset m_0+\sigma'_{d+1-q},
\end{align}
we can see that the tropical cycles $c(a_1)$ and $c(a_2)'$ intersect transversely, via the homeomorphism $\phi_w$.
The homeomorphism $\phi_w$ sends \eqref{eq:c1-support} and \eqref{eq:c2-support2} to 
\begin{align}
\bigcup_{\substack{\sigma_i \in \Sigma_w(i) \\ \sigma_1 \prec \cdots \prec \sigma_{q+1}}} S \cap \cone \lb \lc \check{b}_{\sigma_1}, \cdots, \check{b}_{\sigma_{q+1}} \rc \rb
=
S \cap \bigcup_{\sigma \in \Sigma_w(q+1)} \sigma
\end{align}
and
\begin{align}
\bigcup_{\substack{\sigma'_i \in \Sigma_w(i) \\ \sigma'_1 \prec \cdots \prec \sigma'_{d+1-q}}} S \cap \lc m_0+\cone \lb \lc \check{b}_{\sigma'_1}, \cdots, \check{b}_{\sigma'_{d+1-q}} \rc \rb \rc
=
S \cap \bigcup_{\sigma' \in \Sigma_w(d+1-q)} m_0+\sigma'
\end{align}
respectively.
We can see from these that there is a one-to-one correspondence between the set of intersection points of $c(a_1), c(a_2)'$ and the set
\begin{align}\label{eq:intersection-pts}
    \lc (\sigma, \sigma') \in \Sigma_w (q+1) \times \Sigma_w (d+1-q) \relmid 
    \sigma \cap (m_0+\sigma') \neq \emptyset, 
    \exists \rho \in \Sigma_w(1)\ \mathrm{s.t.} \
    \sigma \succ \rho, \sigma' \succ \rho \rc.
\end{align}
We can also see from this correspondence and the the definition of the right-hand side of \eqref{eq:w1=w2-intersection} that it suffices to show that for any $\rho \in \Sigma_w(1)$ and 
$\sigma \in \Sigma_w(q+1), \sigma' \in \Sigma_w(d+1-q)$ such that 
$\sigma \succ \rho, \sigma' \succ \rho$ and $\sigma \cap (m_0+\sigma') \neq \emptyset$, the number $a_1(\sigma) \cdot a_2(\sigma') \cdot
\ld M : \bZ(M \cap \sigma)+\bZ(M \cap \sigma') \rd$
appearing in \eqref{eq:MW-cup} coincides with the tropical multiplicity 
($\Omega_x \lb a_\Delta \wedge b_{\Delta'} \rb$ 
in \eqref{eq:tropical-intersection}) of the tropical intersection $\la c(a_1), c(a_2)' \ra$ at the intersection point corresponding to $(\sigma, \sigma')$.
We will show this in the following.

Let $\rho \in \Sigma_w(1), \sigma \in \Sigma_w(q+1), \sigma' \in \Sigma_w(d+1-q)$ be cones such that $\sigma \succ \rho, \sigma' \succ \rho$ and $\sigma \cap (m_0+\sigma') \neq \emptyset$.
Suppose that the intersection point corresponding to $(\sigma, \sigma')$ is the intersection point of simplices $\Delta_\scS$ and (the deformation of) $\Delta_{\scS'}$ with 
$\scS=\lc \sigma_1 \prec \cdots \prec \sigma_{q+1}=\sigma \rc, 
\scS'=\lc \sigma'_1 \prec \cdots \prec \sigma'_{d+1-q}=\sigma' \rc$.
The intersection point is contained in the relative interior of the facet of $\nabla_w$ dual to $\rho$.
The simplices $\Delta_\scS$ and $\Delta_{\scS'}$ are contained in the facet, and we have $\sigma_1=\sigma'_1=\rho$.

Let $e_i \in M$ $\lb 1 \leq i \leq d+1\rb$ be the elements such that 
$\lc e_i \relmid 1 \leq i \leq j \rc$ are the primitive generators of 
$\sigma_j$ for all $1 \leq j \leq q+1$, 
and $\lc e_1 \rc \cup \lc e_i \relmid q+2 \leq i \leq k \rc$ are the primitive generators of $\sigma'_{k-q}$ for all $q+2 \leq k \leq d+1$.
We set $I:=\la \bigwedge_{i=1}^{d+1} e_i, \vol (N)\ra \in \bZ$.
Then we have
\begin{align}\label{eq:MW-cup0}
a_1(\sigma) \cdot a_2 (\sigma') \cdot \ld M : \bZ(M \cap \sigma)+\bZ(M \cap \sigma') \rd
=
a_1(\sigma) \cdot a_2 (\sigma') \cdot |I|.
\end{align}
On the other hand, the coefficients of $\Delta_\scS$ and $\Delta_{\scS'}$ in the tropical cycles $c(a_1)$ and $c(a_2)$ are
\begin{align}
a_1(\sigma) \cdot \la \bigwedge_{i=1}^{q+1} e_i \wedge \bullet, \vol (N)\ra, \quad 
a_2 (\sigma') \cdot \la e_1 \wedge \bigwedge_{i=q+2}^{d+1} e_i \wedge \bullet, \vol (N)\ra
\end{align}
respectively.
The integral volume form on the facet of $\nabla_w$ dual to $\rho$, which defines the orientation under which the simplices $\Delta_\scS$ and $\Delta_{\scS'}$ intersect positively is
\begin{align}\label{eq:integral-volume}
\frac{(-1)^d}{|I|} \cdot \bigwedge_{i=2}^{d+1} e_i \in \bigwedge^d \lb \left. M \middle/ \bZ e_1 \right. \rb.
\end{align}
From these, we can see that the tropical multiplicity of the tropical intersection $\la c(a_1), c(a_2) \ra$ at the intersection point corresponding to $(\sigma, \sigma')$ is
\begin{align}\label{eq:int-number}
\begin{split}
\frac{(-1)^d}{|I|} \cdot (-1)^{q(d-q)} \cdot 
a_1(\sigma) \cdot \la \bigwedge_{i=1}^{q+1} e_i \wedge \bigwedge_{j=q+2}^{d+1} e_j, \vol (N)\ra \cdot a_2 (\sigma') \cdot \la e_1 \wedge \bigwedge_{i=q+2}^{d+1} e_i \wedge \bigwedge_{j=2}^{q+1} e_j, \vol (N)\ra\\
=
\frac{(-1)^d}{|I|} \cdot (-1)^{q(d-q)} \cdot
a_1(\sigma) \cdot I \cdot
a_2 (\sigma') \cdot (-1)^{q(d-q)} I,
\end{split}
\end{align}
and this coincides with \eqref{eq:MW-cup0} multiplied by $(-1)^d$.
Thus we can conclude \pref{th:main1}(2-a).

\subsection{Proof of \pref{th:main1}(2-b), (2-c)}\label{sc:main1-3}

We prove \pref{th:main1}(2-b).
Suppose that we have $w_1 \neq w_2$ and $\conv (\lc w_1, w_2 \rc) \in \scrT$.
When choosing a point $b_{\sigma'}$ for each cone $\sigma'=\bR_{\geq 0} \cdot (\tau-w_2) \in \Sigma_{w_2}$ $(\tau \in \scrT)$ in the construction of the tropical cycle $c(a_2)$, we choose it so that if $\tau \succ \conv \lb \lc w_1, w_2\rc \rb$, then $b_{\sigma'}=b_\sigma$, where $\sigma:=\bR_{\geq 0} \cdot (\tau-w_1) \in \Sigma_{w_1}$.
We consider the triangulations $\scrT_{\partial \nabla_{w_2}}$, $\scrT_S^{w_2}$ and the homeomorphism $\phi_{w_2} \colon \partial \nabla_{w_2} \to S$ 
(\eqref{eq:bary-div}, \eqref{eq:bary-div2}, \eqref{eq:phi-w} for $w=w_2$).
The tropical cycles $c(a_1)$ and $c(a_2)$ are supported in $\partial \nabla_{w_1}$ and $\partial \nabla_{w_2}$ respectively.
Therefore, they can intersect only in $\partial \nabla_{w_1} \cap \partial \nabla_{w_2}$ which is the cell dual to $\conv (\lc w_1, w_2 \rc) \in \scrT$.
The homeomorphism $\phi_{w_2}$ sends the parts of the tropical cycles $c(a_1)$ and $c(a_2)$ contained in the cell to 
\begin{align}
\bigcup_{\substack{\sigma_i \in \Sigma_{w_2}(i) \\ \rho=\sigma_1 \prec \cdots \prec \sigma_{q+1}}} S \cap \cone \lb \lc \check{b}_{\sigma_1}, \cdots, \check{b}_{\sigma_{q+1}} \rc \rb
&\subset S\\
\bigcup_{\substack{\sigma'_i \in \Sigma_{w_2}(i) \\ \rho=\sigma'_1 \prec \cdots \prec \sigma'_{d+1-q}}} S \cap \cone \lb \lc \check{b}_{\sigma'_1}, \cdots, \check{b}_{\sigma'_{d+1-q}} \rc \rb
&\subset S
\end{align}
respectively, where $\rho:=\bR_{\geq 0} \cdot (w_1-w_2) \in \Sigma_{w_2}(1)$.
As we did in \pref{sc:main1-2}, we take a small generic element $m_0 \in M_\bR$ and consider \eqref{eq:c2-support2} with $w=w_2$.
We deform the tropical cycle $c(a_2)$ so that 
the support of the part contained in $\partial \nabla_{w_1} \cap \partial \nabla_{w_2}$
is contained in \eqref{eq:c2-support2} with $w=w_2$.
Let $c(a_2)'$ denote the tropical cycle obtained by the deformation.
By the same argument as before, we can see that the tropical cycles $c(a_1)$ and $c(a_2)'$ intersect transversely, and that there is a one-to-one correspondence between the set of intersection points of $c(a_1), c(a_2)'$ and the set
\begin{align}\label{eq:intersection-pts-2}
    \lc (\sigma, \sigma') \in \Sigma_{w_2} (q+1) \times \Sigma_{w_2} (d+1-q) \relmid 
    \sigma \cap (m_0+\sigma') \neq \emptyset, 
    \sigma \succ \rho, \sigma' \succ \rho \rc.
\end{align}

Let $(\sigma, \sigma')$ be an arbitrary element in \eqref{eq:intersection-pts-2}.
Suppose that the intersection point corresponding to $(\sigma, \sigma')$ is the intersection point of simplices $\Delta_\scS$ (in $c(a_1)$) and (the deformation of) $\Delta_{\scS'}$ (in $c(a_2)'$) with 
\begin{align}
\scS&=\lc \sigma_1'' \prec \cdots \prec \sigma_{q+1}'' \rc \in \scrS_q^{w_1} \\
\scS'&=\lc \sigma'_1 \prec \cdots \prec \sigma'_{d+1-q} \rc \in \scrS_{d-q}^{w_2}.
\end{align}
We have $\sigma_1''=-\rho$, $\sigma'_1=\rho$, and $\sigma'_{d+1-q}=\sigma'$.
There exists a sequence of simplices $\tau_1 \prec \cdots \prec \tau_{q+1} \in \scrT$ such that $\sigma_i''=\bR_{\geq 0} \cdot (\tau_i-w_1)$.
(We have $\tau_1=\conv (\lc w_1, w_2 \rc)$.)
We set $\sigma_i:=\bR_{\geq0} \cdot (\tau_i-w_2) \in \Sigma_{w_2}(i)$.
Then we have $\sigma_1=\rho$ and $\sigma_{q+1}=\sigma$.

Let $e_i \in M$ $\lb 1 \leq i \leq d+1\rb$ be the elements such that 
$\lc e_i \relmid 1 \leq i \leq j \rc$ are the primitive generators of 
$\sigma_j$ for all $1 \leq j \leq q+1$, 
and $\lc e_1 \rc \cup \lc e_i \relmid q+2 \leq i \leq k \rc$ are the primitive generators of $\sigma'_{k-q}$ for all $q+2 \leq k \leq d+1$.
Then the coefficients of $\Delta_\scS$ and $\Delta_{\scS'}$ in the tropical cycles $c(a_1)$ and $c(a_2)$ are
\begin{align}
a_1(\sigma_{q+1}'') \cdot \la -e_1 \wedge \bigwedge_{i=2}^{q+1} (e_i-e_1) \wedge \bullet, \vol (N)\ra, \quad 
a_2 (\sigma') \cdot \la e_1 \wedge \bigwedge_{i=q+2}^{d+1} e_i \wedge \bullet, \vol (N)\ra
\end{align}
respectively.
The integral volume form on $\partial \nabla_{w_1} \cap \partial \nabla_{w_2}$ (the cell dual to $\conv (\lc w_1, w_2 \rc) \in \scrT$), which defines the orientation under which the simplices $\Delta_\scS$ and $\Delta_{\scS'}$ intersect positively is the same as \eqref{eq:integral-volume} with $I:=\la \bigwedge_{i=1}^{d+1} e_i, \vol (N)\ra$.
By the same computation as \eqref{eq:int-number}, we can see that the tropical  multiplicity of the tropical intersection $\la c(a_1), c(a_2)' \ra$ at the intersection point corresponding to $(\sigma, \sigma')$ equals 
\begin{align}
(-1)^{d+1} \cdot a_1(\sigma_{q+1}'') \cdot a_2 (\sigma') \cdot |I|
=
(-1)^{d+1} \cdot a_1(\sigma_{q+1}'') \cdot a_2 (\sigma') \cdot \ld M : \bZ(M \cap \sigma)+\bZ(M \cap \sigma') \rd.
\end{align}

We have a bijection from \eqref{eq:intersection-pts-2} to 
\begin{align}\label{eq:intersection-pts-3}
    \lc (\sigma, \sigma') \in \Sigma_{w_1, w_2} (q) \times \Sigma_{w_1, w_2} (d-q) \relmid 
    \sigma \cap (m_0'+\sigma') \neq \emptyset \rc
\end{align}
given by 
$(\sigma, \sigma') \mapsto (\left. (\sigma+\bR \cdot \rho)\middle/ \bR \cdot \rho \right., \left. (\sigma'+\bR \cdot \rho)\middle/ \bR \cdot \rho \right.)
=:(\overline{\sigma}, \overline{\sigma}')$,
where $m_0' \in \left. M_\bR \middle/ \bR \cdot \rho \right.$ is the projection of the generic element $m_0 \in M_\bR$ we took.
Therefore, we obtain
\begin{align}
    \la c(a_1), c(a_2)\ra
    &=
    \sum_{(\sigma, \sigma')} (-1)^{d+1} \cdot a_1(\sigma_{q+1}'') \cdot a_2 (\sigma') \cdot \ld M : \bZ(M \cap \sigma)+\bZ(M \cap \sigma') \rd\\
    &=
     \sum_{(\overline{\sigma}, \overline{\sigma}')} (-1)^{d+1} \cdot \overline{a}_1(\overline{\sigma}) \cdot \overline{a}_2 (\overline{\sigma}') \cdot \ld M' : \bZ(M' \cap \overline{\sigma})+\bZ(M' \cap \overline{\sigma}') \rd\\
    &=(-1)^{d+1} \cdot \overline{a}_1 \cup \overline{a}_2,
\end{align}
where $(\sigma, \sigma')$ and $(\overline{\sigma}, \overline{\sigma}')$ run over
\eqref{eq:intersection-pts-2} and \eqref{eq:intersection-pts-3} respectively, and $M':=\left. M\middle/ \bZ \cdot (w_1-w_2)\right.$.
We conclude \pref{th:main1}(2-b).

The last claim \pref{th:main1}(2-c) is obvious.
When $\conv(\lc w_1, w_2 \rc) \nin \scrT$, we have $\nabla_{w_1} \cap \nabla_{w_2}=\emptyset$.
Since the tropical cycles $c(a_1)$ and $c(a_2)$ are supported in $\partial \nabla_{w_1}$ and $\partial \nabla_{w_2}$ respectively, they never intersect.
Thus we have $\la c(a_1), c(a_2)\ra=0$.

\subsection{Examples}\label{sc:example}

\begin{example}\label{eg:k3}
Let $d=2, q=1$.
Take a basis $\lc e_1, e_2, e_3 \rc$ of $M \cong \bZ^3$, and let $\lc e_1^\ast, e_2^\ast, e_3^\ast \rc$ be the dual basis of $N:=\Hom(M, \bZ)$.
We set $e_0:=-(e_1+e_2+e_3) \in M$.
We consider the polynomial
\begin{align}\label{eq:f1}
    f:=\sum_{m \in \Delta \cap M} x^{\lambda_m} z^m \in K \ld M \rd
\end{align}
with
\begin{align}
    \Delta:=\conv \lb \lc e_0, e_1, e_2, e_3\rc \rb \subset M_\bR, \quad
\lambda_m:=
\left\{ \begin{array}{ll}
1 & m=e_0 \\
0 & m \in \lb \Delta \cap M \rb \setminus \lc e_0 \rc.
\end{array} 
\right. 
\end{align}
The tropical hypersurface $X(\trop (f)) \subset N_\bR \cong \bR^3$ obtained by tropicalization is shown by black lines in \pref{fg:surface}.
\begin{figure}[htbp]
	\centering
	\includegraphics[scale=0.8]{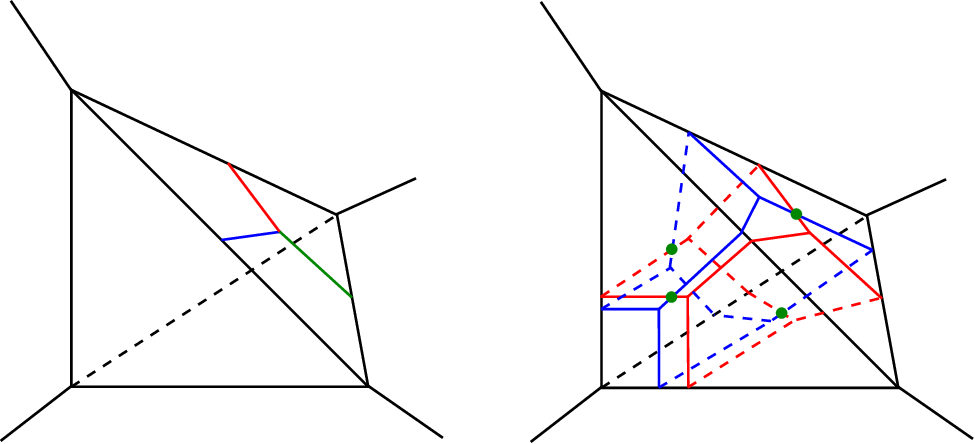}
    \caption{The tropical cycle $c(a_0)$ in the tropical hypersurface $X(\trop(f))$ and its self-intersection}
\label{fg:surface}
\end{figure}

We have $W=\lc 0\rc$, and the fan $\Sigma_{w=0}$ 
is the fan of $\bP^3$.
The group $\MW^1(\Sigma_0)$ of Minkowski weights is isomorphic to $\bZ$ and is generated by the Minkowski weight $a_0 \colon \Sigma_0(2) \to \bZ$ sending all cones in $\Sigma_0(2)$ to $1$.
We consider the sequences of cones 
\begin{align}\label{eq:S_i}
    \scS_i:= \lc \cone(\lc e_0 \rc) \prec \cone(\lc e_0, e_i \rc) \rc \in \scrS_{q=1}^{w=0} \quad (i=1, 2, 3).
\end{align}
The $1$-simplices $\Delta_{\scS_i}$ $(i=1, 2, 3)$ are the red, blue, and green line segments shown on the left of \pref{fg:surface} respectively.
If we set $\vol(N):=e_1^\ast \wedge e_2^\ast \wedge e_3^\ast \in \bigwedge^3 N$, then we have
\begin{align}
    f(\scS_i)
    =\la e_0 \wedge e_i \wedge \bullet, \vol(N) \ra
    =
\left\{ \begin{array}{ll}
-e_2^\ast+e_3^\ast & i=1\\
-e_3^\ast+e_1^\ast & i=2\\
-e_1^\ast+e_2^\ast & i=3,
\end{array} 
\right.
\end{align}
and $\sum_{i=1}^3 a_0 (\sigma[\scS_i]) \cdot f(\scS_i)=\sum_{i=1}^3 f(\scS_i)=0$ for the generator $a_0$ of $\MW^1(\Sigma_0)$.
This corresponds to \eqref{eq:MW-boundary}.
The support of the tropical $(1, 1)$-cycle $c(a_0)$ associated with $a_0 \in  \MW^1(\Sigma_0)$ is shown by red lines on the right of \pref{fg:surface}.

We compute the cup product $a_0 \cup a_0 \in \MW^2(\Sigma_0)$ and the tropical intersection number $\la c(a_0), c(a_0)\ra$.
For instance, let $\varepsilon>0$ be a small real constant, and take $m_0:=\varepsilon \cdot (3e_1+2e_2+e_3) \in M_\bR$ as a generic vector in \pref{th:MW}.
Then \eqref{eq:cone-pair} is
\begin{align}\label{eq:four-1}
    \lc \lb \cone \lb \lc e_0, e_1 \rc \rb, \cone \lb \lc e_0, e_3 \rc \rb \rb \rc
    \quad & \mathrm{for}\ \rho=\cone \lb \lc e_0 \rc \rb\\
    \lc \lb \cone \lb \lc e_1, e_2 \rc \rb, \cone \lb \lc e_0, e_1 \rc \rb \rb \rc
    \quad & \mathrm{for}\ \rho=\cone \lb \lc e_1 \rc \rb\\
    \lc \lb \cone \lb \lc e_1, e_2 \rc \rb, \cone \lb \lc e_0, e_2 \rc \rb \rb \rc
    \quad & \mathrm{for}\ \rho=\cone \lb \lc e_2 \rc \rb\\ \label{eq:four-4}
    \lc \lb \cone \lb \lc e_1, e_3 \rc \rb, \cone \lb \lc e_0, e_3 \rc \rb \rb \rc
    \quad & \mathrm{for}\ \rho=\cone \lb \lc e_3 \rc \rb,
\end{align}
and the indices $\ld M : \bZ(M \cap \sigma)+\bZ(M \cap \sigma') \rd$ appearing in \eqref{eq:MW-cup} for these pairs of cones are all $1$.
We can see that $a_0 \cup a_0 \colon \Sigma_0(1) \to \bZ$ is the map sending all the cones in $\Sigma_0(1)$ to $1$.
We obtain $\psi(a_0 \cup a_0)=4$.

On the right of \pref{fg:surface}, the blue lines show the support of the deformation $c(a_0)'$ of $c(a_0)$ that we considered in \pref{sc:main1-2}.
The tropical cycles $c(a_0)$ and $c(a_0)'$ intersect at four green points.
Each intersection point corresponds to one of \eqref{eq:four-1}-\eqref{eq:four-4}.
For instance, \eqref{eq:four-1} corresponds to the intersection point of $\Delta_{\scS_1}$ and (the deformation of) $\Delta_{\scS_3}$ shown on the left of \pref{fg:surface}.
The integral volume form on the facet of $\nabla_{w=0}$ dual to $\cone \lb \lc e_0 \rc \rb$, which defines the orientation under which the simplices 
$\Delta_{\scS_1}$ and $\Delta_{\scS_3}$ intersect positively is 
$e_1 \wedge e_3 \in \bigwedge^2 (\left. M \middle/ \bZ e_0\right.)$.
Therefore, the tropical multiplicity of the intersection point is 
\begin{align}
\la e_1 \wedge e_3, 
a_0\lb \sigma [\scS_1]\rb \cdot f(\scS_1) \wedge a_0\lb \sigma [\scS_3]\rb \cdot  f(\scS_3) \ra
=\la e_1 \wedge e_3, (-e_2^\ast+e_3^\ast) \wedge (-e_1^\ast+e_2^\ast) \ra
=1.
\end{align}
One can similarly check that the tropical multiplicities of the other intersection points are also $1$.
Thus we have $\la c(a_0), c(a_0)\ra=4=\psi(a_0 \cup a_0)$.
\end{example}

\begin{example}
Let $d=2, q=1$, and take a basis $\lc e_1, e_2, e_3 \rc$ of $M \cong \bZ^3$ again.
We consider the polynomial \eqref{eq:f1} with 
\begin{align}
\Delta:=\conv \lb \lc 2e_1, -e_1, e_2, -e_2, e_3, -e_3 \rc \rb \subset M_\bR, \quad
\lambda_m:=
\left\{ \begin{array}{ll}
3 & m=2e_1 \\
0 & m=0 \\
1 & m \in \lb \Delta \cap M \rb \setminus \lc 0, 2e_1 \rc.
\end{array} 
\right. 
\end{align}
The tropical hypersurface $X(\trop (f)) \subset N_\bR \cong \bR^3$ obtained by tropicalization is shown by black lines in \pref{fg:surface2}.
\begin{figure}[htbp]
	\centering
	\includegraphics[scale=0.475]{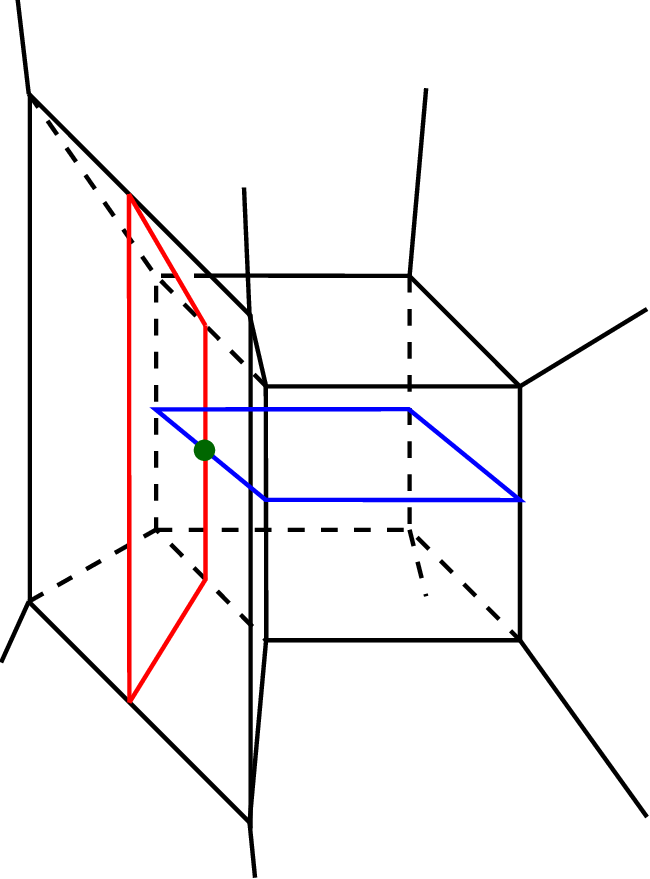}
    \caption{The tropical cycles $c(a_1), c(a_2)$}
\label{fg:surface2}
\end{figure}
We have $W=\lc e_1, 0 \rc$.
Let $a_1 \in \MW^1(\Sigma_{e_1})$ and $a_2 \in \MW^1(\Sigma_{0})$ be the Minkowski weights defined by
\begin{align}
a_1 &\colon \Sigma_{e_1}(2) \to \bZ, \quad \sigma \mapsto 
\left\{ \begin{array}{ll}
1 & \sigma \subset \bR e_1 \oplus \bR e_3 \\
0 & \mathrm{otherwise}\\
\end{array} 
\right.\\
a_2 &\colon \Sigma_0(2) \to \bZ, \quad \sigma \mapsto 
\left\{ \begin{array}{ll}
1 & \sigma \subset \bR e_1 \oplus \bR e_2 \\
0 & \mathrm{otherwise}.
\end{array} 
\right.
\end{align}
The red and blue loops in \pref{fg:surface2} are the tropical cycles $c(a_1), c(a_2)$ associated with these Minkowski weights (the unions of simplices in $c(a_1), c(a_2)$ whose coefficients are non-zero).

We compute the tropical intersection number $\la c(a_1), c(a_2)\ra$.
Taking $m_0:=\varepsilon \cdot (3e_1+2e_2+e_3) \in M_\bR$ as a generic vector, we deform the tropical cycle $c(a_2)$ as we did in \pref{sc:main1-2} and \pref{sc:main1-3}.
Then \eqref{eq:intersection-pts-2} consists of two pairs of cones in $\Sigma_0(2)$
\begin{align}\label{eq:pair1}
    &\lb \cone \lb \lc e_1, e_2 \rc \rb, \cone \lb \lc e_1, -e_3 \rc\rb \rb\\ \label{eq:pair2}
    &\lb \cone \lb \lc e_1, e_3 \rc \rb, \cone \lb \lc e_1, -e_2 \rc\rb \rb.
\end{align}
The former pair \eqref{eq:pair1} corresponds to the intersection point of $\Delta_{\scS_1}$ and (the deformation of) $\Delta_{\scS_2}$ with
\begin{align}
    \scS_1&:=\lc \cone \lb \lc -e_1 \rc \rb \prec \cone \lb \lc -e_1, -e_1+e_2 \rc \rb \rc  \in \scrS_{q=1}^{w=e_1}\\
    \scS_2&:=\lc \cone \lb \lc e_1 \rc \rb \prec \cone \lb \lc e_1, -e_3 \rc \rb \rc  \in \scrS_{q=1}^{w=0}.
\end{align}
We have $a_1\lb \sigma [\scS_1]\rb=a_2\lb \sigma [\scS_2]\rb=0$, and the tropical multiplicity of the intersection point is zero.
The latter pair \eqref{eq:pair2} corresponds to the intersection point of $\Delta_{\scS_1'}$ and (the deformation of) $\Delta_{\scS_2'}$ with
\begin{align}
    \scS_1'&:=\lc \cone \lb \lc -e_1 \rc \rb \prec \cone \lb \lc -e_1, -e_1+e_3 \rc \rb \rc  \in \scrS_{q=1}^{w=e_1}\\
    \scS_2'&:=\lc \cone \lb \lc e_1 \rc \rb \prec \cone \lb \lc e_1, -e_2 \rc \rb \rc  \in \scrS_{q=1}^{w=0}.
\end{align}
We have $a_1\lb \sigma [\scS_1']\rb=a_2\lb \sigma [\scS_2']\rb=1$, and 
\begin{align}
        f(\scS_1')
    &=\la -e_1 \wedge (-e_1+e_3) \wedge \bullet, \vol (N) \ra
    =e_2^\ast\\
    f(\scS_2')
    &=\la e_1 \wedge (-e_2) \wedge \bullet, \vol (N) \ra
    =-e_3^\ast
\end{align}
when setting $\vol(N):=e_1^\ast \wedge e_2^\ast \wedge e_3^\ast \in \bigwedge^3 N$.
The integral volume form on $\partial \nabla_{w=e_1} \cap \partial \nabla_{w=0}$, which defines the orientation under which the simplices 
$\Delta_{\scS_1'}$ and $\Delta_{\scS_2'}$ intersect positively is 
$e_2 \wedge e_3 \in \bigwedge^2 (\left. M \middle/ \bZ e_1\right.)$.
Therefore, the tropical multiplicity of the intersection point is 
\begin{align}
\la e_2 \wedge e_3, 
a_1\lb \sigma [\scS_1']\rb \cdot f(\scS_1') \wedge a_2\lb \sigma [\scS_2']\rb \cdot  f(\scS_2') \ra=-1.
\end{align}
Thus we have $\la c(a_1), c(a_2)\ra=-1$.

The cup product $\overline{a}_1 \cup \overline{a}_2$ can also be computed as follows:
The fan $\Sigma_{e_1, 0} \subset \left. M_\bR \middle/\bR e_1 \right.$ 
(the fan $\Sigma_{w_1, w_2}$ in \pref{th:main1}(2-b) with $w_1=e_1, w_2=0$)
is the fan of $\bP^1 \times \bP^1$, and \eqref{eq:intersection-pts-3} consists of two pairs of cones in $\Sigma_{e_1, 0}(1)$
\begin{align}\label{eq:pair3}
    &\lb \cone \lb \lc e_2 \rc \rb, \cone \lb \lc -e_3 \rc\rb \rb\\ \label{eq:pair4}
    &\lb \cone \lb \lc e_3 \rc \rb, \cone \lb \lc -e_2 \rc\rb \rb,
\end{align}
which respectively correspond to \eqref{eq:pair1} and \eqref{eq:pair2} via the projection $M_\bR \to \left. M_\bR  \middle/ \bR e_1 \right.$.
The Minkowski weights $\overline{a}_1, \overline{a}_2 \in \MW^1(\Sigma_{e_1, 0})$ are given by
\begin{align}
\overline{a}_1 &\colon \Sigma_{e_1, 0}(1) \to \bZ, \quad \sigma \mapsto 
\left\{ \begin{array}{ll}
1 & \sigma \subset \left. (\bR e_1 \oplus \bR e_3) \middle/ \bR e_1 \right. \\
0 & \mathrm{otherwise}
\end{array} 
\right.\\
\overline{a}_2 &\colon \Sigma_{e_1, 0}(1) \to \bZ, \quad \sigma \mapsto 
\left\{ \begin{array}{ll}
1 & \sigma \subset \left. (\bR e_1 \oplus \bR e_2)\middle/ \bR e_1 \right. \\
0 & \mathrm{otherwise}.
\end{array} 
\right.
\end{align}
We can see that $\overline{a}_1(\sigma) \cdot \overline{a}_2(\sigma') \cdot
\ld M : \bZ(M \cap \sigma)+\bZ(M \cap \sigma') \rd$ is $0$ for \eqref{eq:pair3} and is $1$ for \eqref{eq:pair4}.
Thus we have $\overline{a}_1 \cup \overline{a}_2=1=-\la c(a_1), c(a_2)\ra$.
\end{example}

\section{Proof of \pref{th:main2}}\label{sc:proof2}

We will construct the $d$-cycle $C(\scE)$ of \pref{th:main2}(2) in \pref{sc:step1}--\pref{sc:step1-3}.
The small perturbation of \pref{th:main2}(2-c) will be constructed in \pref{sc:step2}.
In \pref{sc:homology}, we will discuss the homology class of the lift $C(\scE)_t \subset \mathring{Z}_t$.
The proof of \pref{th:main2} is completed in \pref{sc:proof-main2}.

\subsection{Constructing cycles from lattice polytopes}\label{sc:step1}

Let $P \subset N_\bR$ be a lattice polytope such that the fan $\Sigma_w$ is a refinement of the normal fan $\Sigma_P$ of $P$, and $\varphi_P \colon M_\bR \to \bR$ be the support function for $P$, which is defined by
\begin{align}\label{eq:varphi}
    \varphi_P(m):=\inf_{n \in  P} \la m, n \ra.
\end{align}
Let further
\begin{align}\label{eq:mu}
\mu \colon \Sigma_w \to \Sigma_P
\end{align}
be the map that sends each cone 
$\sigma \in \Sigma_w$ to the minimal cone in $\Sigma_P$ containing $\sigma$.
For a sequence of cones $\scS=\lc \sigma_1 \prec \cdots \prec \sigma_{q+1}\rc \in \scrS_q^w$, let $P_{\scS}$ denote the face of $P$ corresponding to the cone $\mu \lb \sigma[\scS] \rb=\mu(\sigma_{q+1})$, i.e.,
\begin{align}\label{eq:face-S}
   P_{\scS}:=\lc n \in P \relmid \varphi_P(m)=\la m, n \ra, \forall m \in \mu (\sigma[\scS]) \rc.
\end{align}

Suppose that $\dim \mu(\sigma[\scS])=\dim \sigma[\scS]=q+1$.
This happens if and only if $\dim P_{\scS}=d-q$.
Let $e_i \in M$ $\lb 1 \leq i \leq q+1\rb$ be the elements such that 
$\lc e_i \relmid 1 \leq i \leq j \rc$ are the primitive generators of 
$\sigma_j$ for all $1 \leq j \leq q+1$.
The map \eqref{eq:fd-q} defines an element of 
$\bigwedge^{d-q} (\left. M_\bR \middle/ (\bR \cdot \sigma[\scS] \right.))^\ast
=\bigwedge^{d-q} (\left. M_\bR \middle/ (\bR \cdot \mu (\sigma[\scS] )\right.))^\ast$, which determines an orientation of 
$(\bR \cdot \mu(\sigma[\scS]))^\perp = T(P_{\scS})$, where $T(P_{\scS}) \subset N_\bR$ denotes the subspace generated by vectors tangent to $P_{\scS}$.
We give the orientation induced by this to the face $P_{\scS}$.
We also take a triangulation of $P$, and consider the induced triangulation for every face $P_{\scS} \prec P$.
We regard each triangle as a singular simplex, and define the singular $d$-chain $c(P)$ in $N_\bR \times N_\bR$ by
\begin{align}\label{eq:c(P)}
    c(P):=\sum_{q=0}^d \sum_{\substack{\scS \in \scrS_q^w \\ \dim P_\scS=d-q}} (-1)^q \cdot \Delta_\scS \times P_{\scS},
\end{align}
where we regard \eqref{eq:c(P)} as a singular $d$-chain by taking a triangulation of the product of $\Delta_\scS$ and every simplex in the triangulation of $P_{\scS}$, which we took above in advance.
We refer the reader, for instance, to \cite[Section B.6]{MR4406774} for how to take a triangulation of the product of simplices.

\begin{lemma}\label{lm:b-delta}
Suppose that $\dim P_{\scS}=d-q$ for $\scS=\lc \sigma_1 \prec \cdots \prec \sigma_{q+1}\rc \in \scrS_q^w$.
Then one has
    \begin{align}\label{eq:b-delta}
        \partial P_{\scS}
        =-
        \sum_{\substack{\sigma \in \Sigma_w(q+2)\\ \sigma \succ \sigma[\scS] \\ \dim P_{\scS(\sigma)=d-q-1}}} P_{\scS(\sigma)},
    \end{align}
    where $\scS(\sigma)$ is defined by $\scS(\sigma):=\lc \sigma_1 \prec \cdots \prec \sigma_{q+1} \prec \sigma \rc \in \scrS_{q+1}^w$ for $\scS=\lc \sigma_1 \prec \cdots \prec \sigma_{q+1}\rc \in \scrS_q^w$ and $\sigma \in \Sigma_w(q+2)$ such that $\sigma \succ \sigma_{q+1}$.
\end{lemma}
\begin{proof}
First, we show that the boundary of the polytope $P_{\scS}$ is the union of polytopes $P_{\scS(\sigma)}$ appearing in the right-hand side of \eqref{eq:b-delta}.
When $\dim P_{\scS}=d-q$, the dimension of $\mu(\sigma[\scS])$ is $q+1$.
Each facet $F$ of the polytope $P_{\scS}$ corresponds to a cone $\sigma_F \in \Sigma_P(q+2)$ such that $\sigma_F \succ \mu(\sigma[\scS])$.
Therefore, it suffices to show that for such a cone $\sigma_F$, there uniquely exists a cone $\sigma \in \Sigma_w(q+2)$ such that $\sigma \succ \sigma[\scS]$, and $\mu(\sigma)=\sigma_F$ which implies $P_{\scS(\sigma)}=F$.

First, we prove the existence of such a cone.
Since $\Sigma_w$ is a refinement of $\Sigma_P$, 
the cone $\sigma_F$ coincides with the union of all the cones $\sigma \in \Sigma_w(q+2)$ contained in the cone $\sigma_F$.
Therefore, there exists a cone $\sigma \in \Sigma_w(q+2)$ which is contained in the cone $\sigma_F$ and contains a point in $\rint(\sigma[\scS])$.
The cone $\sigma \in \Sigma_w(q+2)$ satisfies $\sigma \succ \sigma[\scS]$ and $\mu(\sigma)=\sigma_F$.

Next, we prove the uniqueness.
Suppose there are two such different cones $\sigma_1, \sigma_2 \in \Sigma_w(q+2)$.
We have $\sigma_1 \cap  \sigma_2=\sigma[\scS]$.
Let $e_i \in M$ $\lb 1 \leq i \leq q+1\rb$ be the primitive generators of 
$\sigma[\scS]$, and $e_{q+2}$ be the primitive generator of $\sigma_1$ not contained in $\sigma[\scS]$.
Then the primitive generator $e_{q+2}$ of $\sigma_2$ not contained in $\sigma[\scS]$ is written as
\begin{align}\label{eq:e'}
    e_{q+2}'=\sum_{i=1}^{q+2}c_i \cdot e_i
\end{align}
with integers $c_i$ $(1 \leq i \leq q+2)$.
Since $\mu(\sigma[\scS]) \prec \sigma_F$, there exists an element $n_0 \in N_\bR$ which equals $0$ on $\mu(\sigma[\scS])$ and is positive on $\sigma_F \setminus \mu(\sigma[\scS])$.
We have $\la e_i, n_0\ra=0$ for $1 \leq i \leq q+1$.
Since $e_{q+2}, e_{q+2}' \in \sigma_F \setminus \mu(\sigma[\scS])$, we also have $\la e_{q+2}, n_0\ra>0$ and $\la e_{q+2}', n_0\ra>0$.
These and \eqref{eq:e'} imply $c_{q+2}>0$.
Let $\varepsilon>0$ be a small real number, and consider
\begin{align}\label{eq:v0}
    v:=\varepsilon \cdot e_{q+2}'+\sum_{i=1}^{q+1} e_i \in \rint(\sigma_2).
\end{align}
By \eqref{eq:e'}, we obtain
\begin{align}\label{eq:v}
    v=\varepsilon \cdot c_{q+2} \cdot e_{q+2}+\sum_{i=1}^{q+1} (1+\varepsilon \cdot c_i) \cdot e_i.
\end{align}
When $\varepsilon>0$ is sufficiently small, we have $1+\varepsilon \cdot c_i >0$ for all $i \in \lc 1, \cdots, q+1\rc$, and \eqref{eq:v} is in $\rint(\sigma_1)$.
We obtain $v \in \rint (\sigma_1) \cap \rint (\sigma_2)$ which contradicts $\sigma_1 \cap  \sigma_2=\sigma[\scS]$.
Thus we conclude that for $\sigma_F$, there uniquely exists a cone $\sigma \in \Sigma_w(q+2)$ such that $\sigma \succ \sigma[\scS]$ and $P_{\scS(\sigma)}=F$, and the boundary of the polytope $P_{\scS}$ is the union of polytopes $P_{\scS(\sigma)}$ appearing in the right-hand side of \eqref{eq:b-delta}.

Lastly, we discuss the orientations of the polytopes in $\partial P_{\scS}$.
Let $e_i \in M$ $\lb 1 \leq i \leq q+2\rb$ be the elements such that 
$\lc e_i \relmid 1 \leq i \leq q+1 \rc$ are the primitive generators of 
$\sigma_{q+1}$, and $\lc e_i \relmid 1 \leq i \leq q+2 \rc$ are the primitive generators of $\sigma (\succ \sigma_{q+1})$.
Then for the face $P_{\scS(\sigma)} \prec P_{\scS}$, the element $e_{q+2}$ evaluates an outgoing normal vector negatively. 
From this and the definition of the orientations of $P_{\scS}$, we can see that the right-hand side of \eqref{eq:b-delta} should have the minus sign. 
We conclude \eqref{eq:b-delta}.
\end{proof}

\begin{lemma}\label{lm:c(P)-boundary}
    One has $\partial c(P)=0$.
\end{lemma}
\begin{proof}
One has
\begin{align}\label{eq:boundary1}
\partial c(P)
=
\sum_{q=0}^d \sum_{\substack{\scS \in \scrS_q^w \\ \dim P_{\scS}=d-q}}
\lc
\lb \sum_{i=1}^q (-1)^{q+i-1} \cdot \partial_i \Delta_\scS \times P_{\scS} \rb
+
\partial_{q+1} \Delta_\scS \times P_{\scS}
+
\Delta_\scS \times \partial P_{\scS}
\rc,
\end{align}
where $\partial_i \Delta_\scS$ denotes the facet of $\Delta_\scS$, which does not have the $i$-th vertex of $\Delta_\scS$ ($1 \leq i \leq q+1$).
Regarding the first part of \eqref{eq:boundary1}, if we consider $\scS[i] \in \scrS_q^w$ which is defined similarly to \eqref{eq:S[i]}, then one has $\partial_i \Delta_{\scS[i]}=\partial_i \Delta_\scS$ and $P_{\scS}=-P_{\scS[i]}$.
Therefore, we have
\begin{align}\label{eq:b-part1}
\sum_{\substack{\scS \in \scrS_q^w \\ \dim P_{\scS}=d-q}}
\sum_{i=1}^q (-1)^{q+i-1} \cdot \partial_i \Delta_\scS \times P_{\scS}=0,
\end{align}
since the terms of $(\scS, i)$ and $(\scS[i], i)$ cancel each other out.
The second part of \eqref{eq:boundary1} is
\begin{align}\label{eq:b-part2}
\sum_{q=0}^d \sum_{\substack{\scS \in \scrS_q^w \\ \dim P_{\scS}=d-q}}
\partial_{q+1} \Delta_\scS \times P_{\scS}
=
\sum_{q=0}^{d-1} \sum_{\scS \in \scrS_q^w} 
\Delta_\scS \times \lb 
\sum_{\substack{\sigma \in \Sigma_w (q+2)\\ \sigma \succ \sigma[\scS] \\ 
\dim P_{\scS(\sigma)}=d-q-1}} 
P_{\scS(\sigma)} \rb.
\end{align}
By \eqref{eq:b-part1}, \eqref{eq:b-part2}, and \pref{lm:b-delta}, we can see that \eqref{eq:boundary1} equals
\begin{align}\label{eq:boundary2}
\sum_{q=0}^{d-1} 
\sum_{\substack{\scS \in \scrS_q^w \\ \dim P_{\scS} \neq d-q}} 
\Delta_\scS \times 
\lb 
\sum_{\substack{\sigma \in \Sigma_w (q+2)\\ \sigma \succ \sigma[\scS] \\ 
\dim P_{\scS(\sigma)}=d-q-1}} 
P_{\scS(\sigma)} \rb.
\end{align}
(Notice that for $\scS \in \scrS_{q=d}^w$, the lattice polytope $P_{\scS}$ is a point and we have $\partial P_{\scS}=0$.)
We will show that all the terms in \eqref{eq:boundary2} cancel out and \eqref{eq:boundary2} becomes zero.

Let $\scS=\lc \sigma_1 \prec \cdots \prec \sigma_{q+1} \rc \in \scrS_{q}^w$ 
and $\sigma \in \Sigma_w(q+2)$ be elements such that 
$\sigma \succ \sigma[\scS]=\sigma_{q+1}$, $\dim P_{\scS} \neq d-q$, 
and $\dim P_{\scS(\sigma)}=d-q-1$ as in \eqref{eq:boundary2}.
Let further $e_i \in M$ $\lb 1 \leq i \leq q+2\rb$ be the vectors such that 
$\lc e_i \relmid 1 \leq i \leq j \rc$ are the primitive generators of 
$\sigma_j$ for all $1 \leq j \leq q+1$, 
and $\lc e_i \relmid 1 \leq i \leq q+2 \rc$ are the primitive generators of 
$\sigma$.
Since $d-q > \dim P_{\scS} \geq \dim P_{\scS(\sigma)}=d-q-1$, we have 
$P_{\scS} =P_{\scS(\sigma)}$.
We take an element $n_0 \in N_\bR$ such that 
the function $\varphi_{P}+n_0 \colon M_\bR \to \bR$ is constantly equal to $0$ on the cone $\sigma \subset M_\bR$.
The function $\varphi_{P}+n_0$ induces a convex piecewise linear function
\begin{align}\label{eq:star}
\Star (\sigma_{q+1})  \to \bR,
\end{align}
where $\Star (\sigma_{q+1}) \subset \left. M_\bR \middle/ (\bR \sigma_{q+1}) \right.$ is the star fan of $\sigma_{q+1} (\prec \sigma)$ (cf.~\cite[Section 3.2, (3.2.8)]{MR2810322}).
The polytope whose support function is \eqref{eq:star} is the translation 
(by the vector $n_0$) of $P_{\scS} =P_{\scS(\sigma)}$, and is a full-dimensional polytope in $(\bR \sigma)^\perp$.

We consider the polytope $P_{\scS}$ as being in 
$(\bR \sigma_{q+1})^\perp (\supset (\bR \sigma)^\perp)$, 
and take its normal fan $\Sigma_{\scS}$, which is in $\left. M_\bR \middle/ \bR \sigma_{q+1} \right.$.
The fan $\Sigma_{\scS}$ contains the $1$-dimensional linear subspace spanned by $e_{q+2}$ (\cite[Proposition 6.2.3(a)]{MR2810322}), and the fan 
$\Star(\sigma_{q+1})$ is a refinement of $\Sigma_{\scS}$ (\cite[Proposition 6.2.5(b)]{MR2810322}).
Since the fan $\Star(\sigma_{q+1})$ consists only of strictly-convex cones, we can see that the fan $\Star(\sigma_{q+1})$ contains two $1$-dimensional cones, one is generated by $e_{q+2}$ and the other is generated by $-e_{q+2}$, which will be denoted by $\rho_+$ and $\rho_-$ respectively.
(The cone $\rho_+$ corresponds to $\sigma$.)
This implies that there uniquely exists a cone $\sigma' \in \Sigma_w(q+2)$ (corresponding to the cone $\rho_-$) such that $\sigma' \succ \sigma_{q+1}$ and whose primitive generators are $e_i$ $(1 \leq i \leq q+1)$ and
\begin{align}\label{eq:eq'}
    e_{q+2}':=-e_{q+2}+\sum_{i=1}^{q+1} c_i \cdot e_i
\end{align}
written with some $c_i \in \bZ$.
The polytope $P_{\scS(\sigma)}$ (resp. $P_{\scS(\sigma')}$) is the face of the polytope $P_{\scS} \subset (\bR \sigma_{q+1})^\perp$ corresponding to the minimal cone in $\Sigma_\scS$ containing $\rho_+$ (resp. $\rho_-$), which is the $1$-dimensional linear subspace spanned by $e_{q+2}$.
Hence, this face is the whole polytope $P_{\scS}$, and we have $P_{\scS(\sigma)}=P_{\scS(\sigma')}=P_{\scS}$ as sets.
On the other hand, regarding the map \eqref{eq:fd-q} that we used for defining the orientations of polytopes $P_{\scS}$, we have
\begin{align}
    \bigwedge_{i=1}^{q+2} e_i=-\bigwedge_{i=1}^{q+1} e_i \wedge e_{q+2}'
\end{align}
by \eqref{eq:eq'}.
From this, we can see that the orientations of $P_{\scS(\sigma)}$ and $P_{\scS(\sigma')}$ are opposite.
Therefore, in \eqref{eq:boundary2}, the terms for $(\scS, \sigma)$ and $(\scS, \sigma')$ cancel each other out.
Thus \eqref{eq:boundary2} becomes zero.
We obtained $\partial c(P)=0$.
\end{proof}

We will use the following general fact in toric geometry:

\begin{proposition}\label{pr:iritani}
For a smooth projective toric variety $Y$ over $\bC$, the K-group $K(Y)$ is additively generated by anti-ample line bundles on $Y$. 
The same statement also holds when anti-ample line bundles are replaced by ample ones.
\end{proposition}
\begin{proof}
We prove the former claim.
It is well-known that the K-group $K(Y)$ is additively generated by line bundles (cf.~e.g.~\cite[Proposition 3]{{MR1234308}}).
Therefore, it suffices to show that an arbitrary line bundle on $Y$ is written as a linear combination over $\bZ$ of anti-ample line bundles on $Y$.
Let $L$ be an arbitrary line bundle on $Y$, and take an anti-ample line bundle $M$ on $Y$.
Since $\lb 1- \ld M\rd \rb^{k}=0$ for sufficiently large $k \in \bZ_{>0}$ (cf.~\cite[Corollary 3.1.6]{MR224083}), we have
\begin{align}\label{eq:m-1}
    \ld M \rd^{-1} = \lc 1-\lb 1-\ld M \rd \rb \rc^{-1} = \sum_{i=0}^{k-1} \lb 1-\ld M \rd \rb^i.
\end{align}
For sufficiently large $n \in \bZ_{>0}$, the line bundle $N:=L \otimes M^n$ is anti-ample.
We have
\begin{align}\label{eq:Lnm}
    \ld L \rd = \ld N \rd \cdot \ld M \rd^{-n} 
    = \ld N \rd \cdot \lb \sum_{i=0}^{k-1} \lb 1-\ld M \rd \rb^i \rb^n
\end{align}
by \eqref{eq:m-1}, and this is a linear combination of $\ld  N \otimes M^j \rd$ with $j \in \bZ_{\geq 0}$.
Since $N \otimes M^j$ are anti-ample, we can conclude the former claim of \pref{pr:iritani}.
The latter claim also follows by the same argument, with anti-ample line bundles replaced by ample ones.
\end{proof}

Since the fan $\Sigma_w$ is unimodular, the associated toric variety $Y_w$ is smooth.
Furthermore, from the assumption that the function $\lambda \colon A \to \bQ$ of \eqref{eq:lambda} extends to a strictly-convex piecewise affine function on the triangulation $\scrT$ of the polytope $\Delta$, we can see that the function
\begin{align}
    (A_w-w) \cup \lc 0 \rc \to \bQ, \quad m \mapsto \lambda(m+w)-\lambda(w)
\end{align}
extends to a strictly-convex piecewise linear function on the fan $\Sigma_w$, which gives rise to an ample line bundle on $Y_w$ (cf.~e.g.~\cite[Theorem 6.1.14]{MR2810322}).
Therefore, $Y_w$ is a smooth projective toric variety.

By \pref{pr:iritani}, we can see, in particular, that the K-group $K(Y_w)$ of the toric variety $Y_w$ is generated by anti-nef line bundles on $Y_w$.
Let $\scE \in K(Y_w)$, and we express it as
\begin{align}\label{eq:scE}
    \scE=\sum_{j \in J} l_j \cdot [L_j]
\end{align}
with a finite index set $J$, integers $l_j$, and anti-nef line bundle $L_j$ on $Y_w$\footnote{By \pref{pr:iritani}, we can also decompose $\scE \in K(Y_w)$ as the sum of anti-ample line bundles. However, we use anti-nef line bundles in \eqref{eq:scE}, as the anti-nef assumption is sufficient for our construction of cycles.}.
For each $j \in J$, we choose a lattice polytope $P_j \subset N_\bR$ whose associated support function defines the nef line bundle $L_j^{-1}$; such a polytope is unique up to translation.
For any $j \in J$, the fan $\Sigma_w$ is a refinement of the normal fan $\Sigma_{P_j}$ of $P_j$ (\cite[Proposition 6.2.5(b)]{MR2810322}).
We define
\begin{align}\label{eq:c(scE)}
    c(\scE)
    :=\sum_{j \in J} l_j \cdot c(P_j)
    =
    \sum_{j \in J} l_j \cdot \lb \sum_{q=0}^d \sum_{\substack{\scS \in \scrS_q^w \\ \dim P_{j, \scS}=d-q}} (-1)^q \cdot
    \Delta_\scS \times P_{j, \scS} \rb,
\end{align}
where $c(P_j)$ is the the singular $d$-chain $c(P)$ of \eqref{eq:c(P)} for $P=P_j$, and $P_{j, \scS}$ is the face \eqref{eq:face-S} for $P=P_j$.
We regard $c(\scE)$ \eqref{eq:c(scE)} as a singular chain in 
$N_\bR \times (\left. N_\bR \middle/ N \right.)$ 
by the natural projection $N_\bR \times N_\bR \to N_\bR \times \lb \left. N_\bR \middle/ N \right. \rb$.
This singular chain $c(\scE)$ does not depend on the choices of the lattice polytopes $P_j$.
Notice, however, that it depends on the expression \eqref{eq:scE} of $\scE$, and is not uniquely determined just by the element $\scE \in K(Y_w)$.
By \pref{lm:c(P)-boundary}, we have $\partial c(\scE)=0$.

\begin{example}\label{eg:k3-2}
    In \pref{eg:k3}, consider
    \begin{align}\label{eq:eg-E}
        \scE:=\ld \scO_{\bP^3} \rd-\ld \scO_{\bP^3}(-1) \rd \in K(Y_{w=0}=\bP^3).
    \end{align}
We have
\begin{align}
    \ch(\scE)=1- \lb 1-H+\frac{1}{2}H^2\rb=H-\frac{1}{2}H^2,
\end{align}
where $H$ denotes the hyperplane class.
The integer $k$ of \eqref{eq:k} for $\scE$ is $1$, and the Minkowski weight corresponding to $(-1)^{d+1-k} \cdot \ch_{k}(\scE)=\ch_{1}(\scE)=H$ is the generator $a_0 \colon \Sigma_{w=0}(2) \to \bZ$ of $\MW^1(\Sigma_{w=0})$, which sends all cones in $\Sigma_{0}(2)$ to $1$.
The tropical cycle $c(a_0)$ associated with $a_0$ is shown by red lines on the right of \pref{fg:surface}.
It is supported in $\partial \nabla_{w=0}^{d-k=1}$.

    For \eqref{eq:eg-E}, we can take 
    \begin{align}
        P_1&:=\lc 0 \rc \subset N_\bR\\
        P_2&:=\conv \lb \lc 0, e_1^\ast, e_2^\ast, e_3^\ast \rc \rb \subset N_\bR
    \end{align}
    as lattice polytopes corresponding to the nef line bundles $\scO_{\bP^3}^{-1}=\scO_{\bP^3}$ and $\scO_{\bP^3}(-1)^{-1}=\scO_{\bP^3}(1)$ respectively.
 We have
    \begin{align}
        c(P_1)&=\sum_{\scS \in \scrS_{q=2}^{w=0}} \Delta_\scS \times \lc 0 \rc \\
c(P_2)&=\sum_{\scS \in \scrS_{q=2}^{w=0}} \Delta_\scS \times P_{2, \scS}
-\sum_{\scS \in \scrS_{q=1}^{w=0}} \Delta_\scS \times P_{2, \scS}
+\sum_{\scS \in \scrS_{q=0}^{w=0}} \Delta_\scS \times P_{2, \scS},
    \end{align}
and the polytope $P_{2, \scS}$ with $\scS \in \scrS_{q=2}^{w=0}$ is a point in $N$.
Therefore, when regarding $c(P_1)$ and $c(P_2)$ as singular chains in $N_\bR \times \lb \left. N_\bR \middle/ N \right. \rb$, we have
\begin{align}\label{eq:cp1cp2}
    c(\scE):=c(P_1)-c(P_2)=\sum_{\scS \in \scrS_{q=1}^{w=0}} \Delta_\scS \times P_{2, \scS}
-\sum_{\scS \in \scrS_{q=0}^{w=0}} \Delta_\scS \times P_{2, \scS}.
\end{align}
We can see that $c(\scE)$ is supported over $\partial \nabla_{w=0}^{d-k=1}$\footnote{This does not happen in general. In \pref{sc:step1-2}, we deform $c(\scE)$ to be supported in $\partial \nabla_{w}^{d-k}$ by adding boundaries of singular chains $\Delta_{\scS} \times c_\scS$ (cf.~\eqref{eq:t-c(scE)}, \pref{lm:cscds}). In the case of \pref{eg:k3-2}, the singular chains $c_\scS$ of \pref{lm:cscds} are all zero.}.
The lattice polytopes $P_{2, \scS_i}$ for 
$\scS_i \in \scrS_{q=1}^{w=0}$ of \eqref{eq:S_i} and
$\scS_0:=\cone \lb \lc e_0 \rc\rb \in \scrS_{q=0}^{w=0}$ are
\begin{align}
    P_{2, \scS_i}=
    \left\{ \begin{array}{ll}
\conv \lb \lc e_2^\ast, e_3^\ast \rc\rb & i=1 \\
\conv \lb \lc e_1^\ast, e_3^\ast \rc\rb & i=2 \\
\conv \lb \lc e_1^\ast, e_2^\ast \rc\rb & i=3 \\
\conv \lb \lc e_1^\ast, e_2^\ast, e_3^\ast \rc\rb & i=0.
\end{array} 
\right. 
\end{align}
The polytopes $P_{2, \scS_i}$ $(i=1, 2, 3)$ are arranged over the line segments $\Delta_{\scS_i}$ shown on the left of \pref{fg:surface}, and the polytope $P_{2, \scS_0}$ is placed at the intersection point $\bigcap_{i=1, 2, 3}\Delta_{\scS_i}=\Delta_{\scS_0}$.
These form part of $c(\scE)$ of \eqref{eq:cp1cp2}.
As stated in \pref{lm:b-delta}, the polytopes $P_{2, \scS_i}$ are oriented so that we have
\begin{align}
    \partial P_{2, \scS_0}=-P_{2, \scS_1}-P_{2, \scS_2}-P_{2, \scS_3}.
\end{align}
\end{example}

\subsection{Deforming cycles by adding boundaries of singular chains}\label{sc:step1-2}

Let $\scS \in \scrS_q^w$ ($q \in \lc 0, \cdots, d\rc$), and
$V(\sigma[\scS]) \subset Y_w$ be the closure of the torus orbit corresponding to the cone $\sigma[\scS] \in \Sigma_w (q+1)$.
By applying \cite[Theorem 13.4.1(b)]{MR2810322} for the restriction of the nef line bundle $L_j^{-1}$ to $V(\sigma[\scS])$, we obtain
\begin{align}
\vol \lb P_{j, \scS} \rb
    =
\deg \lb \ch_{d-q}\lb L_j^{-1} \rb \cap \ld V(\sigma[\scS])\rd \rb,
\end{align}
where $\vol \lb P_{j, \scS} \rb$ denotes the Euclidean volume of the lattice polytope $P_{j, \scS}$.
This implies
\begin{align}\label{eq:vol-sum0}
    \sum_{j \in J} l_j \cdot \vol \lb P_{j, \scS} \rb
    =
-\deg \lb \ch_{d-q}(\scE) \cap \ld V(\sigma[\scS])\rd \rb.
\end{align}
Let $k$ be the integer \eqref{eq:k} for $\scE \in K(Y_w)$ of \eqref{eq:scE}.
By the definition of the integer $k$, we have $\ch_{k'}(\scE)=0$ for all $k' <k$.
By this and \eqref{eq:vol-sum0}, we have
\begin{align}\label{eq:vol-sum}
\sum_{j \in J} l_j \cdot \vol \lb P_{j, \scS} \rb
=
0
\end{align}
for $\scS \in \scrS_q^w$ with $q \in \lc d+1-k, \cdots, d \rc$.

\begin{lemma}\label{lm:cscds}
There exists a family of singular chains 
$\lc c_{\scS} \relmid \scS \in \scrS_q^w, d+1-k \leq q \leq d\rc$ 
in the torus $\left. N_\bR \middle/ N \right.$ such that 
\begin{enumerate}
    \item for $\scS \in \scrS_q^w$ $(d+1-k \leq q \leq d)$, $c_{\scS}$ is a singular $(d+1-q)$-chain whose image is contained in the subtorus $\left. \sigma[\scS]^\perp \middle/ \lb \sigma[\scS]^\perp\cap N \rb \right. \subset \left. N_\bR \middle/ N \right.$, and
\item one has
\begin{align}\label{eq:cds1}
\partial c_{\scS}
&=\sum_{\substack{\sigma \in \Sigma_w(q+2) \\ \sigma \succ \sigma[\scS]}} c_{\scS(\sigma)}
+
\sum_{\substack{j \in J \\ \dim P_{j, \scS}=d-q}}
l_j \cdot P_{j, \scS} \\ \label{eq:cds2}
c_{\scS}&=-c_{\scS[i]}\quad  (i \in \lc 1, \cdots, q\rc),
    \end{align}
where $\scS(\sigma):=\lc \sigma_1 \prec \cdots \prec \sigma_{q+1} \prec \sigma \rc \in \scrS_{q+1}^w$ as in \pref{lm:b-delta}, and $S[i]$ is the one of \eqref{eq:S[i]}.
In \eqref{eq:cds1}, $P_{j, \scS} \subset N_\bR$ is regarded as a singular chain in $\left. N_\bR \middle/ N \right.$ by the natural projection $N_\bR \to \left. N_\bR \middle/ N \right.$.
\end{enumerate}
\end{lemma}
\begin{proof}
We prove the lemma by constructing the singular chains $c_{\scS}$ $(\scS \in \scrS_q^w)$ inductively on $q$.
When $q=d$ and $\scS \in \scrS_d^w$, we have $\Sigma_w(q+2)= \emptyset$, and the polytope $P_{j, \scS}$ is a single point in $N$.
By \eqref{eq:vol-sum} for $q=d$, we can see that the right-hand side of \eqref{eq:cds1} equals zero.
We set $c_{\scS}:=0$ for all $\scS \in \scrS_d^w$.
This obviously satisfies the conditions (1) and (2) in \pref{lm:cscds}.

Let $l$ be an integer such that $d+1-k <l \leq d$.
Suppose that we have constructed all the singular chains $c_{\scS}$ $(\scS \in \scrS_q^w)$ with $l \leq q \leq d$ satisfying the conditions (1) and (2).
We will construct singular chains $c_{\scS}$ $(\scS \in \scrS_q^w)$ with 
$q=l-1$.
In the same way as in the proof of \pref{lm:c(P)-boundary}, we can obtain
\begin{align}
    0=\partial c(\scE)
    &=\sum_{j \in J} l_j \cdot \lb
    \sum_{q=0}^{d}
    \lb
    \sum_{\scS \in \scrS_q^w}
\Delta_\scS \times \lb 
\sum_{\substack{\sigma \in \Sigma_w(q+2)\\ \sigma \succ \sigma[\scS] \\ \dim P_{j, \scS(\sigma)}=d-q-1}}
P_{j,\scS(\sigma)} \rb
+
\sum_{\substack{\scS \in \scrS_q^w \\ \dim P_{j, \scS}=d-q}}
\Delta_\scS \times \partial P_{j, \scS}
\rb \rb \\
&=
\sum_{q=0}^{d} \sum_{\scS \in \scrS_q^w}
\Delta_\scS \times
\lb
\sum_{\substack{\sigma \in \Sigma_w(q+2)\\ \sigma \succ \sigma[\scS]}}
\sum_{\substack{j \in J \\ \dim P_{j, \scS(\sigma)}=d-q-1}}
l_j \cdot P_{j, \scS(\sigma)}
+
\sum_{\substack{j \in J \\ \dim P_{j, \scS}=d-q}}
l_j \cdot \partial P_{j, \scS}
\rb.
\end{align}
This implies
\begin{align}\label{eq:S-fiber}
\sum_{\substack{\sigma \in \Sigma_w(q+2)\\ \sigma \succ \sigma[\scS]}}
\sum_{\substack{j \in J \\ \dim P_{j, \scS(\sigma)}=d-q-1}}
l_j \cdot P_{j, \scS(\sigma)}
+
\sum_{\substack{j \in J \\ \dim P_{j, \scS}=d-q}}
l_j \cdot \partial P_{j, \scS}
=0
\end{align}
for all $\scS \in \scrS_q^w$ $(0 \leq q \leq d)$.

Consider \eqref{eq:S-fiber} for $\scS \in \scrS_q^w$ with $q=l-1$.
By \eqref{eq:cds1} for the singular chain $c_{\scS(\sigma)}$ which we have already constructed, the former part of \eqref{eq:S-fiber} equals
\begin{align}\label{eq:S-fiber-1}
\sum_{\substack{\sigma \in \Sigma_w(l+1)\\ \sigma \succ \sigma[\scS]}}
\lb
\partial c_{\scS(\sigma)}
-\sum_{\substack{\sigma' \in \Sigma_w(l+2) \\ \sigma' \succ \sigma}} c_{\scS(\sigma, \sigma')}
\rb
=
\sum_{\substack{\sigma \in \Sigma_w(l+1)\\ \sigma \succ \sigma[\scS]}}
\partial c_{\scS(\sigma)}
-
\sum_{\substack{\sigma \in \Sigma_w(l+1)\\ \sigma \succ \sigma[\scS]}}
\sum_{\substack{\sigma' \in \Sigma_w(l+2) \\ \sigma' \succ \sigma}} 
c_{\scS(\sigma, \sigma')},
\end{align}
where $\scS(\sigma, \sigma') \in \scrS_{l+1}^w$ is the sequence of cones obtained by adding $\sigma' \in \Sigma_w(l+2)$ to $\scS(\sigma)$ as the last cone.
The latter part of \eqref{eq:S-fiber-1} is zero, since the term of $\scS(\sigma, \sigma')$ and the term of $\scS(\sigma, \sigma')[l+1]$ (\eqref{eq:S[i]} with $\scS=\scS(\sigma, \sigma')$ and $i=l+1$) cancel each other out by \eqref{eq:cds2}.
Therefore, \eqref{eq:S-fiber} for $\scS \in \scrS_q^w$ with $q=l-1$ is written as $\partial d_\scS=0$ setting
\begin{align}\label{eq:dds}
d_{\scS}:=
    \sum_{\substack{\sigma \in \Sigma_w(l+1)\\ \sigma \succ \sigma[\scS]}}
c_{\scS(\sigma)}
+
\sum_{\substack{j \in J\\ \dim P_{j, \scS}=d+1-l}}
l_j \cdot P_{j, \scS}.
\end{align}

Regarding \eqref{eq:dds}, by the definition of $P_{j, \scS}$ (cf.~\eqref{eq:face-S}), we can see that $P_{j, \scS}$ sits in 
\begin{align}\label{eq:P-plane}
    \left. \mu_j\lb \sigma[\scS]\rb^\perp \middle/ 
    \lb \mu_j \lb \sigma[\scS] \rb^\perp \cap N \rb \right.
\subset
\left. \sigma[\scS]^\perp \middle/ \lb \sigma[\scS]^\perp\cap N \rb \right.,
\end{align}
where $\mu_j \colon \Sigma_w \to \Sigma_{P_j}$ denotes the map $\mu$ of \eqref{eq:mu} for $P=P_j$.
By the condition (1) for the singular chain $c_{\scS(\sigma)}$ (the induction hypothesis), we can see that $c_{\scS(\sigma)}$ also sits in 
\begin{align}\label{eq:codim1}
    \left. \sigma^\perp \middle/ \lb \sigma^\perp\cap N \rb \right. \subset
\left. \sigma[\scS]^\perp \middle/ \lb \sigma[\scS]^\perp\cap N \rb \right..
\end{align}
From these and $\partial d_\scS=0$, we can see that the singular chain $d_{\scS}$ is a $(d+1-l)$-cycle in the $(d+1-l)$-dimensional subtorus $\left. \sigma[\scS]^\perp \middle/ \lb \sigma[\scS]^\perp\cap N \rb \right.$.

We show that the cycle $d_{\scS}$ is actually trivial by computing the Euclidean volume (with sign) of the area covered by $d_{\scS}$.
Since $c_{\scS(\sigma)}$ appearing in \eqref{eq:dds} is in \eqref{eq:codim1}, the subset of codimension $1$ in $\left. \sigma[\scS]^\perp \middle/ \lb \sigma[\scS]^\perp\cap N \rb \right.$, the volume of the area covered by it is zero.
The volume of the area covered by the latter part of \eqref{eq:dds} is also zero by \eqref{eq:vol-sum} for $q=l-1 \geq d+1-k$.
Hence, the cycle $d_{\scS}$ is trivial, and there exists a singular $(d+2-l)$-chain on the subtorus $\left. \sigma[\scS]^\perp \middle/ \lb \sigma[\scS]^\perp\cap N \rb \right.$ whose boundary coincides with $d_{\scS}$.
We write it as $c_{\scS}$.
It is clear that this $c_{\scS}$ satisfies \eqref{eq:cds1} and the condition (1) of the lemma.

In the above construction of the singular chains $c_{\scS}$, we can construct them so that \eqref{eq:cds2} is also satisfied as follows:
For every cone $\sigma'' \in \Sigma_w(l)$, we choose an arbitrary sequence 
$\scS= \lc\sigma_1 \cdots \prec \sigma_l \rc \in \scrS_{l-1}^w$ such that $(\sigma_l=:)\sigma[\scS]=\sigma''$, and construct the singular chain $c_{\scS}$ for $\scS$ as explained above.
Let $e_i \in M$ $(1 \leq i \leq l)$ be the elements such that $e_i \in M$ $(1 \leq i \leq j)$ are the primitive generators of the cone $\sigma_j$ for all $j \in \lc 1, \cdots, l \rc$.
For every $\scS' \in \scrS_{l-1}^w$ which can be written as 
\begin{align}
    \scS'= \lc \sigma_1' \prec \cdots \prec \sigma_l' \rc  
    \quad 
    \lb \sigma_j':=\cone \lb \lc e_1', \cdots, e_j' \rc \rb, 1 \leq j \leq l \rb
\end{align}
with a permutation $\lc e_1', \cdots, e_l' \rc$ of $\lc e_1, \cdots, e_l \rc$, we define
\begin{align}
   c_{\scS'}:=(-1)^{\mathrm{sign}(\scS')} \cdot c_{\scS},
\end{align}
where $\mathrm{sign}(\scS')$ denotes the signature of the permutation $\lc e_1', \cdots, e_l' \rc$.
Then we have
\begin{align}
c_{\scS'(\sigma)}&=(-1)^{\mathrm{sign}(\scS')} \cdot c_{\scS(\sigma)} \quad \lb \scS'(\sigma):=\lc \sigma_1' \prec \cdots \prec \sigma_l' \prec \sigma \rc \in \scrS_l^w \rb \\
P_{j, \scS'}&=(-1)^{\mathrm{sign}(\scS')}  \cdot  P_{j, \scS},
\end{align}
respectively by the induction hypothesis and by the definition of the orientation of $P_{j, \scS}$.
By these, we have
\begin{align}
\partial c_{\scS'}
&=(-1)^{\mathrm{sign}(\scS')} \cdot \partial c_{\scS}\\
&=(-1)^{\mathrm{sign}(\scS')} \cdot \lb \sum_{\substack{\sigma \in \Sigma_w(l+1) \\ \sigma \succ \sigma[\scS]}} c_{\scS(\sigma)}
+
\sum_{\substack{j \in J \\  \dim P_{j, \scS}=d+1-l}}
l_j \cdot P_{j, \scS} \rb \\
&=\sum_{\substack{\sigma \in \Sigma_w(l+1) \\ \sigma \succ \sigma[\scS']}} c_{\scS'(\sigma)}
+
\sum_{\substack{j \in J \\  \dim P_{j, \scS'}=d+1-l}}
l_j \cdot P_{j, \scS'}
\end{align}
which is \eqref{eq:cds1} for $\scS' \in \scrS_{l-1}^w$.
Therefore, the family of singular chains $\lc c_{\scS} \rc_\scS$ constructed in this way satisfies \eqref{eq:cds1}.
It is also clear that this satisfies \eqref{eq:cds2} and the condition (1) of the lemma.
Thus we conclude the lemma.
\end{proof}

We define the singular chain $\tilde{c}(\scE)$ in $N_\bR \times (\left. N_\bR \middle/ N \right.)$ by
\begin{align}\label{eq:t-c(scE)}
\tilde{c}(\scE)
:=
c(\scE)
-\sum_{q=d+1-k}^d \sum_{\scS \in \scrS_q^w}
\partial \lb \Delta_\scS \times c_{\scS} \rb,
\end{align}
It is clear from $\partial c(\scE)=0$ and $\partial \circ \partial=0$ that $\tilde{c}(\scE)$ is also a cycle.
By abuse of notation, we write the image of the cycle $\tilde{c}(\scE)$ also as $\tilde{c}(\scE) \subset N_\bR \times (\left. N_\bR \middle/ N \right.)$.

For any cone $\sigma \in \Sigma_w(d+1-k)$, let $\iota \colon V(\sigma) \hookrightarrow Y_w$ be the closure of the torus orbit corresponding to the cone $\sigma$.
By the definition \eqref{eq:k} of the integer $k$, we have $\ch_{k'}(\scE)=0$ for all $k' <k$.
This implies $\ch_{k'}(\iota^\ast \scE)=0$ for all $k' <k$.
By using this and the Grothendieck--Riemann--Roch theorem, we obtain
\begin{align}
    \deg \lb \ch_{k}(\scE) \cap \ld V(\sigma)\rd \rb
    =
    \int_{V(\sigma)} \iota^\ast \lb \ch_{k}(\scE) \rb
    =
    \int_{V(\sigma)}  \ch (\iota^\ast \scE) \cdot \td\lb T_{V(\sigma)} \rb
    =
    \chi \lb V(\sigma), \iota^\ast \scE \rb \in \bZ.
\end{align}
This implies $\ch_{k}(\scE) \in  A^k(Y_w)$ (cf.~\cite[Proposition 2.4]{MR1415592}).
We write the Minkowski weight corresponding to $(-1)^{d+1-k} \cdot \ch_{k}(\scE)$ 
as $a_\scE \in \MW^{k}(\Sigma_w)$.

\begin{lemma}\label{lm:t-c(scE)}
The following statements hold:
\begin{enumerate}
    \item The image of $\tilde{c}(\scE)$ by the first projection $\pi_1$ \eqref{eq:1-proj} is contained in 
$\partial \nabla_w^{d-k}$.
\item For any point $n \in \mathring{\Delta}_\scS \subset \partial \nabla_w^{d-k}$ 
($\scS \in \scrS_q^w$, $q \in \lc 0, \cdots, d-k \rc$), 
the intersection $\tilde{c}(\scE) \cap \pi_1^{-1}(n)$ is in the $(d-q)$-dimensional subtorus
\begin{align}\label{eq:subtorus}
\lc n \rc 
\times 
\lb \left. \sigma[\scS]^\perp \middle/ \lb \sigma[\scS]^\perp\cap N \rb \right. \rb
\subset 
N_\bR \times \lb \left. N_\bR \middle/ N \right.\rb.
\end{align}
\item Suppose $q=d-k$, $\scS \in \scrS_{d-k}^w$, and 
$n \in \mathring{\Delta}_\scS \subset \partial \nabla_w^{d-k}$.
Then the intersection $\tilde{c}(\scE) \cap \pi_1^{-1}(n)$ is a $k$-cycle in the torus $\lc n \rc \times \left. \lb N_\bR \middle/ N \right. \rb$ whose homology class coincides with that of the $k$-dimensional subtorus \eqref{eq:subtorus} multiplied by the integer $a_{\scE}(\sigma[\scS])$.
Here, we consider the orientation of the subtorus \eqref{eq:subtorus} determined by the map \eqref{eq:fd-q}, in the same way as for the orientations of the polytopes $P_{\scS}$.
\end{enumerate}
\end{lemma}
\begin{proof}
In order to prove the lemma, we compute the intersection $\tilde{c}(\scE) \cap \pi_1^{-1}(n)$.
The latter part of \eqref{eq:t-c(scE)} is
\begin{align}\label{eq:c(scE)2}
-\sum_{q=d+1-k}^d \sum_{\scS \in \scrS_q^w}
\lb
\sum_{i=1}^q (-1)^{i-1} \cdot \partial_i \Delta_\scS \times c_{\scS}
+
(-1)^q \cdot
\partial_{q+1} \Delta_\scS \times c_{\scS}
+
(-1)^q \cdot
\Delta_\scS \times \partial c_{\scS}
\rb,
\end{align}
where $\partial_i \Delta_\scS$ denotes the facet of $\Delta_\scS$, which does not have the $i$-th vertex of $\Delta_\scS$ ($1 \leq i \leq q+1$).
Regarding the first part of \eqref{eq:c(scE)2}, if we consider $\scS[i]$ of \eqref{eq:S[i]}, then one has $\partial_i \Delta_{\scS[i]}=\partial_i \Delta_\scS$ and $c_{\scS}=-c_{\scS[i]}$ from \eqref{eq:cds2}.
Therefore, the sum of the first part of \eqref{eq:c(scE)2}  becomes zero.
The remaining part of \eqref{eq:c(scE)2} can be written as
\begin{align}\label{eq:t-cd0}
\sum_{q=d-k}^{d-1} \sum_{\scS \in \scrS_q^w} \sum_{\substack{\sigma \in \Sigma_w(q+2) \\  \sigma \succ \sigma[\scS]}}
(-1)^q \cdot
\Delta_\scS \times c_{\scS(\sigma)}
-\sum_{q=d+1-k}^d \sum_{\scS \in \scrS_q^w}
(-1)^q \cdot
\Delta_\scS \times \partial c_{\scS}.
\end{align}
The singular chain $\tilde{c}(\scE)$ is the sum of \eqref{eq:t-cd0} and the singular chain $c(\scE)$ \eqref{eq:c(scE)}.
The part of the singular chain $\tilde{c}(\scE)$, which contains $\Delta_\scS$ $(\scS \in \scrS_q^w)$ is written as $\Delta_\scS \times F_\scS$ with
\renewcommand{\arraystretch}{2.5}
\begin{align}\label{eq:FS}
F_\scS:=
\left\{ \begin{array}{ll}
\displaystyle (-1)^{q} \cdot \lb \sum_{\substack{j \in J \\ \dim P_{j, \scS}=d-q}} l_j \cdot P_{j, \scS}
+ \sum_{\substack{\sigma \in \Sigma_w(q+2) \\  \sigma \succ \sigma[\scS]}}
c_{\scS(\sigma)}
-
\partial c_{\scS} \rb & q \in \lc d+1-k, \cdots, d \rc \\
\displaystyle (-1)^{d-k} \cdot \lb \sum_{\substack{j \in J \\ \dim P_{j, \scS}=k}} 
l_j \cdot P_{j, \scS}
+ \sum_{\substack{\sigma \in \Sigma_w(d-k+2) \\  \sigma \succ \sigma[\scS]}}
c_{\scS(\sigma)} \rb & q=d-k \\
\displaystyle (-1)^{q} \cdot \lb \sum_{\substack{j \in J \\ \dim P_{j, \scS}=d-q}} 
l_j \cdot P_{j, \scS} \rb & q \in \lc 0, \cdots, d-k-1 \rc.
\end{array} 
\right.
\end{align}
\renewcommand{\arraystretch}{1.0}

By \eqref{eq:cds1}, we can see that $F_\scS=0$ for $\scS \in \scrS_q^w$ with $q \in \lc d+1-k, \cdots, d \rc$.
The claim (1) of the lemma follows from this.
We can also see that for any point $n \in \mathring{\Delta}_\scS$ 
($\scS \in \scrS_q^w$, $q \in \lc 0, \cdots, d-k \rc$), 
the $\lb \left. N_\bR \middle/ N \right.\rb$-component of the intersection 
$\tilde{c}(\scE) \cap \pi_1^{-1}(n) \subset \lc n \rc \times \lb \left. N_\bR \middle/ N \right.\rb$ 
is
\begin{align}\label{eq:int-preimage}
\sum_{q'=q}^{d-k}
\sum_{\substack{\scS' \in \scrS_{q'}^w \\ \scS' \supset \scS} }
F_{\scS'},
\end{align}
since $\Delta_{\scS'}$ with $\scS' \supset \scS$ contains $\mathring{\Delta}_\scS \ni n$.
As we saw in \eqref{eq:P-plane}, the polytope $P_{j, \scS'}$ appearing in $F_{\scS'}$ of \eqref{eq:int-preimage} sits in 
\begin{align}\label{eq:ssss}
\left. \sigma[\scS']^\perp \middle/ \lb \sigma[\scS']^\perp\cap N \rb \right. 
    \subset \left. \sigma[\scS]^\perp \middle/ \lb \sigma[\scS]^\perp\cap N \rb \right..
\end{align}
We can see from \pref{lm:cscds}(1) that $c_{\scS'(\sigma)}$ appearing in $F_{\scS'}$ with $\scS' \in \scrS_{q'=d-k}^w$ also sits in \eqref{eq:ssss}.
Therefore, \eqref{eq:int-preimage} sits in \eqref{eq:ssss}, and we can see that the claim (2) of the lemma holds.

Lastly, we prove the claim (3) of the lemma.
\eqref{eq:int-preimage} with $\scS \in \scrS_{d-k}^w$ is
\begin{align}\label{eq:int-preimage2}
    (-1)^{d-k} \cdot 
    \lb 
    \sum_{\substack{j \in J \\ \dim P_{j, \scS}=k}} 
l_j \cdot P_{j, \scS}
+ \sum_{\substack{\sigma \in \Sigma_w(d-k+2) \\  \sigma \succ \sigma[\scS]}}
c_{\scS(\sigma)}
\rb.
\end{align}
The argument in the third paragraph of the proof of \pref{lm:cscds} also works for $l=d+1-k$, and we have $\partial d_\scS=0$ for $d_\scS$ of \eqref{eq:dds} with $\scS \in \scrS_{d-k}^w$.
We can see that \eqref{eq:int-preimage2} is a $k$-cycle in the $k$-dimensional subtorus \eqref{eq:subtorus}.
The singular chain $c_{\scS(\sigma)}$ in \eqref{eq:int-preimage2} is in the subset of codimension $1$, $\left. \sigma^\perp \middle/ \lb \sigma^\perp\cap N \rb \right. \subset
\left. \sigma[\scS]^\perp \middle/ \lb \sigma[\scS]^\perp\cap N \rb \right.$, and the volume of the area covered by it is zero.
By \eqref{eq:vol-sum0} for $q=d-k$, the volume of the area covered by the former part of \eqref{eq:int-preimage2} is 
\begin{align}
    (-1)^{d+1-k} \cdot \deg \lb \ch_{k}(\scE) \cap \ld V(\sigma[\scS])\rd \rb
    =a_\scE \lb \sigma[\scS] \rb
\end{align}
(cf.~\eqref{eq:amw}).
Therefore, the volume of the area covered by \eqref{eq:int-preimage2} in total is $a_\scE \lb \sigma[\scS] \rb$.
This implies the claim (3).
We conclude \pref{lm:t-c(scE)}.
\end{proof}

\subsection{Phase shift}\label{sc:step1-3}

For $m \in A$, we set
\begin{align}\label{eq:mu_m}
\mu_m \colon N_\bC \to \bC, \quad n \mapsto \lambda_m + \la m, n \ra.
\end{align}
We take a small real constant $\kappa >0$, and define
\begin{align}\label{eq:nkw}
N_\kappa^w:=\lc n \in N_\bR \relmid \mu_w (n)-\kappa \leq \min_{m \in A \setminus \lc w \rc} \mu_m(n) \leq \mu_w(n)+\kappa \rc.
\end{align}
Here we choose the constant $\kappa >0$ so that the statements in \cite[Lemma 3.0.2]{MR4782805} hold.
The set $N_\kappa^w$ is a neighborhood of $\partial \nabla_w$.
For a point $n \in N_\bR$, we also set
\begin{align}\label{eq:K}
K_\kappa^n:= \lc k \in A \setminus \lc w \rc \relmid \mu_k(n) \leq \mu_w(n)+\kappa \rc.
\end{align}
We choose a branch of the argument $\arg \lb - c_m/c_w \rb$ for every $m \in A \setminus \lc w \rc$, and take a smooth function $\phi \colon N_{\kappa}^w \to N_\bR$ such that
\begin{align}\label{eq:phase-shift}
\la k-w, \phi(n) \ra 
= - \frac{1}{2 \pi} \arg \lb -\frac{c_k}{c_w} \rb
\end{align}
for any $n \in N_{\kappa}^w$ and $k \in K_{2\kappa}^n$.
Such a function $\phi$ is called a \emph{phase-shifting function} in \cite[Section 5.2]{MR4194298}.
We refer the reader to Example 5.4 in loc.cit.~for an example of a function $\phi$ (with $w=0, c_w=-1$).
For the construction of $\phi$, see also \cite[Section 3]{MR4782805}.
For notational convenience, we define our phase-shifting function to be $\left. 1 \middle/ 2 \pi \right.$ times the ones used in \cite{MR4194298, MR4782805}.

We define the $d$-cycle $C(\scE) \subset N_\bR \times \lb \left. N_\bR \middle/ N \right. \rb$ to be the image of $\tilde{c}(\scE) \subset N_\bR \times \lb \left. N_\bR \middle/ N \right. \rb$ of \eqref{eq:t-c(scE)} by the map
\begin{align}\label{eq:shift-map}
    \Phi \colon  
    N_\kappa^w \times \lb \left. N_\bR \middle/ N \right.\rb 
    \to 
    N_\bR \times \lb \left. N_\bR \middle/ N \right. \rb, 
    \quad (n, n') \mapsto (n, n'+\phi(n)).
\end{align}
We will show that the cycle $C(\scE)$ satisfies the conditions in \pref{th:main2}(2).
It is obvious from \pref{lm:t-c(scE)}(1) that the $d$-cycle $C(\scE)$ satisfies the condition (2-a) of \pref{th:main2}.
The homology class of the subtorus \eqref{eq:subtorus} is exactly $f(\scS)$ as an element of $\bigwedge^{d-p} N$, and we can see from \pref{lm:t-c(scE)}(3) that the condition (2-b) of \pref{th:main2} is also satisfied. 
The remaining condition (2-c) of \pref{th:main2} will be proved in the next subsection.

\subsection{Construction of small perturbations}\label{sc:step2}

In this subsection, we construct the small perturbation of \pref{th:main2}(2-c) by using the technique in \cite[Section 5.2]{MR4194298}.
There are also similar or related constructions in \cite{MR2240909, MR2529936, MR2871160, MR3228454, MR3948684, MR4782805}.
For $w \in W$, we rewrite the equation $f=0$ as $-\lb f- k_w z^w \rb / k_w z^w=1$, and set
\begin{align}
f_t^w:=\left. -\lb f- k_w z^w \rb / k_w z^w \right|_{x=t} \in \bC \ld M \rd.
\end{align}
Then $\mathring{Z}_t=\lc z \in N_{\bC^\ast} \relmid f_t^w(z)=1 \rc$.
For $m \in A$, we set $\lambda_m:=\mathrm{val}(k_m) \in \mathbb{Q}$, 
and let $c_m \in \bC^\ast$ denote the coefficient of $x^{\lambda_m}$ in $k_m \in K$.
We also write its absolute value $|c_m|$ as $r_m \in \bR_{>0}$.
We define 
\begin{align}\label{eq:tft}
\tilde{f}_t^w&:=\sum_{m \in A \setminus \lc w \rc} \frac{r_m}{r_w} t^{\lambda_m-\lambda_w} z^{m-w}.
\end{align}

\begin{lemma}
\label{lm:delta1}
For sufficiently small $t >0$, there exists a continuous map 
$\delta_{1, t} \colon \partial \nabla_w \to N_\bR$ 
satisfying the following conditions:
\begin{enumerate}
\item For any $n \in \partial \nabla_w$, one has
\begin{align}
\tilde{f}_t^w \lb i_t \lb n+\delta_{1, t}(n) \rb \rb
=
1.
\end{align}
\item $\left| \left| \delta_{1, t} \right| \right|_{C^0} = O \lb(-\log t)^{-1} \rb$, 
where $|| \bullet ||_{C^0}$ denotes the $C^0$-norm over $\partial \nabla_w$.
\end{enumerate}
\end{lemma}
\begin{proof}
    This is proved in the same way as \cite[Proposition 5.3]{MR4194298} and \cite[Proposition 3.0.5]{MR4782805}.
We consider the smooth function
\begin{align}
    \tilde{g}_t^w \colon N_\bR \to \bR, 
    \quad n \mapsto \tilde{f}_t^w(i_t(n))=\sum_{m \in A \setminus \lc w \rc} \frac{r_m}{r_w} t^{(\mu_m-\mu_w)(n)},
\end{align}
and let
    \begin{align}
        \grad \tilde{g}_t^w
        =\lb \frac{\partial \tilde{g}_t^w}{\partial n_0}, \cdots, \frac{\partial \tilde{g}_t^w}{\partial n_d} \rb
    \end{align}
    denote the gradient vector field on $N_\bR$, where $(n_0, \cdots, n_d)$ are coordinates on $N_\bR \cong \bR^{d+1}$.
    We consider the differential equation for an unknown function 
    $c \colon \partial \nabla_w \times [0, 1+\varepsilon] \to N_\bR$
\begin{align}\label{eq:de}
        \frac{d}{ds} c \lb n, s \rb
        =\lb 1- \tilde{g}_t^w(n) \rb \cdot 
        \frac{\grad \tilde{g}_t^w}{\left| \grad \tilde{g}_t^w \right|^2} \lb c \lb n, s \rb \rb
\end{align}
with the initial condition $c \lb n, 0 \rb=n \in \partial \nabla_w$, where $\varepsilon>0$ is a small real constant.
We will prove \pref{lm:delta1} by showing that there exists a global solution $c \lb n, s \rb$ for \eqref{eq:de} contained in $N_{\kappa}^w$, 
and that the map $\delta_{1, t} \colon \partial \nabla_w \to N_\bR$ defined by
\begin{align}
\delta_{1, t} (n):=c(n, 1)-n
\end{align}
satisfies the conditions (1) and (2) of \pref{lm:delta1}.

We will use the following technical lemma:

\begin{lemma}\label{lm:grad}
When $t>0$ is sufficiently small, one has 
$\left| \grad \tilde{g}_t^w (n) \right| \neq 0$ on $N_{\kappa}^w$.
Furthermore, there exist constants $C_1, C_2 >0$ such that for any $n \in N_{\kappa}^w$ satisfying $\tilde{g}_t^w \lb n \rb \neq 0$, we have
\begin{align}\label{eq:grad/g}
\frac{\left| \grad \tilde{g}_t^w (n) \right| }{\left| \tilde{g}_t^w (n) \right|}
\geq
(-\log t)\lb C_1-C_2 t^\kappa \rb.
\end{align}
\end{lemma}

\pref{lm:grad} is \cite[Lemma 3.0.7]{MR4782805} with $k_m=r_m \cdot x^{\lambda_m}$ $(m \in A \setminus \lc w \rc)$, $k_w=-r_w \cdot x^{\lambda_w}$, and $x=0$.
The range of $n$ is also restricted to $N_\kappa^w$, since this suffices for our present purposes.

We fix $n \in \partial \nabla_w$.
Let $s(n) \in \ld 0, 1+\varepsilon\rd$ be the supremum of $s' \in \ld 0, 1+\varepsilon\rd$ such that there exists a solution $c(n, s)$ of \eqref{eq:de} on the interval $\ld 0, s' \rb$ and $c(n, s) \in N_{\kappa}^w$ for all $s \in \ld 0, s' \rb$.
For any $s_0 \in \ld 0, s(n) \rb$, we have
\begin{align}
\tilde{g}_t^w \lb c \lb n, s_0 \rb \rb-\tilde{g}_t^w \lb c \lb n, 0 \rb \rb
&=
\int_0^{s_0} \grad \tilde{g}_t^w \lb c \lb n, s \rb \rb \cdot \frac{d}{ds} c \lb n, s \rb ds\\
&=\int_0^{s_0} \lb 1- \tilde{g}_t^w(n) \rb ds\\
&=s_0 \cdot \lb 1- \tilde{g}_t^w(n) \rb
\end{align}
by the definition of the differential equation \eqref{eq:de}.
Hence, we have
\begin{align}\label{eq:ggxi}
\tilde{g}_t^w \lb c \lb n, s_0 \rb \rb
=
(1-s_0)\tilde{g}_t^w(n)+s_0.
\end{align}
From this, we can see that $\tilde{g}_t^w \lb c \lb n, s_0 \rb \rb$ is in the line connecting $1$ and $\tilde{g}_t^w(n)(>0)$, and that $\tilde{g}_t^w \lb c \lb n, s_0 \rb \rb$ is greater than some positive real constant.
By this, \eqref{eq:de}, and \eqref{eq:grad/g}, we obtain
\begin{align}
\left| \frac{d}{ds} c \lb n, s \rb \right| 
= \frac{\left| 1-\tilde{g}_t^w(n) \right|}{\left| \grad \tilde{g}_t^w \lb c \lb n, s \rb \rb \right|}
< C_3 (-\log t)^{-1}
\end{align}
for $s \in \ld 0, s(n) \rb$, where $C_3>0$ is some real constant.
This implies the limit $\lim_{s \to s(n)-0} c(n,s) \in N_\bR$ exists.

Suppose $s(n) <1+\varepsilon$.
Then the solution for \eqref{eq:de} can be extended to a larger interval 
$\ld 0, s(n)+\varepsilon_1 \rb (\subset \ld 0, 1+\varepsilon \rb)$ 
with some small $\varepsilon_1 >0$, and for any $s_0 \in \ld 0, s(n)+\varepsilon_1 \rb$, one has
\begin{align}\label{eq:cc}
\left| c(n, s_0)-c(n, 0) \right| 
\leq \int_0^{s_0} \left| \frac{d}{ds} c \lb n, s \rb \right| ds 
\leq (1+\varepsilon) C_3 (-\log t)^{-1}.
\end{align}
Since $c \lb n, 0 \rb=n \in \partial \nabla_w$, we have
\begin{align}
\left| \lb \min_{m \in A \setminus \lc w \rc} \mu_m - \mu_w \rb \lb c(n, s_0) \rb \right|
&=\left| \lb \min_{m \in A \setminus \lc w \rc} \mu_m - \mu_w \rb 
\lb n+\lb c(n, s_0)-c(n, 0)\rb\rb \right| \\
& \leq C_4 (-\log t)^{-1} \\
& \leq \kappa
\end{align}
when $t >0$ is sufficiently small.
Here $C_4 >0$ is also some constant.
Hence, one has $c \lb n, s_0 \rb \in N_{\kappa}^w$ for any $s_0 \in \ld 0, s(n)+\varepsilon_1 \rb$.
This contradicts the original assumption on $s(n)$.
We conclude that $s(n)=1+\varepsilon$ and the solution $c(n, s) \in N_{\kappa}^w$ exists on the interval $\ld 0, 1+\varepsilon \rd$.

By \eqref{eq:ggxi} and \eqref{eq:cc} for $s_0=1$, we obtain the conditions (1) and (2) of \pref{lm:delta1} respectively.
We conclude \pref{lm:delta1}.
\end{proof}

We set
\begin{align}
R_t&:=\lc n \in N_\bR \relmid \frac{1}{2} \leq \tilde{f}_t^w \lb i_t(n) \rb \leq \frac{3}{2} \rc\\ \label{eq:S_t}
S_t&:=
\lc (n, n') \in R_t \times N_\bR
\relmid 
\la k-w, n'\ra=-\frac{1}{2 \pi} \arg \lb -\frac{c_k}{c_w} \rb, \forall k \in K_\kappa^n \rc.
\end{align}
One has $R_t \subset N_\kappa^w$ for sufficiently small $t>0$ (\cite[Lemma 3.0.4]{MR4782805}).

\begin{lemma}
\label{lm:delta2}
For sufficiently small $t >0$, there exists a smooth map 
$\delta_{2, t} \colon S_t \to N_\bR \times N_\bR$ 
satisfying the following conditions:
\begin{enumerate}
\item For any $(n, n') \in S_t$, one has
\begin{align}
f_t^w \lb i_t \circ j_t \lb (n, n')+\delta_{2, t}(n, n') \rb \rb
=
\tilde{f}_t^w \lb i_t(n) \rb.
\end{align}
\item For any $(n, n') \in S_t$ and $e^\ast \in \lc n'' \in N \relmid \la k-w, n'' \ra=0, \forall k \in K_\kappa^n \rc$, one has $(n, n'+ e^\ast) \in S_t$ and
\begin{align}
\delta_{2, t}(n, n')
=\delta_{2, t} \lb n, n'+ e^\ast \rb.
\end{align}
\item $\left| \left| \delta_{2, t} \right| \right|_{C^0} = O(t^{\kappa})$, where $|| \bullet ||_{C^0}$ denotes the $C^0$-norm over $S_t$.
\end{enumerate}
\end{lemma}
\begin{proof}
We consider the holomorphic function
\begin{align}
    g_t^w \colon N_\bC \to \bC, 
    \quad n \mapsto f_t^w(i_t(n))
    =\sum_{m \in A \setminus \lc w \rc} \lb - \frac{k_{m, t}}{k_{w, t}} \rb t^{\la m-w, n \ra},
\end{align}
and let
\begin{align}
        \grad g_t^w
        =\lb \overline{\frac{\partial g_t^w}{\partial n_0}}, \cdots, \overline{\frac{\partial g_t^w}{\partial n_d}} \rb
\end{align}
denote the gradient vector field on $N_\bC$, where $(n_0, \cdots, n_d)$ are complex coordinates on $N_\bC \cong \bC^{d+1}$.
We define a function $\xi_t^w \colon S_t \to \bC$ by
\begin{align}\label{eq:df-xi}
\xi_t^w \lb n, n' \rb
:=\tilde{f}_t^w \lb i_t \lb n \rb \rb
-f_t^w \lb i_t \circ j_t (n, n') \rb,
\end{align}
and consider the differential equation for an unknown function 
    $c \colon S_t \times [0, 1+\varepsilon] \to N_\bC$
\begin{align}\label{eq:de'}
        \frac{d}{ds} c \lb n, n', s \rb
        = \xi_t^w \lb n, n' \rb \cdot 
        \frac{\grad g_t^w}{\left| \grad g_t^w \right|^2} \lb c \lb n, n', s \rb \rb
\end{align}
with the initial condition $c \lb n, n', 0 \rb=j_t \lb n, n' \rb \in N_\bC$, where $\varepsilon>0$ is a small real constant.
\pref{lm:delta2} is proved by showing that there exists a global solution $c \lb n, n', s \rb$ for \eqref{eq:de'} contained in 
\begin{align}
N_{2\kappa, \bC}^w:=\lc n \in N_\bC \relmid \Re \lb n \rb \in N_{2\kappa}^w \rc,
\end{align}
where $N_{2\kappa}^w$ is \eqref{eq:nkw} with $\kappa$ replaced with $2 \kappa$, and that the map $\delta_{2, t} \colon S_t \to N_\bR \times N_\bR$ defined by
\begin{align}
\delta_{2, t} (n, n'):=j_t^{-1} \lb c(n, n', 1)-j_t \lb n, n' \rb \rb
\end{align}
satisfies the conditions (1)--(3) in \pref{lm:delta2}.
For doing this, we will use the following lemma:

\begin{lemma}{\rm(cf.~\cite[Lemma 5.2]{MR4194298}, \cite[Lemma 3.0.6]{MR4782805})}\label{lm:diff}
There is some constant $C >0$ such that for any $(n, n') \in S_t$, one has
\begin{align}
\left| \xi_t^w (n, n') \right| \leq C \cdot t^\kappa.
\end{align}
\end{lemma}
\begin{proof}
By the definition \eqref{eq:S_t} of $S_t$ and \cite[Lemma 3.0.2(3)]{MR4782805}, we have
\begin{align}\label{eq:argument-phase}
f_t^w \lb i_t \circ j_t (n, n') \rb
=\sum_{m \in A \setminus \lc w \rc} \lb - \frac{k_{m, t}}{k_{w, t}} \rb 
t^{\la m-w, n+\frac{2\pi \sqrt{-1}}{\log t}n' \ra}
=O \lb t^\kappa \rb+\sum_{k \in K_\kappa^n} \frac{r_k}{r_w} \lb 1+O(t^{2\kappa}) \rb t^{\lb \mu_k-\mu_w\rb(n)}.
\end{align}
One also has
\begin{align}
\tilde{f}_t^w \lb i_t \lb n \rb \rb
=O \lb t^\kappa \rb+\sum_{k \in K_\kappa^n} \frac{r_k}{r_w} t^{\lb \mu_k-\mu_w\rb(n)}.
\end{align}
By combining these, we get
\begin{align}\label{eq:otk}
\xi_t^w \lb n, n' \rb=O \lb t^\kappa \rb
+\sum_{k \in K_\kappa^n} O(t^{2\kappa}) \frac{r_k}{r_w} t^{\lb \mu_k-\mu_w\rb(n)}.
\end{align}
Since $\min_{m \in A \setminus \lc w \rc} \mu_m(n) -\mu_w (n) \geq -\kappa$ for $n \in R_t \subset N_\kappa^w$, \eqref{eq:otk} is $O \lb t^\kappa \rb$. 
Thus we obtain the claim.
\end{proof}

One can show the existence of a solution of \eqref{eq:de'} and the conditions (1), (3) in \pref{lm:delta2} in the same way as in \cite[Proposition 5.3]{MR4194298}, \cite[Proposition 3.0.5]{MR4782805}, and \pref{lm:delta1} above, by using \pref{lm:diff} and \cite[Lemma 3.0.7]{MR4782805}.
We omit the details.

We show that the condition (2) is also satisfied.
Let $(n, n') \in S_t$ and $e^\ast$ be elements in the condition (2).
It is clear from the definition \eqref{eq:S_t} of $S_t$ that one has $(n, n'+ e^\ast) \in S_t$.
We can also easily check
\begin{align}
\xi_t^w \lb n, n'+e^\ast \rb=\xi_t^w \lb n, n' \rb, \quad
g_t^w \lb n+\frac{2 \pi \sqrt{-1}}{\log t}e^\ast\rb=g_t^w \lb n\rb.
\end{align}
One can see from these that the solution $c(n, n'+e^\ast, s)$ of the differential equation \eqref{eq:de'} for $(n, n'+e^\ast) \in S_t$ is given by 
$c(n, n', s)+\frac{2 \pi \sqrt{-1}}{\log t}e^\ast$, where $c(n, n', s)$ is the solution of \eqref{eq:de'} for $(n, n') \in S_t$.
Thus we obtain
\begin{align}
    \delta_{2, t} (n, n'+e^\ast)
    &=j_t^{-1} \lb c(n, n'+e^\ast, 1)-j_t \lb n, n'+e^\ast \rb \rb\\
    &=j_t^{-1} \lb c(n, n', 1)+\frac{2 \pi \sqrt{-1}}{\log t}e^\ast-j_t \lb n, n'\rb-\frac{2 \pi \sqrt{-1}}{\log t}e^\ast \rb\\
    &=\delta_{2, t} (n, n').
\end{align}
We conclude \pref{lm:delta2}.
\end{proof}

For any point $n \in N_\bR$, we define 
$T_n \subset \left. N_\bR \middle/ N \right.$ 
to be the image of the affine subspace
\begin{align}\label{eq:linear-sub}
    \lc n' \in N_\bR \relmid 
    \la k-w, n'\ra=-\frac{1}{2 \pi} \arg \lb -\frac{c_k}{c_w} \rb, \forall k \in K_\kappa^n \rc 
    \subset N_\bR
\end{align}
by the natural projection $N_\bR \to \left. N_\bR \middle/ N \right.$.
This is a subtorus of $\left. N_\bR \middle/ N \right.$.
We also set
\begin{align}
U_t:=
\lc (n, n') \in R_t \times \lb \left. N_\bR \middle/ N \right. \rb
\relmid 
n' \in T_n \rc.
\end{align}
By the condition (2) of \pref{lm:delta2}, the map $\delta_{2, t}$ induces
\begin{align}
    \delta_{2, t} \colon U_t \to N_\bR \times N_\bR.
\end{align}
By abuse of notation, we also write this induced map as $\delta_{2, t}$.
We define
\begin{align}
r_{1, t} &\colon \partial \nabla_w \times \lb \left. N_\bR \middle/ N \right. \rb 
    \to N_\bR \times \lb \left. N_\bR \middle/ N \right. \rb, \quad
    (n, n') \mapsto \lb n+\delta_{1, t} (n), n' \rb\\
r_{2, t} &\colon U_t \to N_\bR \times \lb \left. N_\bR \middle/ N \right. \rb, \quad (n, n') \mapsto \lb n, n' \rb+\delta_{2, t} \lb n, n'\rb.
\end{align}
\begin{lemma}\label{lm:image-d1}
The images of
    \begin{enumerate}
    \item $C(\scE):=\Phi(\tilde{c}(\scE)) \subset \partial \nabla_w^{d-k} \times (\left. N_\bR \middle/ N \right.)$,
    \item $\Phi \lb \Delta_\scS \times c_{\scS} \rb 
    \subset \partial \nabla_w \times (\left. N_\bR \middle/ N \right.)$ ($\Delta_\scS \times c_{\scS}$ is the one that appeared in \eqref{eq:t-c(scE)}.), and
    \item $\Phi \lb c(P) \rb \subset \partial \nabla_w \times (\left. N_\bR \middle/ N \right.)$ ($c(P)$ is \eqref{eq:c(P)}, and we regard it as a subset of $\partial \nabla_w \times (\left. N_\bR \middle/ N \right.)$ by the projection $N_\bR \times N_\bR \to N_\bR \times \lb \left. N_\bR \middle/ N \right. \rb$.)
\end{enumerate}
by the map $r_{1, t}$ are all contained in $U_t$.
\end{lemma}
\begin{proof}
We first prove the claim for $\Phi \lb c(P)\rb$ of (3).
The singular chain $c(P)$ is a sum of $\Delta_\scS \times P_{\scS}$.
Let $(n, n') \in \Delta_\scS \times P_{\scS} \subset c(P)$ be an arbitrary point.
We will prove $r_{1, t} \lb \Phi (n , n')\rb \in U_t$.
It is obvious from \pref{lm:delta1}(1) that the first component $n+\delta_{1, t}(n)$ of $r_{1, t} \lb \Phi (n , n')\rb$ is contained in $R_t$.
We will check that the second component $n'+\phi(n)$ of $r_{1, t} \lb \varphi (n , n')\rb$ is contained in $T_{n+\delta_{1, t}(n)}$.
It suffices to show
\begin{align}\label{eq:component2}
    \la k-w, n'+\phi(n)\ra=-\frac{1}{2 \pi} \arg \lb -\frac{c_k}{c_w} \rb 
\end{align}
for all $k \in K_\kappa^{n+\delta_{1, t}(n)}$.

The simplex $\Delta_\scS$ is contained in the open star of the face of $\nabla_w$ which is dual to the cone $\sigma[\scS] \in \Sigma_w$.
In the open star, the tropical monomial $(\val(k_w)+\la w, \bullet \ra)$ in $\trop(f)$ corresponding to $w (\in A)$ attains the minimum of $\trop(f)$, and the other tropical monomials that may attain the minimum are only the ones corresponding to the elements in
\begin{align}
    M_\scS:=\lc m \in A_w \relmid \bR_{\geq0} \cdot (m-w) \prec \sigma[\scS] \rc.
\end{align}
Since $\delta_{1, t}$ is small (\pref{lm:delta1}(2)), one has $K_\kappa^{n+\delta_{1, t}(n)} \subset M_\scS$ when $\kappa$ is sufficiently small.
Since we have
\begin{align}
    \sigma[\scS]^\perp 
    =\bigcap_{m \in M_\scS} (m-w)^\perp 
    \subset \bigcap_{k \in K_\kappa^{n+\delta_{1, t}(n)}} (k-w)^\perp
\end{align}
and $n' \in P_{\scS} \subset \left. \sigma[\scS]^\perp \middle/ \lb \sigma[\scS]^\perp \cap N \rb \right.$, we obtain 
\begin{align}\label{eq:kw1}
    \la k-w, n'\ra=0
\end{align}
for all $k \in K_\kappa^{n+\delta_{1, t}(n)}$.
Since $\delta_{1, t}$ is small (\pref{lm:delta1}(2)), we also have $K_\kappa^{n+\delta_{1, t}(n)} \subset K_{2\kappa}^{n}$.
By this and \eqref{eq:phase-shift}, we obtain
\begin{align}\label{eq:kw2}
    \la k-w, \phi(n)\ra=-\frac{1}{2 \pi} \arg \lb -\frac{c_k}{c_w} \rb
\end{align}
for all $k \in K_\kappa^{n+\delta_{1, t}(n)}$.
By \eqref{eq:kw1} and \eqref{eq:kw2}, we obtain \eqref{eq:component2} and conclude $r_{1, t} \lb \Phi (c(P)) \rb \subset U_t$.

The claim for $\Phi \lb \Delta_\scS \times c_{\scS}\rb$ of (2) can be checked similarly.
The singular chain $c_{\scS}$ is also contained in $\left. \sigma[\scS]^\perp \middle/ \lb \sigma[\scS]^\perp \cap N \rb \right.$ by \pref{lm:cscds}(1).
By using this, we can obtain 
$r_{1, t} \lb \Phi \lb \Delta_\scS \times c_{\scS}\rb \rb \subset U_t$ 
by the same argument as that for $r_{1, t} \lb \Phi (c(P)) \rb \subset U_t$.

Lastly, we prove the claim for (1), i.e., $r_{1, t} \lb C(\scE)\rb \subset U_t$.
We have
\begin{align}
\partial \nabla_w^{d-k}
=\bigcup_{\scS \in \scrS_{d-k}^w} 
\Delta_\scS
=
\bigcup_{q=0}^{d-k}
\bigcup_{\scS \in \scrS_{q}^w} 
\mathring{\Delta}_\scS.
\end{align}
For any point $n \in \mathring{\Delta}_\scS$ 
($\scS \in \scrS_{q}^w$, $0 \leq q \leq d-k$), 
we have $K_\kappa^{n+\delta_{1, t}(n)} \subset M_\scS$ again by the same reason as above.
By this and \pref{lm:t-c(scE)}(2), we can obtain 
$r_{1, t} \lb C(\scE)\rb \subset U_t$ by the same argument as the previous ones for $r_{1, t} \lb \Phi \lb c(P)\rb \rb$ and 
$r_{1, t} \lb \Phi \lb \Delta_\scS \times c_{\scS}\rb \rb$.
\end{proof}

We show that the image of $C(\scE)$ by the map $r_{2, t} \circ r_{1, t}$ gives the small perturbation of $C(\scE)$ in \pref{th:main2}(2-c).
The image $(r_{2, t} \circ r_{1, t})(C(\scE))$ is well-defined by \pref{lm:image-d1}(1).
By \pref{lm:delta2}(1) and \pref{lm:delta1}(1), we have
\begin{align}\label{eq:in-hypersurface}
f_t^w \lb h_t \lb r_{2, t} \circ r_{1, t} \lb n, n' \rb \rb \rb
=
f_t^w \lb h_t \lb (n+\delta_{1, t}(n), n')+\delta_{2, t}(n+\delta_{1, t}(n), n') \rb \rb
=
\tilde{f}_t^w \lb i_t(n+\delta_{1, t}(n)) \rb=1
\end{align}
for all $\lb n, n' \rb \in C(\scE)$.
Therefore, the image by the map $h_t$ of $(r_{2, t} \circ r_{1, t})(C(\scE))$ is contained in the hypersurface $\mathring{Z}_t$.
For each element $(n, n') \in C(\scE)$,
\begin{align}\label{eq:added}
    (\delta_{1, t}(n), 0)+\delta_{2, t} \lb n+\delta_{1, t}(n), n'\rb
\end{align}
is added by the map $r_{2, t} \circ r_{1, t}$.
By \pref{lm:delta1}(2) and \pref{lm:delta2}(3), we can see that the $C^0$-norm of \eqref{eq:added} is $O \lb(-\log t)^{-1} \rb$.
The group homomorphism $\Psi_t^w$ of \eqref{eq:Psi} will be defined in \eqref{eq:Psi'} (\pref{sc:proof-main2}) so that it is obvious that the image by the map $h_t$ of $(r_{2, t} \circ r_{1, t})(C(\scE))$ represents the homology class $\Psi_t^w (\scE) \in H_d ( \mathring{Z}_t, \bZ)$.

\subsection{Homology classes of lifts}\label{sc:homology}

Let $P \subset N_\bR$ be a lattice polytope of \pref{sc:step1}.
We define $C(P)_t \subset N_{\bC^\ast}$ to be the image of $\Phi(c(P))$ by the map $h_t \circ r_{2, t} \circ r_{1, t}$.
This is also well-defined by \pref{lm:image-d1}(3), and is contained in the hypersurface $\mathring{Z}_t$ by the same computation as \eqref{eq:in-hypersurface}.
Furthermore, we can see that $C(P)_t$ is a cycle by \pref{lm:c(P)-boundary}. 

\begin{construction}\label{transport}
Let $\tilde{Z}_t^w$ be the complex hypersurface defined by the polynomial $\tilde{f}_t^w$ of \eqref{eq:tft} in the toric variety $Y_\Sigma$.
The hypersurfaces $Z_t$ and $\tilde{Z}_t^w$ are members of the family of complex hypersurfaces defined by the second projection
\begin{align}\label{eq:ch-family}
\lc \lb z, \lc a_m \rc_{m \in A \setminus \lc w \rc} \rb \in Y_\Sigma \times \lb \bC^\ast \rb^{A \setminus \lc w \rc} \relmid \sum_{m \in A \setminus \lc w \rc} a_m z^{m-w}=1 \rc 
\to \lb \bC^\ast \rb^{A \setminus \lc w \rc}.
\end{align}
The positive real locus $\tilde{Z}_t^w \cap N_{\bR_{>0}} \subset \tilde{Z}_t^w$ is homeomorphic to a $d$-sphere (cf.~\cite[Section 6.7]{MR2079993}).
We choose the branch of the argument $\arg \lb -k_{m, t}/k_{w, t} \rb$ of every coefficient in the polynomial $f_t^w$ so that
\begin{align}\label{eq:arg-arg}
    \arg \lb -k_{m, t}/k_{w, t} \rb \sim 
    \left\{ 
    \begin{array}{ll}
    \arg \lb - c_m/c_w \rb-2\pi \cdot \varphi_P(m-w) & m \in A_w\\
    \arg \lb - c_m/c_w \rb & m \in A \setminus (A_w \cup \lc w \rc)
    \end{array} 
    \right.
    \quad (t \to +0),
\end{align}
where the branch of $\arg \lb - c_m/c_w \rb$ is the one we chose for constructing the phase-shifting function $\phi$ in \pref{eq:phase-shift}, 
and $\varphi_P$ is the support function \eqref{eq:varphi} for the lattice polytope $P$.
We consider transporting the positive real locus $\tilde{Z}_t^w \cap N_{\bR_{>0}} \subset \tilde{Z}_t^w$ into $Z_t$ inside \eqref{eq:ch-family} by varying the complex coefficients of the polynomial defining the hypersurface from $\tilde{f}_t^w$ to $f_t^w$ 
so that the argument $\arg \lb r_m/r_w \rb=0$ of every coefficient in $\tilde{f}_t^w$ changes to the argument $\arg \lb -k_{m, t}/k_{w, t} \rb$ of the corresponding coefficient in $f_t^w$ continuously.
\end{construction}

\begin{proposition}\label{pr:C(P)-transport}
    The cycle $C(P)_t$ is the same as the cycle obtained by the transportation of Construction \ref{transport}.
\end{proposition}

\begin{proof}
 We give $\partial \nabla_w$ the orientation\footnote{This orientation coincides with that of $B_t^w (\cong \partial \nabla_w)$ \eqref{eq:Btw}, which we took for computation of period integrals in \cite{MR4782805}. See Remark 5.0.1 in loc.cit.} defined by the interior product of $\vol (N)$ with an incoming normal vector field on it.
 The cycle $c(P)$ \eqref{eq:c(P)} is written as
    \begin{align}\label{eq:c(P)'}
        c(P)=
        \sum_{\scS \in \scrS_d^w} (-1)^d \cdot \Delta_\scS \times P_{\scS}
        +\sum_{q=0}^{d-1} \sum_{\substack{\scS \in \scrS_q^w \\ \dim P_\scS=d-q}} (-1)^q \cdot \Delta_\scS \times P_{\scS}.
    \end{align}
    Regarding the former part of \eqref{eq:c(P)'}, we have $f(\scS) \in \lc \pm 1 \rc$ for $\scS \in \scrS_d^w$, since the fan $\Sigma_w$ is unimodular.
    We can easily check that the orientation of $\Delta_\scS \subset \partial \nabla_w$ $(\scS \in \scrS_d^w)$ multiplied by $(-1)^d \cdot f(\scS)$ coincides with that of $\partial \nabla_w$.
For every $\sigma \in \Sigma_w(d+1)$, we set 
     \begin{align}\label{eq:delta-sigma}
        \Delta_\sigma:=\sum_{\substack{\scS \in \scrS_d^w \\ \sigma[\scS]=\sigma}} 
          (-1)^d \cdot f(\scS) \cdot \Delta_\scS.
    \end{align}
    The polytope $P_{\scS}$ $(\scS \in \scrS_d^w)$ is, as a set, a single point which depends only on the cone $\sigma[\scS] \in \Sigma_w(d+1)$.
     We write the point as $n(P, \sigma) \in N_\bR$ with $\sigma=\sigma[\scS] \in \Sigma_w(d+1)$.
     The orientation of the point $P_{\scS}$ $(\scS \in \scrS_d^w)$ is positive when $f(\scS)=1$, and negative when $f(\scS)=-1$.
Therefore, the former part of \eqref{eq:c(P)'} can be written as
    \begin{align}\label{eq:c(P)1}
        \sum_{\sigma \in \Sigma_w(d+1)} 
 \Delta_\sigma \times n(P, \sigma) \subset N_\bR \times N_\bR.
        \end{align}
        \eqref{eq:c(P)1} can be regarded as the space obtained by translating each piece $\Delta_\sigma \times \lc 0\rc$ of the decomposition of $\partial \nabla_w\times \lc 0\rc$
        \begin{align}\label{eq:decomp}
        \partial \nabla_w \times \lc 0 \rc=
        \sum_{\sigma \in \Sigma_w(d+1)} \Delta_\sigma \times \lc 0 \rc 
        \subset N_\bR \times N_\bR
    \end{align}
    by the vector $n(P, \sigma) \in \lc 0\rc \times N_\bR$, and we can see from \pref{lm:c(P)-boundary} that the latter part of \eqref{eq:c(P)'} connect together these pieces $\Delta_\sigma \times n(P, \sigma)$ dispersed by the translations.

    Before considering $C(P)_t:=h_t \circ r_{2, t} \circ r_{1, t} \lb \Phi(c(P)) \rb$ for general $P \subset N_\bR$,
    let us consider the cycle $C(P)_t$ with $P=\lc 0 \rc$, i.e., $C(\lc 0 \rc)_t:=h_t \circ r_{2, t} \circ r_{1, t} \lb \Phi(\partial \nabla_w \times \lc 0 \rc) \rb$.
    It is the image of \eqref{eq:decomp} by $h_t \circ r_{2, t} \circ r_{1, t} \circ\Phi$.
    By \pref{lm:delta1}(1), the first component of 
    $r_{1, t} \lb \Phi (\partial \nabla_w  \times \lc 0 \rc) \rb
    \subset N_\bR \times N_\bR$ is contained in 
    \begin{align}\label{eq:Btw}
        B_t^w:=\lc n \in N_\bR \relmid \tilde{f}_t^w(i_t(n))=1 \rc.
    \end{align}
    The map $\Phi$ \eqref{eq:shift-map} corresponds to the map $\Phi_t$ in \cite[Section 5.2]{MR4194298} and $\Phi_t$ with $x=0$ in \cite[Section 3]{MR4782805}.
We can see that the set $r_{1, t} \lb \Phi (\partial \nabla_w  \times \lc 0 \rc) \rb$ corresponds to the set $\Phi_t(B_t)$ in \cite[Section 5.2]{MR4194298} and $\Phi_t(B_t^w \times \lc 0 \rc)$ in \cite[Section 3]{MR4782805}.
    Furthermore, the map $r_{2, t}$ corresponds to adding $\delta_t$ of \cite[Proposition 5.3]{MR4194298} and \cite[Proposition 3.0.5]{MR4782805}.
    From these, it is clear that the cycle $C(\lc 0 \rc)_t$ coincides with the cycle constructed in \cite[Section 5.2]{MR4194298} and \cite[Section 3]{MR4782805}. 
    The cycle constructed in \cite[Section 5.2]{MR4194298} and \cite[Section 3]{MR4782805} further coincides with the cycle obtained by the transportation in Construction \ref{transport} with $\varphi_P$ in \eqref{eq:arg-arg} replaced by the constant zero function.
    This is because the phase shift $\Phi$ corresponds to the changes of the arguments of the coefficients of leading terms of $f_t^w$ when transporting from $\tilde{Z}_t^w$ to $Z_t$.
    We can see this from the computation in \eqref{eq:argument-phase} and the condition \eqref{eq:phase-shift} for the phase-shifting function $\phi$, which also appears in the definition \eqref{eq:S_t} of $S_t$.
    
    Now we consider $C(P)_t:=h_t \circ r_{2, t} \circ r_{1, t} \lb \Phi(c(P)) \rb$ for general $P \subset N_\bR$.
    Each translation of $\Delta_\sigma \times \lc 0\rc$ by the vector $n(P, \sigma)$ considered in \eqref{eq:c(P)1}-\eqref{eq:decomp} amounts to a constant phase shift via the map $h_t$, and the collection of these translations gives rise to, so to speak, a ``discrete" phase shift of $\partial \nabla_w$.
By the definition \eqref{eq:face-S} of $P_{\scS}$, we have
\begin{align}\label{eq:nDs}
    \la m-w, n(P, \sigma) \ra=\varphi_{P}(m-w)
\end{align}
for any $m \in A_w$ such that $\bR_{\geq 0} \cdot (m-w) \prec \sigma$.
Tropical monomials of $\trop(f)$ that may attain the minimum of $\trop(f)$ on $\Delta_\sigma$ are the monomials corresponding to $w \in A$ or $m \in A_w$ such that $\bR_{\geq 0} \cdot (m-w) \prec \sigma$.
Recall that we saw from the computation in \eqref{eq:argument-phase} and the condition \eqref{eq:phase-shift} that the phase shift $\Phi$ corresponds to the changes of the arguments of the coefficients of leading terms of $f_t^w$ when transporting from $\tilde{Z}_t^w$ to $Z_t$.
Similarly, we can see from \eqref{eq:nDs} that the discrete phase shift given by translations of $\Delta_\sigma \times \lc 0\rc$ by the vector $n(P, \sigma)$ also corresponds to increasing 
the arguments of the coefficients $-k_{m, t}/k_{w, t}$ of leading terms of $f_t^w$ by $-2\pi \cdot \varphi_P(m-w)$.
By combining this effect from the discrete phase shift and the phase shift of $\Phi$, we can see that the cycle $C(P)_t$ is exactly the cycle obtained by the transportation of Construction \ref{transport}.
We conclude \pref{pr:C(P)-transport}.
\end{proof}

\begin{example}\label{eg:curve}
Let $d=1$.
Take a basis $\lc e_1, e_2 \rc$ of $M \cong \bZ^2$, and let $\lc e_1^\ast, e_2^\ast \rc$ be the dual basis of $N:=\Hom(M, \bZ)$.
We set $e_0:=-(e_1+e_2) \in M$.
We consider the polynomial
\begin{align}\label{eq:f-theta}
    f_{t, \theta}(z):=-1+\sum_{m \in (\Delta \cap M) \setminus \lc 0\rc} e^{\sqrt{-1} \theta_m}t^{\lambda_m} z^m 
    \in \bC \ld M \rd
\end{align}
with
\begin{align}
    \Delta:=\conv \lb \lc e_0, e_1, e_2\rc \rb \subset M_\bR, \quad 
    \theta_m \in \bR, \quad
\lambda_m:=
\left\{ \begin{array}{ll}
1 & m=e_0 \\
0 & m \in \lc e_1, e_2 \rc.
\end{array} 
\right. 
\end{align}
We have $W=\lc 0 \rc$.
The fan $\Sigma_{w=0}$ is the fan of $\bP^2$, and consists of three $2$-dimensional cones
\begin{align}
    \sigma_i:=\cone \lb  \lc e_0, e_1, e_2 \rc \setminus \lc e_i \rc \rb 
    \quad (i \in \lc 0, 1, 2 \rc)
\end{align}
and their faces.
\pref{fg:phase} shows the decomposition of $\partial \nabla_{w=0}$ of \eqref{eq:decomp}.
The set $\partial \nabla_{w=0}$ is the union of the colored lines, and the subsets $\Delta_{\sigma_i} \subset \partial \nabla_{w=0}$ $(i \in \lc 0, 1, 2 \rc)$ of \eqref{eq:delta-sigma} are the unions of the red, blue, green lines respectively.

\begin{figure}[htbp]
	\centering
	\includegraphics[scale=0.5]{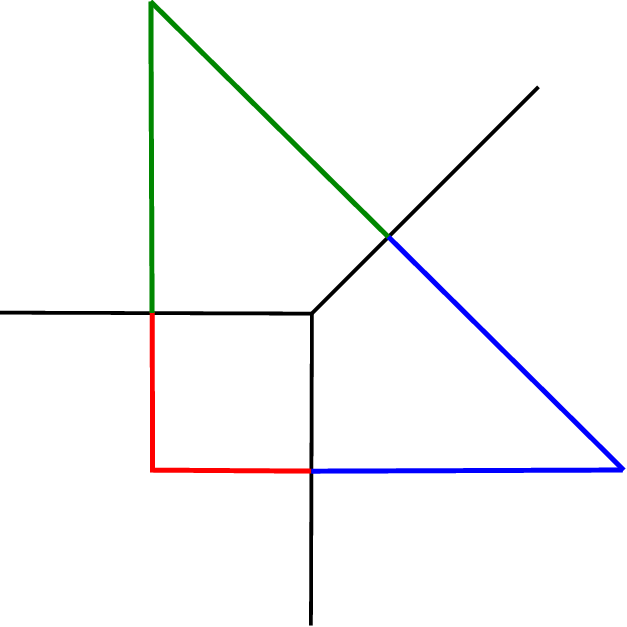}
    \caption{The decomposition of $\partial \nabla_{w=0}$ and the fan $\check{\Sigma}$}
\label{fg:phase}
\end{figure}

Let $a_i$ $(i \in \lc 0, 1, 2 \rc)$ be integers such that $-(a_0+a_1+a_2) > 0$.
We consider the piecewise linear function $\varphi_P \colon M_\bR \to \bR$ on the fan $\Sigma_{w=0}$ determined by $\varphi_P(e_i)=a_i$ $(i \in \lc 0, 1, 2 \rc)$.
This is the support function for the lattice polytope
\begin{align}\label{eq:P}
    P:=\conv \lb \lc v_0, v_1, v_2 \rc \rb \subset N_\bR
\end{align}
with
\begin{align}
    v_0:=a_1 \cdot e_1^\ast+a_2 \cdot e_2^\ast, \quad 
    v_1:=-(a_0+a_2) \cdot e_1^\ast+a_2 \cdot e_2^\ast, \quad 
    v_2:=a_1 \cdot e_1^\ast-(a_0+a_1) \cdot e_2^\ast.
\end{align}
The translation vectors $n(P, \sigma_i) \in N_\bR$ that we considered in the proof of \pref{pr:C(P)-transport} are
\begin{align}
    n(P, \sigma_i)=v_i \quad (i \in \lc 0, 1, 2 \rc).
\end{align}
The set of sequences of cones $\scrS_{q=0}^{w=0}$ consists of
\begin{align}
    \scS_i:=\lc \bR_{\geq 0} \cdot e_i \rc \quad (i \in \lc 0, 1, 2\rc),
\end{align}
and the faces $P_{\scS_i} \prec P$ are
\begin{align}
    P_{\scS_i}=\conv \lb \lc v_0, v_1, v_2 \rc \setminus \lc v_ i \rc \rb \quad (i \in \lc 0, 1, 2 \rc).
\end{align}
We have
\begin{align}
    c(P)=
    \sum_{i=0}^2 \Delta_{\sigma_i} \times n(P, \sigma_i)+
    \sum_{i=0}^2 \Delta_{\scS_i} \times P_{\scS_i},
\end{align}
and for any permutation $\lc i_0, i_1, i_2 \rc$ of $\lc 0, 1, 2\rc$, the dispersed pieces $\Delta_{\sigma_{i_1}} \times n(P, \sigma_{i_1})$ and $\Delta_{\sigma_{i_2}} \times n(P, \sigma_{i_2})$ are connected by $\Delta_{\scS_{i_0}} \times P_{\scS_{i_0}}$ (\pref{lm:c(P)-boundary}).

Consider the fan $\check{\Sigma}$ in $N_\bR$ that is dual to $\partial \nabla_{w=0}$ and divides $\partial \nabla_{w=0}$ into $\Delta_{\sigma_i}$ $(i \in \lc 0, 1, 2 \rc)$ as shown by black lines in \pref{fg:phase}.
It has three $1$-dimensional cones $\rho_i$ $(i \in \lc 0, 1, 2\rc)$ which are parallel to $e_1^\ast+e_2^\ast$, $-e_1^\ast$, and $-e_2^\ast$ respectively. 
Let $\check{\varphi} \colon N_\bR \to \bR$ be the piecewise linear function on the fan $\check{\Sigma}$, which takes the value $-a_i$ at the primitive generator of $\rho_i$ for every $i \in \lc 0, 1, 2\rc$.
In \cite[Example 5.4]{MR4194298}, a phase-shifting function is constructed by taking the gradient of a smoothing of the piecewise linear function $\check{\varphi} \colon N_\bR \to \bR$.
This phase shift corresponds to increasing each phase $\theta_{m=e_i}$ $(i \in \lc 0, 1, 2 \rc)$ in \eqref{eq:f-theta} by $-2 \pi \cdot a_i=-2 \pi \cdot \varphi_P(e_i)$ continuously, and this is exactly the transportation in Construction \ref{transport} for the polytope $P$ of \eqref{eq:P}.
See \eqref{eq:arg-arg}.

On the other hand, for any $i \in \lc 0, 1, 2\rc$, the gradient of the piecewise linear function $\check{\varphi}$ on the $2$-dimensional cone in $\check{\Sigma}$ not containing $\rho_i$ is $v_i=n(P, \sigma_i)$.
Therefore, translating each piece $\Delta_{\sigma_i}$ by $n(P, \sigma_i)$ amounts to the phase shift constructed by using the gradient of the piecewise linear function $\check{\varphi}$ (without taking its smoothing) as a phase-shifting function.
It is now clear that the cycle $C(P)_t$ is the same as the cycle obtained by the transportation in Construction \ref{transport} for the polytope $P$ of \eqref{eq:P}.
\end{example}

Let $C(\scE)_t \subset \mathring{Z}_t$ be the image by the map $h_t$ of the small  perturbation of $C(\scE)$ constructed in \pref{sc:step2}.

\begin{proposition}\label{pr:homologous}
The cycle $C(\scE)_t \subset \mathring{Z}_t$ is homologous to the cycle
\begin{align}\label{eq:homologous}
\sum_{j \in J} l_j \cdot C \lb P_{j} \rb_t,
\end{align}
where the integers $l_j$ and the polytopes $P_{j}$ $(j \in  J)$ are those of the expression \eqref{eq:scE} of $\scE \in K(Y_w)$.
\end{proposition}
\begin{proof}
One has
\begin{align}
C(\scE)_t
&=h_t \circ r_{2, t} \circ r_{1, t} \lb \Phi(\tilde{c}(\scE)) \rb\\ \label{eq:homologous'}
&=h_t \circ r_{2, t} \circ r_{1, t} \lb \Phi\lb 
\sum_{j \in J} l_j \cdot c(P_j)
-\sum_{q=d+1-k}^d \sum_{\scS \in \scrS_q^w}
\partial \lb \Delta_\scS \times c_{\scS} \rb, \rb \rb,
\end{align}
by \eqref{eq:t-c(scE)} and \eqref{eq:c(scE)}.
We can see from \pref{lm:image-d1}(2), \pref{lm:delta1}(1), and \pref{lm:delta2}(1) that 
$h_t \circ r_{2, t} \circ r_{1, t} 
\lb \Phi \lb \Delta_\scS \times c_{\scS} \rb \rb$ is a singular chain in the hypersurface $\mathring{Z}_t$.
Therefore, \eqref{eq:homologous'} is homologous to
\begin{align}
h_t \circ r_{2, t} \circ r_{1, t} \lb \Phi \lb \sum_{j \in J} l_j \cdot c(P_j) \rb \rb
=\sum_{j \in J} l_j \cdot \lb h_t \circ r_{2, t} \circ r_{1, t} \lb \Phi \lb c(P_j) \rb \rb \rb,
\end{align}
which equals \eqref{eq:homologous}.
Thus we can conclude the proposition.
\end{proof}

\subsection{End of the proof of \pref{th:main2}}\label{sc:proof-main2}

We define the map \eqref{eq:Psi} by
\begin{align}\label{eq:Psi'}
    \Psi_t^w \colon K \lb Y_w \rb \to H_d \lb \mathring{Z}_t, \bZ \rb, 
    \quad \scE \mapsto \ld C(\scE)_t \rd.
\end{align}
We show that the homology class $\ld C(\scE)_t \rd$ does not depend on the expression \eqref{eq:scE} of $\scE$, and that the map \eqref{eq:Psi'} is well-defined.
For any $\scE \in K(Y_w)$, let $C(\scE)_t$ and $C(\scE)_t'$ be the cycles constructed by using two different expressions of $\scE$
\begin{align}
    \scE=\sum_{j \in J} l_j \cdot [L_j] =\sum_{j' \in J'} l_{j'} \cdot [L_{j'}],
\end{align}
where $l_j,l_{j'} \in \bZ$, and $L_j, L_{j'}$ are anti-nef line bundles on the toric variety $Y_w$.
For each $j \in J$ (resp. $j' \in J'$), let 
$P_{j} \subset N_\bR$ (resp. $P_{j'} \subset N_\bR$) be a lattice polytope whose associated support function defines the nef line bundle $L_j^{-1}$ (resp. $L_{j'}^{-1}$).
Then the zero element $0 \in K(Y_w)$ is expressed as
\begin{align}\label{eq:exp-0}
    0=\sum_{j \in J} l_j \cdot [L_j] -\sum_{j' \in J'} l_{j'} \cdot [L_{j'}].
\end{align}
The integer $k$ of \eqref{eq:k} for $0 \in K(Y_w)$ is $d$, and the Minkowski weight $a_\scE$ in \pref{th:main2}(2) with $\scE=0$ is zero.
We can see from \pref{lm:t-c(scE)}(1), (3) that the cycle $\tilde{c}(\scE)$ with $\scE=0$ associated with the expression \eqref{eq:exp-0} of $0 \in K(Y_w)$ is supported over $\partial \nabla_w^0$, and the fibers over the points of $\partial \nabla_w^0$ are all homologous to zero.
This implies that the cycle $C(\scE)_t$ with $\scE=0$ associated with the expression \eqref{eq:exp-0} is trivial.
Therefore, by using \pref{pr:homologous}, we can obtain
\begin{align}
0=\ld \sum_{j \in J} l_j \cdot C \lb P_{j} \rb_t-\sum_{j' \in J'} l_{j'} \cdot C \lb P_{j'} \rb_t\rd
=\ld C(\scE)_t \rd-\ld C(\scE)_t' \rd,
\end{align}
and see that the homology class $\ld C(\scE)_t \rd$ does not depend on the expression \eqref{eq:scE}, and \eqref{eq:Psi'} is well-defined.
We can also easily see from \pref{pr:homologous} that \eqref{eq:Psi'} is a group homomorphism.

Lastly, we show \pref{th:main2}(1).
Since $\Psi_t^w$ and $\ch(\bullet)$ are homomorphisms and $K(Y_w)$ is generated by anti-ample line bundles (\pref{pr:iritani}),
it suffices to show the asymptotic formulas of period integrals \eqref{eq:period-formula-1} and \eqref{eq:period-formula-2} in the case where $\scE \in  K(Y_w)$ is represented by an anti-ample line bundle $L$ on $Y_w$.
Let $P \subset N_\bR$ be a lattice polytope whose associated support function defines the ample line bundle $L^{-1}$.
Then the  homology class $\Psi_t^w (\scE)$ is the same as $[C(P)_t]$ by \pref{pr:homologous}.
The asymptotics of period integrals for the cycle $C(P)_t$ is computed in \cite[Theorem 1.1.2]{MR4782805}.
By \pref{pr:C(P)-transport}, we can obtain them by replacing $\arg \lb - c_m/c_w \rb$ in
\begin{align}
    \lb
-\frac{c_m}{c_{w}}
\rb^{-D_m^w}
:=\exp \lb -D_m^w \log \lb -\frac{c_m}{c_{w}} \rb \rb
=\exp \lb -D_m^w 
\lb \log \left| -\frac{c_m}{c_{w}} \right|+\sqrt{-1} \arg \lb -\frac{c_m}{c_{w}} \rb \rb \rb
\end{align}
of \cite[Theorem 1.1.2]{MR4782805} with $\arg \lb - c_m/c_w \rb-2\pi \cdot \varphi_P(m-w)$.
When $\conv \lb \lc w \rc \cup \tau_v \rb \in \scrT$, we obtain the asymptotics
\begin{align}\label{eq:pre-period}
\int_{C(P)_t}  \Omega_t^{l, v} 
&=
\frac{(-1)^{d+p_w}}{(l-1)!}
\int_{Y_{w}} 
t^{-\omega_{\lambda}^{w}}
\cdot 
\widehat{\Gamma}_w
\cdot 
E_{v, w}
\cdot
\prod_{m \in A_{w}} 
\lb
-\frac{c_m}{c_{w}}
\rb^{-D_m^w}
\exp \lb 2 \pi \sqrt{-1} \varphi_{P}(m-w) D_m^w \rb
+O \lb t^\epsilon \rb
\end{align}
as $t \to +0$. 
We have 
\begin{align}
\exp \lb 2 \pi \sqrt{-1} \varphi_{P}(m-w) D_m^w \rb
=(2 \pi \sqrt{-1})^{\deg/2} \ch(L)
=(2 \pi \sqrt{-1})^{\deg/2} \ch(\scE),
\end{align}
and we can see that \eqref{eq:pre-period} equals \eqref{eq:period-formula-1}.
When $\conv \lb \lc w \rc \cup \tau_v \rb \nin \scrT$, we also obtain
\begin{align}\label{eq:period2}
\int_{C(P)_t}  \Omega_t^{l, v} 
=
O \lb t^\epsilon \rb
\end{align}
as $t \to +0$. 
Thus we can conclude \pref{th:main2}(1).
Since \pref{th:main2}(2) has been confirmed in \pref{sc:step1-3} and \pref{sc:step2}, we can conclude \pref{th:main2}.

\section{Proofs of \pref{th:main3} and \pref{cr:CY}}\label{sc:proof3}

\subsection{Proof of \pref{th:main3}}

Recall that we have $A^{d-q}(Y_w) \cong \MW^{d-q}(\Sigma_w)$ (cf.~\eqref{eq:amw}).
Since the Chow cohomology group $A^{d-q}(Y_w)$ is generated by the toric strata of the toric variety $Y_w$ (cf.~e.g.~\cite[Lemma 12.5.1]{MR2810322}), an arbitrary non-zero element of $A^{d-q}(Y_w)$ can be written as
\begin{align}\label{eq:a-class}
    \sum_{\sigma \in \Sigma_w(d-q)} c_\sigma \cdot [D_{\sigma, 1}] \cdots [D_{\sigma, d-q}] \in A_{q+1}(Y_w)=A^{d-q}(Y_w),
\end{align}
where $c_\sigma \in \bZ$, and $D_{\sigma, i}$ $(1 \leq i \leq d-q)$ are the toric divisors in the toric variety $Y_w$ corresponding to the $1$-dimensional faces of the cone $\sigma \in \Sigma_w(d-q)$.
For proving \pref{th:main3}, it suffices to prove that there exists an element 
$\scE \in K(Y_w)$ 
such that the integer $k$ of \eqref{eq:k} for $\scE$ is $d-q$, and $(-1)^{q+1} \cdot \ch_{d-q}(\scE)$ equals \eqref{eq:a-class}.

Let $D_0$ be an ample toric divisor in the toric variety $Y_w$.
For any divisor $D_{\sigma, i}$ appearing in \eqref{eq:a-class}, there exists a natural number $l$ such that $D_{\sigma, i}+l\cdot D_0$ is ample.
Therefore, we can write 
\begin{align}
D_{\sigma, i}=D_{\sigma, i}^+-D_{\sigma, i}^-
\end{align}
with ample toric divisors $D_{\sigma, i}^+(=D_{\sigma, i}+l\cdot D_0), D_{\sigma, i}^-(=l\cdot D_0)$.
Therefore, the class \eqref{eq:a-class} can also be written as a finite sum
\begin{align}\label{eq:a-class2}
    \sum_{j} c_j \cdot [D_{j, 1}] \cdots [D_{j, d-q}]
\end{align}
with integers $c_j \in \bZ$ and ample toric divisors $D_{j, i}$ $(1 \leq i \leq d-q)$ in the toric variety $Y_w$.
Let $Y_{j, i} \subset Y_w$ be the closed subscheme defined by a general section of $\scO_{Y_w}(D_{j, i})$, and set 
$X_{j, 0}:=Y_w$, $X_{j, s}:=\bigcap_{i=1}^{s} Y_{j, i}$ $(1 \leq s \leq d-q)$.
We consider
\begin{align}
    \scE:=(-1)^{q+1} \cdot \lb \sum_j c_j \cdot \ld \scO_{X_{j, d-q}} \rd \rb \in K(Y_w).
\end{align}
We will prove that the integer $k$ of \eqref{eq:k} for this $\scE$ is $d-q$, 
and $(-1)^{q+1} \cdot \ch_{d-q} \lb \scE \rb$ equals \eqref{eq:a-class2}.

In order to prove it, we compute $\ch(\scE)$.
First, we claim
\begin{align}\label{eq:ch}
    \ch \lb \scO_{X_{j, s}} \rb
    =\sum_{I \subset \lc 1, \cdots, s\rc} (-1)^{\# I} \cdot \ch \lb \scO_{Y_w} \lb -D_{j, I} \rb \rb \quad (0 \leq s \leq d-q),
\end{align}
where $D_{j, I}:=\sum_{i \in I} D_{j, i}$, and the sum is taken over all subsets $I \subset \lc 1, \cdots, s\rc$, including $I=\emptyset$.
We prove \eqref{eq:ch} by induction on $s$.
\eqref{eq:ch} is trivial when $s=0$.
We suppose that \eqref{eq:ch} holds for $s=s_0 (\geq 0)$, and prove it for $s=s_0+1$.
By the exact sequence
\begin{align}
    0 \to \scO_{X_{j, s_0}} (-Y_{j, s_0+1})
    \to \scO_{X_{j, s_0}}
    \to \scO_{X_{j,  s_0+1}}\to 0
\end{align}
and the induction hypothesis, we obtain
\begin{align}
\ch \lb \scO_{X_{j, s_0+1}}\rb
    &=\ch \lb \scO_{X_{j, s_0}}\rb-\ch \lb \scO_{X_{j, s_0}} (-Y_{j, s_0+1}) \rb\\
    &=\sum_{I \subset \lc 1, \cdots, s_0\rc} (-1)^{\# I} \cdot \ch \lb \scO_{Y_w} \lb -D_{j, I} \rb \rb
    -\sum_{I \subset \lc 1, \cdots, s_0\rc} (-1)^{\# I} \cdot \ch \lb \scO_{Y_w} \lb -D_{j, I}-D_{j, s_0+1} \rb \rb\\
    &=\sum_{I \subset \lc 1, \cdots, s_0+1\rc} (-1)^{\# I} \cdot \ch \lb \scO_{Y_w} \lb -D_{j, I} \rb \rb.
\end{align}
We obtained \eqref{eq:ch} for $s=s_0+1$.
Thus we can conclude \eqref{eq:ch}.

The right-hand side of \eqref{eq:ch} is 
\begin{align}\label{eq:ch2}
\sum_{I \subset \lc 1, \cdots, s\rc} (-1)^{\# I} \cdot \exp \lb -D_{j, I} \rb
=
\sum_{r=0}^{d+1} \frac{(-1)^r}{r!} \cdot
\lb \sum_{I \subset \lc 1, \cdots, s\rc} (-1)^{\# I} \cdot \lb \sum_{i \in I} D_{j, i} \rb^r \rb.
\end{align}
Regarding the right-hand side of \eqref{eq:ch2}, we have the following formula:

\begin{lemma}\label{lm:divisor-intersection}
    Let $r$ be an integer such that $0 \leq r \leq s (\leq d-q)$.
    One has
    \begin{align}\label{eq:divisor-intersection}
        \sum_{I \subset \lc 1, \cdots, s\rc} (-1)^{\#I} 
        \cdot 
        \lb \sum_{i \in I}  D_{j, i} \rb^r
        =
        \left\{ \begin{array}{ll}
        (-1)^{s} \cdot s! \cdot D_{j, 1} \cdot \cdots \cdot D_{j, s} & r=s\\
0 & 0 \leq r \leq s-1.
\end{array} 
\right.
    \end{align}
\end{lemma}
\begin{proof}
    We prove the formula by induction on $s$.
    When $s=0$, the claim is obvious.
    We suppose that the formula holds for $s=s_0$, and prove it for $s=s_0+1$ ($0 \leq s_0 \leq  d-q-1$).
The left-hand side of \eqref{eq:divisor-intersection} with $s=s_0+1$ is written as
\begin{align}
\sum_{I \subset \lc 1, \cdots, s_0\rc}
(-1)^{\#I} \cdot
\lb \sum_{i \in I} D_{j, i} \rb^{r}
+
\sum_{I \subset \lc 1, \cdots, s_0\rc}
(-1)^{\#I+1} 
\lc
\sum_{u=0}^r \binom{r}{u} \cdot (D_{j, s_0+1})^{u} \cdot  \lb \sum_{i \in I} D_{j, i} \rb^{r-u}
\rc,
\end{align}
which is equal to
\begin{align}\label{eq:divisor-intersection1}
\sum_{I \subset \lc 1, \cdots, s_0 \rc}
(-1)^{\#I} \cdot
\lb \sum_{i \in I} D_{j, i} \rb^{r}
-
\sum_{u=0}^r \binom{r}{u} \cdot
(D_{j, s_0+1})^{u}
\cdot
\lc
\sum_{I \subset \lc 1, \cdots, s_0\rc}
(-1)^{\#I} \cdot
\lb \sum_{i \in I} D_{j, i} \rb^{r-u} 
\rc.
\end{align}
By using the induction hypothesis, we can see in \eqref{eq:divisor-intersection1} that the former and latter parts cancel out when $r=s_0$, and both parts are zero when $0 \leq r \leq s_0-1$.
When $r=s_0+1$, the latter part is non-zero only for $u=0, 1$.
The term of $u=0$ and the former part cancel out, and the term of $u=1$ is 
\begin{align}
-(s_0+1)\cdot D_{j, s_0+1} \cdot (-1)^{s_0} \cdot s_0 ! \cdot D_{j, 1} \cdot \cdots \cdot D_{j, s_0}
=
(-1)^{s_0+1} \cdot (s_0+1)! \cdot D_{j, 1} \cdot \cdots \cdot D_{j, s_0+1}.
\end{align}
Hence, \eqref{eq:divisor-intersection1} equals $(-1)^{s_0+1} \cdot (s_0+1)! \cdot D_{j, 1} \cdot \cdots \cdot D_{j, s_0+1}$.
We obtained the formula \eqref{eq:divisor-intersection} for $s=s_0+1$.
\end{proof}

We can see from \eqref{eq:ch}, \eqref{eq:ch2}, and \pref{lm:divisor-intersection} that we have
\begin{align}
    \ch_{k'} \lb \scO_{X_{j, s}} \rb
    =
    \left\{ \begin{array}{ll}
    [D_{j, 1}] \cdots [D_{j, s}] & k'=s\\
0 & k' <s.
\end{array} 
\right.
\end{align}
By this for $s=d-q$, we can see that $(-1)^{q+1} \cdot \ch_{d-q} \lb \scE \rb$ equals \eqref{eq:a-class2}, and the integer $k$ of \eqref{eq:k} for $\scE$ is $d-q$.
Thus we can conclude \pref{th:main3}.

\subsection{Proof of \pref{cr:CY}}

The formula \eqref{eq:period-formula-1'} immediately follows from \pref{th:main2}(1-a).
The latter claim of \pref{cr:CY} can also be checked as follows:
The argument in the previous subsection also works when \eqref{eq:a-class2} is replaced with 
$\ld D_1 \rd \cdots \ld D_r \rd \in A_{d+1-r}(Y_{w=0})=A^r(Y_{w=0})$ in the setting of \pref{cr:CY}, and we can see for $\scO_X$ in \eqref{eq:period-formula-1'} that the integer $k$ of \eqref{eq:k} for $\scO_X$ is $r$ (the number of divisors $D_i$), and $\ch_r(\scO_X)=\ld D_1 \rd \cdots \ld D_r \rd$.
Therefore, the cycle $C(\scO_X)_t$ in \eqref{eq:period-formula-1'} is a lift of the tropical $(r, d-r)$-cycle $c(a_\scD)$, where $a_{\scD} \in \MW^{r}(\Sigma_{w=0})$ is the Minkowski weight corresponding to $(-1)^{d+1-r}\ld D_1 \rd \cdots \ld D_r \rd$.

\section{Proof of \pref{th:main4}}\label{sc:proof4}

Let $w_1, w_2 \in \rint(\Delta)$, 
and $a_1 \in \MW^{d-q} \lb \Sigma_{w_1} \rb, a_2 \in \MW^{q} \lb \Sigma_{w_2} \rb$.
We consider again the deformation $c(a_2)'$ of the tropical cycle $c(a_2)$, which we constructed in \pref{sc:main1-2} with the homeomorphism $\phi_{w_2} \colon \partial \nabla_{w_2} \to S$ \eqref{eq:phi-w} and a generic element $m_0 \in M_\bR$.
(The tropical cycles $c(a_1)$ and $c(a_2)'$ intersect transversely.)
We also consider the corresponding deformation of the cycle $C(a_2)$ of \pref{th:main2}.
Namely, we deform the first component of
$C(a_2) \subset \partial \nabla_{w_2}^{d-q} \times \lb \left. N_\bR \middle/ N \right. \rb$ by the deformation $c(a_2)'$ of $c(a_2) \subset \partial \nabla_{w_2}^{d-q}$, leaving the second component unchanged.
We write the resulting cycle as 
$C(a_2)' \subset \partial \nabla_{w_2} \times \lb \left. N_\bR \middle/ N \right. \rb$.
We further consider its small perturbation, the image of $C(a_2)'$ by the map $r_{2, t} \circ r_{1, t}$, and its further image $C(a_2)_t' \subset \mathring{Z}_t$ by the map $h_t$.

The cycles $C(a_1)_t$ and $C(a_2)_t'$ intersect at the points corresponding to the intersection points of the cycle $C(a_1)$ and the deformed cycle $C(a_2)'$.
Furthermore, the cycles $C(a_1)$ and $C(a_2)'$ can only intersect over the intersection points of the tropical cycles $c(a_1)$ and $c(a_2)'$.
Therefore, it suffices to show that the multiplicity of an intersection point of $C(a_1)_t$ and $C(a_2)_t'$ coincides with the tropical multiplicity of the corresponding intersection point of the tropical cycles $c(a_1)$ and $c(a_2)'$.
We will see this in the following.

First, we consider the case of $w_1=w_2=:w \in \rint(\Delta)$.
Let $n \in \partial \nabla_w$ be an intersection point of $c(a_1)$ and $c(a_2)'$, and suppose that it is the intersection point of a simplex $\Delta_\scS$ in $c(a_1)$ and (a deformation of) a simplex $\Delta_{\scS'}$ in $c(a_2)'$ with 
$\scS=\lc \sigma_1 \prec \cdots \prec \sigma_{q+1}=\sigma \rc, 
\scS'=\lc \sigma'_1 \prec \cdots \prec \sigma'_{d+1-q}=\sigma' \rc$.
Since the simplices $\Delta_\scS$ and $\Delta_{\scS'}$ should be in the same facet of $\nabla_w$, we have $\sigma_1=\sigma'_1=:\rho$, and the intersection point $n$ is in the relative interior of the facet of $\nabla_w$ dual to $\rho$.
Let $m \in A_w$ be the point such that $\rho = \bR_{\geq 0} \cdot (m-w)$.
The tropical monomials of $\trop(f)$ that attain the minimum around $n$ are 
$\val(k_w)+\la w, \bullet \ra$
and
$\val(k_m)+\la m, \bullet \ra$.
Therefore, in a small neighborhood of $h_t (\pi_1^{-1} (n))$, 
the defining equation $f_t^w(z)=1$ of the hypersurface $\mathring{Z}_t$ is approximately written as
\begin{align}\label{eq:approximate}
    1+\frac{c_m}{c_w} t^{\lambda_m-\lambda_w} z^{m-w} \sim 0,
\end{align}
ignoring lower order terms.
Let $\lc m_1, \cdots, m_{d}\rc \subset M$ be an integral linear coordinate system on 
the facet of $\nabla_w$ dual to $\rho$.
Then $\lc z^{m_1}, \cdots, z^{m_d}, z^{m-w} \rc$ and 
$\lc z^{m_1}, \cdots, z^{m_d}\rc$ give coordinate systems on $N_{\bC^\ast}$ and the algebraic subtorus
\begin{align}\label{eq:ag-subtorus}
    \lc z \in N_{\bC^\ast} \relmid z^{m-w}=1\rc \subset N_{\bC^\ast}
\end{align}
respectively.
By applying the implicit function theorem to \eqref{eq:approximate}, 
we can write $z^{m-w}$ as a function $y\lb z \rb$ of $z=(z^{m_1}, \cdots, z^{m_d})$, and the map
\begin{align}
    \lb z^{m_1}, \cdots, z^{m_d}\rb \mapsto \lb z^{m_1}, \cdots, z^{m_d}, z^{m-w}=y\lb z \rb \rb
\end{align}
gives an isomorphism from a subset of the algebraic subtorus \eqref{eq:ag-subtorus} to a small neighborhood of the intersection point of $C(a_1)_t$ and $C(a_2)_t'$ (corresponding to $n$) in the hypersurface $\mathring{Z}_t$.
\eqref{eq:ag-subtorus} is further isomorphic to the trivial torus bundle
\begin{align}\label{eq:local-bdl}
\rho^\perp
\times \lb \left. \rho^\perp \middle/ (N \cap \rho^\perp) \right. \rb
\end{align}
via the map $h_t$ \eqref{eq:h_t}.
We identify these sets in the following.

From \pref{th:main2}(2), we can see that locally around the intersection point, the cycles $C(a_1)_t$ and $C(a_2)_t'$ are small perturbations of trivial torus bundles over $\mathring{\Delta}_\scS$ and (the deformation of) $\mathring{\Delta}_{\scS'}$, whose torus fibers define the homology classes
\begin{align}\label{eq:fiber-class1}
    a_1(\sigma[\scS]) \cdot f (\scS) \in \bigwedge^{d-q} (N \cap \rho^\perp) 
    &\cong H_{d-q} \lb \left. \rho^\perp \middle/ (N \cap \rho^\perp) \right., \bZ \rb\\ \label{eq:fiber-class2}
    a_2(\sigma[\scS']) \cdot f (\scS') \in \bigwedge^{q} (N \cap \rho^\perp)
    &\cong H_{q} \lb \left. \rho^\perp \middle/ (N \cap \rho^\perp) \right., \bZ \rb
\end{align}
in \eqref{eq:local-bdl} respectively.

Let $\Omega_\rho$ be the integral volume form on the facet of $\nabla_w$ dual to $\rho$, which defines the orientation under which simplices $\Delta_\scS$ and (the deformation of) $\Delta_{\scS'}$ intersect positively at the point $n$.
Then the tropical multiplicity ($\Omega_x \lb a_\Delta \wedge b_{\Delta'} \rb$ 
in \eqref{eq:tropical-intersection}) of the tropical intersection $\la c(a_1), c(a_2)' \ra$ at the intersection point $n$ is 
\begin{align}\label{eq:trop-multi}
    \Omega_\rho 
    \lb a_1(\sigma[\scS]) \cdot f (\scS) \wedge a_2(\sigma[\scS']) \cdot f (\scS') \rb. 
\end{align}
The integral volume form $\Omega_\rho$ defines an orientation of $\rho^\perp$, which also defines an orientation of \eqref{eq:local-bdl}.
By \eqref{eq:fiber-class1} and \eqref{eq:fiber-class2}, we can see that the multiplicity of the intersection point of $C(a_1)_t$ and $C(a_2)_t'$ corresponding to $n$ in \eqref{eq:local-bdl} is
\begin{align}\label{eq:pre-multi}
 (-1)^{(d-q)^2} \cdot
    \Omega_\rho 
    \lb a_1(\sigma[\scS]) \cdot f (\scS) \wedge a_2(\sigma[\scS']) \cdot f (\scS') \rb. 
\end{align}
The sign $(-1)^{(d-q)^2}=(-1)^{(d-q)}$ in \eqref{eq:pre-multi} corresponds to rearranging the $(d-q)$ torus fiber directions of $C(a_1)_t$ and the $(d-q)$ tropical cycle directions of $C(a_2)_t'$ for comparison with the orientation of \eqref{eq:local-bdl}.
The orientation of the hypersurface $\mathring{Z}_t$ coincides with that of \eqref{eq:ag-subtorus}.
This is $(-1)^{d(d+1)/2}$ times that of \eqref{eq:local-bdl}, since if we write $z^{m_i}=r_i e^{\sqrt{-1}\theta_i}$ using polar coordinates, set $n_i:=\la m_i, n\ra, n_i':=\la m_i, n'\ra$ for elements $(n, n')$ in \eqref{eq:local-bdl}, and let 
$(n_1, \cdots, n_d, n_1', \cdots, n_d')$ be the coordinate system on \eqref{eq:local-bdl}, then one has
\begin{align}
h_t^\ast \lb \bigwedge_{i=1}^d dr_i \wedge d\theta_i\rb
=
\bigwedge_{i=1}^d d\lb t^{n_i} \rb \wedge d(2 \pi n'_i)
=(-1)^{d+d(d-1)/2} \cdot
\lb \bigwedge_{i=1}^d (-t^{n_i} \log t) dn_i\rb \wedge \lb \bigwedge_{i=1}^d 2\pi dn'_i \rb,
\end{align}
and $d+d(d-1)/2=d(d+1)/2$.
(Notice that $-t^{n_i} \log t>0$.)
Therefore, the multiplicity of the intersection point of $C(a_1)_t$ and $C(a_2)_t'$ corresponding to $n$ in the hypersurface $\mathring{Z}_t$ is 
\begin{align}
 (-1)^{d(d+1)/2} \cdot (-1)^{(d-q)} \cdot
    \Omega_\rho 
    \lb a(\sigma[\scS]) \cdot f (\scS) \wedge a(\sigma[\scS']) \cdot f (\scS') \rb. 
\end{align}
Comparing this with \eqref{eq:trop-multi}, we can obtain \eqref{eq:intersection-compare} in \pref{th:main4}.
We conclude \pref{th:main4} for the case of $w_1=w_2$.
We can also prove \pref{th:main4} for the case where $w_1 \neq w_2$ and $\conv (\lc w_1, w_2 \rc) \in \scrT$ in the same way.

Lastly, we consider the case of $\conv (\lc w_1, w_2 \rc) \nin \scrT$.
We have $\la c(a_1), c(a_2) \ra=0$ (\pref{th:main1}(2-c)).
The tropical cycles $c(a_1)$ and $c(a_2)$ are supported in $\partial \nabla_{w_1}$ and $\partial \nabla_{w_2}$ respectively, and we have $\nabla_{w_1} \cap \nabla_{w_2}=\emptyset$.
From \pref{th:main2}(2), it is clear that we have $C(a_1)_t \cap C(a_2)_t =\emptyset$, and
$\la C(a_1)_t, C(a_2)_t \ra=0$.
Thus \pref{th:main4} also holds in this case.

\section*{Acknowledgements}

I am very grateful to Hiroshi Iritani for many insightful comments on an earlier draft of this article. 
In particular, the idea of formulating the statement of \pref{th:main2} in terms of K-groups of toric varieties was suggested by him. 
Furthermore, I also learned \pref{pr:iritani} and its proof from him.
I also thank Bernd Siebert for helpful discussions at the early stage of this project.
I learned about \cite{MR3088918} from Kris Shaw when I was a graduate student and am grateful for the kind explanations I received.
This work was supported by JSPS KAKENHI Grant Number JP25K17262.

\bibliographystyle{amsalpha}
\bibliography{bibs}

@book {MR224083,
    AUTHOR = {Atiyah, M. F.},
     TITLE = {{$K$}-theory},
      NOTE = {Lecture notes by D. W. Anderson},
 PUBLISHER = {W. A. Benjamin, Inc., New York-Amsterdam},
      YEAR = {1967},
     PAGES = {v+166+xlix},
   MRCLASS = {55.30 (57.00)},
  MRNUMBER = {224083},
}

@book {MR4406774,
    AUTHOR = {Friedman, Greg},
     TITLE = {Singular intersection homology},
    SERIES = {New Mathematical Monographs},
    VOLUME = {33},
 PUBLISHER = {Cambridge University Press, Cambridge},
      YEAR = {2020},
     PAGES = {xxiii+798},
      ISBN = {978-1-107-15074-4},
   MRCLASS = {55N33 (57N80 57Pxx)},
  MRNUMBER = {4406774},
MRREVIEWER = {Juan\ Antonio\ P\'erez},
       DOI = {10.1017/9781316584446},
       URL = {https://doi.org/10.1017/9781316584446},
}

@article {MR1234308,
    AUTHOR = {Morelli, Robert},
     TITLE = {The {$K$}-theory of a toric variety},
   JOURNAL = {Adv. Math.},
  FJOURNAL = {Advances in Mathematics},
    VOLUME = {100},
      YEAR = {1993},
    NUMBER = {2},
     PAGES = {154--182},
      ISSN = {0001-8708,1090-2082},
   MRCLASS = {14M25 (19E99 52B20 52B45)},
  MRNUMBER = {1234308},
MRREVIEWER = {G.\ K.\ Sankaran},
       DOI = {10.1006/aima.1993.1032},
       URL = {https://doi.org/10.1006/aima.1993.1032},
}

@incollection {MR3088918,
    AUTHOR = {Shaw, Kristin M.},
     TITLE = {Tropical {$(1,1)$}-homology for floor decomposed surfaces},
 BOOKTITLE = {Algebraic and combinatorial aspects of tropical geometry},
    SERIES = {Contemp. Math.},
    VOLUME = {589},
     PAGES = {329--350},
 PUBLISHER = {Amer. Math. Soc., Providence, RI},
      YEAR = {2013},
      ISBN = {978-0-8218-9146-9},
   MRCLASS = {14T05 (14C30)},
  MRNUMBER = {3088918},
MRREVIEWER = {Joaquim\ Ro\'e},
       DOI = {10.1090/conm/589/11750},
       URL = {https://doi.org/10.1090/conm/589/11750},
}

@article {MR3032930,
    AUTHOR = {Shaw, Kristin M.},
     TITLE = {A tropical intersection product in matroidal fans},
   JOURNAL = {SIAM J. Discrete Math.},
  FJOURNAL = {SIAM Journal on Discrete Mathematics},
    VOLUME = {27},
      YEAR = {2013},
    NUMBER = {1},
     PAGES = {459--491},
      ISSN = {0895-4801,1095-7146},
   MRCLASS = {14T05 (05B35 14C17)},
  MRNUMBER = {3032930},
MRREVIEWER = {Patrick\ Popescu-Pampu},
       DOI = {10.1137/110850141},
       URL = {https://doi.org/10.1137/110850141},
}

@article {MR2887109,
    AUTHOR = {Katz, Eric},
     TITLE = {Tropical intersection theory from toric varieties},
   JOURNAL = {Collect. Math.},
  FJOURNAL = {Collectanea Mathematica},
    VOLUME = {63},
      YEAR = {2012},
    NUMBER = {1},
     PAGES = {29--44},
      ISSN = {0010-0757,2038-4815},
   MRCLASS = {14T05 (14M25)},
  MRNUMBER = {2887109},
MRREVIEWER = {Frank\ Sottile},
       DOI = {10.1007/s13348-010-0014-8},
       URL = {https://doi.org/10.1007/s13348-010-0014-8},
}

@article {MR2591823,
    AUTHOR = {Allermann, Lars and Rau, Johannes},
     TITLE = {First steps in tropical intersection theory},
   JOURNAL = {Math. Z.},
  FJOURNAL = {Mathematische Zeitschrift},
    VOLUME = {264},
      YEAR = {2010},
    NUMBER = {3},
     PAGES = {633--670},
      ISSN = {0025-5874,1432-1823},
   MRCLASS = {14T05 (14C17)},
  MRNUMBER = {2591823},
MRREVIEWER = {Joaquim\ Ro\'e},
       DOI = {10.1007/s00209-009-0483-1},
       URL = {https://doi.org/10.1007/s00209-009-0483-1},
}

@article {MR4782805,
    AUTHOR = {Yamamoto, Yuto},
     TITLE = {Period integrals of hypersurfaces via tropical geometry},
   JOURNAL = {Int. Math. Res. Not. IMRN},
  FJOURNAL = {International Mathematics Research Notices. IMRN},
      YEAR = {2024},
    NUMBER = {15},
     PAGES = {11386--11425},
      ISSN = {1073-7928,1687-0247},
   MRCLASS = {14T20 (14D05 14M25)},
  MRNUMBER = {4782805},
       DOI = {10.1093/imrn/rnae123},
       URL = {https://doi.org/10.1093/imrn/rnae123},
}

@article {MR2019444,
    AUTHOR = {Mavlyutov, Anvar R.},
     TITLE = {On the chiral ring of {C}alabi-{Y}au hypersurfaces in toric
              varieties},
   JOURNAL = {Compositio Math.},
  FJOURNAL = {Compositio Mathematica},
    VOLUME = {138},
      YEAR = {2003},
    NUMBER = {3},
     PAGES = {289--336},
      ISSN = {0010-437X},
   MRCLASS = {14J32 (14M25 32Q25)},
  MRNUMBER = {2019444},
MRREVIEWER = {Lev A. Borisov},
       DOI = {10.1023/A:1027367922964},
       URL = {https://doi.org/10.1023/A:1027367922964},
}

@article {MR1290195,
    AUTHOR = {Batyrev, Victor V. and Cox, David A.},
     TITLE = {On the {H}odge structure of projective hypersurfaces in toric
              varieties},
   JOURNAL = {Duke Math. J.},
  FJOURNAL = {Duke Mathematical Journal},
    VOLUME = {75},
      YEAR = {1994},
    NUMBER = {2},
     PAGES = {293--338},
      ISSN = {0012-7094},
   MRCLASS = {14M25 (14C30 14D07 14F10)},
  MRNUMBER = {1290195},
MRREVIEWER = {I. Dolgachev},
       DOI = {10.1215/S0012-7094-94-07509-1},
       URL = {https://doi.org/10.1215/S0012-7094-94-07509-1},
}

@article{MR2079993,
	Author = {Mikhalkin, Grigory},
	Coden = {TPLGAF},
	Doi = {10.1016/j.top.2003.11.006},
	Fjournal = {Topology. An International Journal of Mathematics},
	Issn = {0040-9383},
	Journal = {Topology},
	Mrclass = {14J70 (14M25 32Q55)},
	Mrreviewer = {G. K. Sankaran},
	Number = {5},
	Pages = {1035--1065},
	Title = {Decomposition into pairs-of-pants for complex algebraic hypersurfaces},
	Url = {http://dx.doi.org/10.1016/j.top.2003.11.006},
	Volume = {43},
	Year = {2004}}

@article {MR2240909,
    AUTHOR = {Abouzaid, Mohammed},
     TITLE = {Homogeneous coordinate rings and mirror symmetry for toric
              varieties},
      NOTE = {[Paging previously given as 1097--1157]},
   JOURNAL = {Geom. Topol.},
  FJOURNAL = {Geometry and Topology},
    VOLUME = {10},
      YEAR = {2006},
     PAGES = {1097--1156},
      ISSN = {1465-3060,1364-0380},
   MRCLASS = {14J32 (14M25 53D40)},
  MRNUMBER = {2240909},
MRREVIEWER = {Michael\ J.\ Usher},
       DOI = {10.2140/gt.2006.10.1097},
       URL = {https://doi.org/10.2140/gt.2006.10.1097},
}

@article{MR3961331,
	Author = {Itenberg, Ilia and Katzarkov, Ludmil and Mikhalkin, Grigory and Zharkov, Ilia},
	Doi = {10.1007/s00208-018-1685-9},
	Fjournal = {Mathematische Annalen},
	Issn = {0025-5831},
	Journal = {Math. Ann.},
	Mrclass = {14T05},
	Mrnumber = {3961331},
	Mrreviewer = {Helge Ruddat},
	Number = {1-2},
	Pages = {963--1006},
	Title = {Tropical homology},
	Url = {https://doi.org/10.1007/s00208-018-1685-9},
	Volume = {374},
	Year = {2019}}

@incollection{MR2275625,
	Author = {Mikhalkin, Grigory},
	Booktitle = {International {C}ongress of {M}athematicians. {V}ol. {II}},
	Mrclass = {14P99 (14N10 14N35 52B20)},
	Mrnumber = {2275625},
	Mrreviewer = {Jean-Yves Welschinger},
	Pages = {827--852},
	Publisher = {Eur. Math. Soc., Z\"urich},
	Title = {Tropical geometry and its applications},
	Year = {2006}}

@article{MR3112512,
	Author = {Iritani, Hiroshi},
	Fjournal = {Universit\'e de Grenoble. Annales de l'Institut Fourier},
	Issn = {0373-0956},
	Journal = {Ann. Inst. Fourier (Grenoble)},
	Mrclass = {14N35 (14D07 32G20 53D37)},
	Mrnumber = {3112512},
	Mrreviewer = {Guosong Zhao},
	Number = {7},
	Pages = {2909--2958},
	Title = {Quantum cohomology and periods},
	Url = {http://aif.cedram.org/item?id=AIF_2011__61_7_2909_0},
	Volume = {61},
	Year = {2011}}

@article{MR2669728,
	Author = {Gross, Mark and Siebert, Bernd},
	Doi = {10.1090/S1056-3911-2010-00555-3},
	Fjournal = {Journal of Algebraic Geometry},
	Issn = {1056-3911},
	Journal = {J. Algebraic Geom.},
	Mrclass = {14J33 (14C30 14F40 14T05)},
	Mrnumber = {2669728},
	Mrreviewer = {Diego Matessi},
	Number = {4},
	Pages = {679--780},
	Title = {Mirror symmetry via logarithmic degeneration data, {II}},
	Url = {https://doi.org/10.1090/S1056-3911-2010-00555-3},
	Volume = {19},
	Year = {2010}}

@article{MR2576286,
	Author = {Iwao, Shinsuke},
	Doi = {10.1093/imrn/rnp129},
	Fjournal = {International Mathematics Research Notices. IMRN},
	Issn = {1073-7928},
	Journal = {Int. Math. Res. Not. IMRN},
	Mrclass = {14T05 (14H50)},
	Mrnumber = {2576286},
	Mrreviewer = {Hannah Markwig},
	Number = {1},
	Pages = {112--148},
	Title = {Integration over tropical plane curves and ultradiscretization},
	Url = {https://doi.org/10.1093/imrn/rnp129},
	Year = {2010}}

@article {MR4194298,
    AUTHOR = {Abouzaid, Mohammed and Ganatra, Sheel and Iritani, Hiroshi and
              Sheridan, Nick},
     TITLE = {The gamma and {S}trominger-{Y}au-{Z}aslow conjectures: a
              tropical approach to periods},
   JOURNAL = {Geom. Topol.},
  FJOURNAL = {Geometry \& Topology},
    VOLUME = {24},
      YEAR = {2020},
    NUMBER = {5},
     PAGES = {2547--2602},
      ISSN = {1465-3060},
   MRCLASS = {14J33 (11G42 14T20 32G20 53D37)},
  MRNUMBER = {4194298},
       DOI = {10.2140/gt.2020.24.2547},
       URL = {https://doi.org/10.2140/gt.2020.24.2547},
}

@article {MR4179831,
    AUTHOR = {Ruddat, Helge and Siebert, Bernd},
     TITLE = {Period integrals from wall structures via tropical cycles,
              canonical coordinates in mirror symmetry and analyticity of
              toric degenerations},
   JOURNAL = {Publ. Math. Inst. Hautes \'{E}tudes Sci.},
  FJOURNAL = {Publications Math\'{e}matiques. Institut de Hautes \'{E}tudes
              Scientifiques},
    VOLUME = {132},
      YEAR = {2020},
     PAGES = {1--82},
      ISSN = {0073-8301},
   MRCLASS = {14J33 (14T20)},
  MRNUMBER = {4179831},
       DOI = {10.1007/s10240-020-00116-y},
       URL = {https://doi.org/10.1007/s10240-020-00116-y},
}

@article {MR4347312,
    AUTHOR = {Ruddat, Helge},
     TITLE = {A homology theory for tropical cycles on integral affine
              manifolds and a perfect pairing},
   JOURNAL = {Geom. Topol.},
  FJOURNAL = {Geometry \& Topology},
    VOLUME = {25},
      YEAR = {2021},
    NUMBER = {6},
     PAGES = {3079--3132},
      ISSN = {1465-3060},
   MRCLASS = {14J32 (05E45 14D06 14T20 32S60 55U10)},
  MRNUMBER = {4347312},
       DOI = {10.2140/gt.2021.25.3079},
       URL = {https://doi.org/10.2140/gt.2021.25.3079},
}

@article{MR2553377,
	Author = {Iritani, Hiroshi},
	Doi = {10.1016/j.aim.2009.05.016},
	Fjournal = {Advances in Mathematics},
	Issn = {0001-8708},
	Journal = {Adv. Math.},
	Mrclass = {53D37 (14J33 14N35 53D45)},
	Mrnumber = {2553377},
	Mrreviewer = {Hsian-Hua Tseng},
	Number = {3},
	Pages = {1016--1079},
	Title = {An integral structure in quantum cohomology and mirror symmetry for toric orbifolds},
	Url = {https://doi.org/10.1016/j.aim.2009.05.016},
	Volume = {222},
	Year = {2009}}

@incollection {MR2681794,
    AUTHOR = {Ruddat, Helge},
     TITLE = {Log {H}odge groups on a toric {C}alabi-{Y}au degeneration},
 BOOKTITLE = {Mirror symmetry and tropical geometry},
    SERIES = {Contemp. Math.},
    VOLUME = {527},
     PAGES = {113--164},
 PUBLISHER = {Amer. Math. Soc., Providence, RI},
      YEAR = {2010},
   MRCLASS = {14D06 (14J33 14M25 32S35)},
  MRNUMBER = {2681794},
MRREVIEWER = {Siu-Cheong Lau},
       DOI = {10.1090/conm/527/10402},
       URL = {https://doi.org/10.1090/conm/527/10402},
}

@incollection{MR3330789,
	Author = {Mikhalkin, Grigory and Zharkov, Ilia},
	Booktitle = {Homological mirror symmetry and tropical geometry},
	Doi = {10.1007/978-3-319-06514-4_7},
	Mrclass = {14T05 (52B40 55N35)},
	Mrnumber = {3330789},
	Mrreviewer = {Susan J. Colley},
	Pages = {309--349},
	Publisher = {Springer, Cham},
	Series = {Lect. Notes Unione Mat. Ital.},
	Title = {Tropical eigenwave and intermediate {J}acobians},
	Url = {https://doi.org/10.1007/978-3-319-06514-4_7},
	Volume = {15},
	Year = {2014}}

@article{MR3228462,
	Author = {Casta\~{n}o-Bernard, Ricardo and Matessi, Diego},
	Doi = {10.2140/gt.2014.18.1769},
	Fjournal = {Geometry \& Topology},
	Issn = {1465-3060},
	Journal = {Geom. Topol.},
	Mrclass = {14J33 (14T05 53D37)},
	Mrnumber = {3228462},
	Mrreviewer = {Siu-Cheong Lau},
	Number = {3},
	Pages = {1769--1863},
	Title = {Conifold transitions via affine geometry and mirror symmetry},
	Url = {https://doi.org/10.2140/gt.2014.18.1769},
	Volume = {18},
	Year = {2014}}

@incollection{MR2024634,
	Author = {Symington, Margaret},
	Booktitle = {Topology and geometry of manifolds ({A}thens, {GA}, 2001)},
	Doi = {10.1090/pspum/071/2024634},
	Mrclass = {53D35 (53D20 55R55 57R17)},
	Mrnumber = {2024634},
	Mrreviewer = {Vicente Mu\~{n}oz},
	Pages = {153--208},
	Publisher = {Amer. Math. Soc., Providence, RI},
	Series = {Proc. Sympos. Pure Math.},
	Title = {Four dimensions from two in symplectic topology},
	Url = {https://doi.org/10.1090/pspum/071/2024634},
	Volume = {71},
	Year = {2003}}

@article{MR3894860,
	Author = {Jell, Philipp and Rau, Johannes and Shaw, Kristin},
	Fjournal = {\'{E}pijournal de G\'{e}om\'{e}trie Alg\'{e}brique. EPIGA},
	Journal = {\'{E}pijournal G\'{e}om. Alg\'{e}brique},
	Mrclass = {14T05 (14C22 14C25 52B40 55N35)},
	Mrnumber = {3894860},
	Mrreviewer = {Gary P. Kennedy},
	Pages = {Art. 11, 27},
	Title = {Lefschetz {$(1,1)$}-theorem in tropical geometry},
	Volume = {2},
	Year = {2018}}

@book{MR2810322,
	Author = {Cox, David A. and Little, John B. and Schenck, Henry K.},
	Doi = {10.1090/gsm/124},
	Isbn = {978-0-8218-4819-7},
	Mrclass = {14M25 (05A15 05E45 52B12)},
	Mrnumber = {2810322},
	Mrreviewer = {Ivan Arzhantsev},
	Pages = {xxiv+841},
	Publisher = {American Mathematical Society, Providence, RI},
	Series = {Graduate Studies in Mathematics},
	Title = {Toric varieties},
	Url = {https://doi.org/10.1090/gsm/124},
	Volume = {124},
	Year = {2011}}

@article {MR1415592,
    AUTHOR = {Fulton, William and Sturmfels, Bernd},
     TITLE = {Intersection theory on toric varieties},
   JOURNAL = {Topology},
  FJOURNAL = {Topology. An International Journal of Mathematics},
    VOLUME = {36},
      YEAR = {1997},
    NUMBER = {2},
     PAGES = {335--353},
      ISSN = {0040-9383},
   MRCLASS = {14M25 (14C17 52B20)},
  MRNUMBER = {1415592},
MRREVIEWER = {Michel\ Brion},
       DOI = {10.1016/0040-9383(96)00016-X},
       URL = {https://doi.org/10.1016/0040-9383(96)00016-X},
}

@incollection {MR4294796,
    AUTHOR = {Ruddat, Helge and Zharkov, Ilia},
     TITLE = {Compactifying torus fibrations over integral affine manifolds
              with singularities},
 BOOKTITLE = {2019--20 {MATRIX} annals},
    SERIES = {MATRIX Book Ser.},
    VOLUME = {4},
     PAGES = {609--622},
 PUBLISHER = {Springer, Cham},
      YEAR = {[2021] \copyright 2021},
      ISBN = {978-3-030-62496-5; 978-3-030-62497-2},
   MRCLASS = {14D06 (57S12)},
  MRNUMBER = {4294796},
       DOI = {10.1007/978-3-030-62497-2\_37},
       URL = {https://doi.org/10.1007/978-3-030-62497-2_37},
}

@unpublished{RZ21,
	Author = {Ruddat, Helge and Zharkov, Ilia},
	Note = {in preparation},
	Title = {Topological {S}trominger-{Y}au-{Z}aslow fibrations},
	Year = {}}

@article {MR4179650,
    AUTHOR = {Mak, Cheuk Yu and Ruddat, Helge},
     TITLE = {Tropically constructed {L}agrangians in mirror quintic
              threefolds},
   JOURNAL = {Forum Math. Sigma},
  FJOURNAL = {Forum of Mathematics. Sigma},
    VOLUME = {8},
      YEAR = {2020},
     PAGES = {Paper No. e58, 55},
      ISSN = {2050-5094},
   MRCLASS = {53D12 (14J33 53D37 57R17)},
  MRNUMBER = {4179650},
       DOI = {10.1017/fms.2020.54},
       URL = {https://doi.org/10.1017/fms.2020.54},
}

@article {MR4125753,
    AUTHOR = {Hicks, Jeffrey},
     TITLE = {Tropical {L}agrangian hypersurfaces are unobstructed},
   JOURNAL = {J. Topol.},
  FJOURNAL = {Journal of Topology},
    VOLUME = {13},
      YEAR = {2020},
    NUMBER = {4},
     PAGES = {1409--1454},
      ISSN = {1753-8416,1753-8424},
   MRCLASS = {53D37 (14J33 14T20 53D12)},
  MRNUMBER = {4125753},
MRREVIEWER = {Roman\ Golovko},
       DOI = {10.1112/topo.12165},
       URL = {https://doi.org/10.1112/topo.12165},
}

@article {MR4284602,
    AUTHOR = {Matessi, Diego},
     TITLE = {Lagrangian submanifolds from tropical hypersurfaces},
   JOURNAL = {Internat. J. Math.},
  FJOURNAL = {International Journal of Mathematics},
    VOLUME = {32},
      YEAR = {2021},
    NUMBER = {7},
     PAGES = {Paper No. 2150046, 63},
      ISSN = {0129-167X,1793-6519},
   MRCLASS = {53D12 (14T20 53D37)},
  MRNUMBER = {4284602},
MRREVIEWER = {Bhupendra\ Nath\ Tiwari},
       DOI = {10.1142/S0129167X21500464},
       URL = {https://doi.org/10.1142/S0129167X21500464},
}

@article {MR3993277,
    AUTHOR = {Mikhalkin, Grigory},
     TITLE = {Examples of tropical-to-{L}agrangian correspondence},
   JOURNAL = {Eur. J. Math.},
  FJOURNAL = {European Journal of Mathematics},
    VOLUME = {5},
      YEAR = {2019},
    NUMBER = {3},
     PAGES = {1033--1066},
      ISSN = {2199-675X,2199-6768},
   MRCLASS = {53D20 (14T05 53D12)},
  MRNUMBER = {3993277},
MRREVIEWER = {D.\ A.\ Stepanov},
       DOI = {10.1007/s40879-019-00319-6},
       URL = {https://doi.org/10.1007/s40879-019-00319-6},
}

@article {MR4294119,
    AUTHOR = {Matessi, Diego},
     TITLE = {Lagrangian pairs of pants},
   JOURNAL = {Int. Math. Res. Not. IMRN},
  FJOURNAL = {International Mathematics Research Notices. IMRN},
      YEAR = {2021},
    NUMBER = {15},
     PAGES = {11306--11356},
      ISSN = {1073-7928,1687-0247},
   MRCLASS = {53D12 (14M25 14T20)},
  MRNUMBER = {4294119},
MRREVIEWER = {Ilia\ V.\ Itenberg},
       DOI = {10.1093/imrn/rnz126},
       URL = {https://doi.org/10.1093/imrn/rnz126},
}

@unpublished{AMH25,
	Author = {Asakura, Masanori and Matsubara-Heo, Saiei-Jaeyeong},
	Note = {Preprint},
	Title = {Periods of limiting mixed {H}odge structures of projective
hypersurfaces},
	Year = {2025}}

@article {MR4484542,
    AUTHOR = {Yamamoto, Yuto},
     TITLE = {Periods of tropical {C}alabi-{Y}au hypersurfaces},
   JOURNAL = {J. Algebraic Geom.},
  FJOURNAL = {Journal of Algebraic Geometry},
    VOLUME = {31},
      YEAR = {2022},
    NUMBER = {2},
     PAGES = {303--343},
      ISSN = {1056-3911,1534-7486},
   MRCLASS = {14J33 (14D07 14T20)},
  MRNUMBER = {4484542},
MRREVIEWER = {Benjamin\ Gammage},
}

@article {MR4950977,
    AUTHOR = {Han, Zengrui},
     TITLE = {Central charges in local mirror symmetry via hypergeometric
              duality},
   JOURNAL = {Adv. Math.},
  FJOURNAL = {Advances in Mathematics},
    VOLUME = {480},
      YEAR = {2025},
     PAGES = {Paper No. 110502, 37},
      ISSN = {0001-8708,1090-2082},
   MRCLASS = {14J33 (14J32 14M25 33C90)},
  MRNUMBER = {4950977},
       DOI = {10.1016/j.aim.2025.110502},
       URL = {https://doi.org/10.1016/j.aim.2025.110502},
}

@article {MR4922780,
    AUTHOR = {Wang, Junxiao},
     TITLE = {Mirror symmetric gamma conjecture for tropical curves in local
              mirror symmetry},
   JOURNAL = {J. Symplectic Geom.},
  FJOURNAL = {The Journal of Symplectic Geometry},
    VOLUME = {23},
      YEAR = {2025},
    NUMBER = {2},
     PAGES = {309--383},
      ISSN = {1527-5256,1540-2347},
   MRCLASS = {14J33 (14T20 53D37 53D45)},
  MRNUMBER = {4922780},
       DOI = {10.4310/jsg.250624011941},
       URL = {https://doi.org/10.4310/jsg.250624011941},
}

@unpublished{You25,
	Author = {You, Fenglong},
	Note = {arXiv:2508.06750},
	Title = {Relative mirror symmetry, theta functions and the {G}amma conjecture},
	Year = {2025}}

@unpublished{BL22,
	Author = {Berglund, Per and Lathwood, Michael},
	Note = {arXiv:2212.11906},
	Title = {Tropical {P}eriods for {C}alabi-{Y}au {H}ypersurfaces in non--{F}ano {T}oric {V}arieties},
	Year = {2022}}

@unpublished{Iri23,
	Author = {Iritani, Hiroshi},
	Note = {arXiv:2307.15940},
	Title = {Mirror symmetric {G}amma conjecture for {F}ano and {C}alabi-{Y}au manifolds},
	Year = {2023}}

@unpublished{CLL25,
	Author = {Chiu, Shih-Kai and Li, Yang and Lin, Yu-Shen},
	Note = {arXiv:2509.04843},
	Title = {From tropical curves to special {L}agrangians},
	Year = {2025}}

@article {MR4904614,
    AUTHOR = {Li, Yang},
     TITLE = {Special {L}agrangian pair of pants},
   JOURNAL = {Comm. Pure Appl. Math.},
  FJOURNAL = {Communications on Pure and Applied Mathematics},
    VOLUME = {78},
      YEAR = {2025},
    NUMBER = {7},
     PAGES = {1320--1356},
      ISSN = {0010-3640,1097-0312},
   MRCLASS = {53D12},
  MRNUMBER = {4904614},
       DOI = {10.1002/cpa.22248},
       URL = {https://doi.org/10.1002/cpa.22248},
}

@unpublished{FWZ23,
	Author = {Fang, Bohan and Wang, Junxiao and Zhou, Yan},
	Note = {arXiv:2309.02154},
	Title = {Mirror symmetric {G}amma conjecture for del {P}ezzo surfaces},
	Year = {2023}}

@unpublished{AFW25,
	Author = {Aleshkin, Konstantin and Fang, Bohan and Wang, Junxiao},
	Note = {arXiv:2501.14222},
	Title = {Mirror symmetric {G}amma conjecture for toric {GIT} quotients via {F}ourier transform},
	Year = {2025}}

@unpublished{Yam21,
	Author = {Yamamoto, Yuto},
	Note = {arXiv:2105.10141},
	Title = {Tropical contractions to integral affine manifolds with singularities},
	Year = {2021}}

@incollection {MR2282969,
    AUTHOR = {Hosono, Shinobu},
     TITLE = {Central charges, symplectic forms, and hypergeometric series
              in local mirror symmetry},
 BOOKTITLE = {Mirror symmetry. {V}},
    SERIES = {AMS/IP Stud. Adv. Math.},
    VOLUME = {38},
     PAGES = {405--439},
 PUBLISHER = {Amer. Math. Soc., Providence, RI},
      YEAR = {2006},
   MRCLASS = {14J32 (14D05 14D07 32Q25 33C20)},
  MRNUMBER = {2282969},
MRREVIEWER = {Michele Rossi},
}

@article {MR0260733,
    AUTHOR = {Griffiths, Phillip A.},
     TITLE = {On the periods of certain rational integrals. {I}, {II}},
   JOURNAL = {Ann. of Math. (2) 90 (1969), 460-495; ibid. (2)},
  FJOURNAL = {Annals of Mathematics. Second Series},
    VOLUME = {90},
      YEAR = {1969},
     PAGES = {496--541},
      ISSN = {0003-486X},
   MRCLASS = {14.01},
  MRNUMBER = {0260733},
MRREVIEWER = {F. Gherardelli},
       DOI = {10.2307/1970746},
       URL = {https://doi.org/10.2307/1970746},
}

@article {MR1733735,
    AUTHOR = {Mavlyutov, Anvar R.},
     TITLE = {Semiample hypersurfaces in toric varieties},
   JOURNAL = {Duke Math. J.},
  FJOURNAL = {Duke Mathematical Journal},
    VOLUME = {101},
      YEAR = {2000},
    NUMBER = {1},
     PAGES = {85--116},
      ISSN = {0012-7094},
   MRCLASS = {14M25 (14F25)},
  MRNUMBER = {1733735},
MRREVIEWER = {Dag E. Sommervoll},
       DOI = {10.1215/S0012-7094-00-10114-7},
       URL = {https://doi.org/10.1215/S0012-7094-00-10114-7},
}

@article {MR3228454,
    AUTHOR = {Ruddat, Helge and Sibilla, Nicol\`o and Treumann, David and
              Zaslow, Eric},
     TITLE = {Skeleta of affine hypersurfaces},
   JOURNAL = {Geom. Topol.},
  FJOURNAL = {Geometry \& Topology},
    VOLUME = {18},
      YEAR = {2014},
    NUMBER = {3},
     PAGES = {1343--1395},
      ISSN = {1465-3060},
   MRCLASS = {14J70 (14M25 14R99)},
  MRNUMBER = {3228454},
MRREVIEWER = {Helena Ferreira Soares},
       DOI = {10.2140/gt.2014.18.1343},
       URL = {https://doi.org/10.2140/gt.2014.18.1343},
}

@article {MR2529936,
    AUTHOR = {Abouzaid, Mohammed},
     TITLE = {Morse homology, tropical geometry, and homological mirror
              symmetry for toric varieties},
   JOURNAL = {Selecta Math. (N.S.)},
  FJOURNAL = {Selecta Mathematica. New Series},
    VOLUME = {15},
      YEAR = {2009},
    NUMBER = {2},
     PAGES = {189--270},
      ISSN = {1022-1824},
   MRCLASS = {53D37 (14J33 14M25 14T05 53D40)},
  MRNUMBER = {2529936},
MRREVIEWER = {Kwokwai Chan},
       DOI = {10.1007/s00029-009-0492-2},
       URL = {https://doi.org/10.1007/s00029-009-0492-2},
}

@article {MR2871160,
    AUTHOR = {Fang, Bohan and Liu, Chiu-Chu Melissa and Treumann, David and
              Zaslow, Eric},
     TITLE = {T-duality and homological mirror symmetry for toric varieties},
   JOURNAL = {Adv. Math.},
  FJOURNAL = {Advances in Mathematics},
    VOLUME = {229},
      YEAR = {2012},
    NUMBER = {3},
     PAGES = {1875--1911},
      ISSN = {0001-8708},
   MRCLASS = {14J33 (14M25)},
  MRNUMBER = {2871160},
MRREVIEWER = {Siu-Cheong Lau},
       DOI = {10.1016/j.aim.2011.10.022},
       URL = {https://doi.org/10.1016/j.aim.2011.10.022},
}

@article {MR3948684,
    AUTHOR = {Hanlon, Andrew},
     TITLE = {Monodromy of monomially admissible {F}ukaya-{S}eidel
              categories mirror to toric varieties},
   JOURNAL = {Adv. Math.},
  FJOURNAL = {Advances in Mathematics},
    VOLUME = {350},
      YEAR = {2019},
     PAGES = {662--746},
      ISSN = {0001-8708},
   MRCLASS = {14F08 (14J33 14M25 53D37)},
  MRNUMBER = {3948684},
MRREVIEWER = {Amin Gholampour},
       DOI = {10.1016/j.aim.2019.04.056},
       URL = {https://doi.org/10.1016/j.aim.2019.04.056},
}
\end{document}